\documentclass[twoside]{article}  

\usepackage{microtype}
\usepackage[margin=2cm,bottom=2.5cm]{geometry} 
\usepackage{amsmath,amsthm,amssymb,latexsym}
\usepackage[v2,tips]{xy}

\newcommand{\vnten}{\overline\otimes}
\newcommand{\aone}{\Box}
\newcommand{\atwo}{\Diamond}
\newcommand{\proten}{{\widehat{\otimes}}}
\newcommand{\hten}{\otimes^h}
\newcommand{\ehten}{\otimes^{eh}}
\newcommand{\nucten}{\otimes^{nuc}}
\newcommand{\lin}{\operatorname{lin}}
\newcommand{\mf}[1]{\mathfrak{#1}}
\newcommand{\mc}[1]{\mathcal{#1}}
\newcommand{\ip}[2]{{\langle {#1} , {#2} \rangle}}
\newcommand{\ipp}[2]{{\langle\langle {#1} , {#2} \rangle\rangle}}
\newcommand{\id}{\operatorname{id}}
\newcommand{\wap}{\operatorname{WAP}}

\newcommand{\module}{\operatorname{mod}}
\newcommand{\G}{{\mathbb G}}
\newcommand{\tfac}{\Gamma^2_c}

\theoremstyle{plain}
\newtheorem{proposition}{Proposition}[section]
\newtheorem{theorem}[proposition]{Theorem}
\newtheorem{corollary}[proposition]{Corollary}
\newtheorem{lemma}[proposition]{Lemma}
\theoremstyle{definition}

\newtheorem{example}[proposition]{Example}

\begin{document}

\large
\title{Multipliers, Self-Induced and Dual Banach Algebras}
\author{Matthew Daws}
\maketitle

\begin{abstract}
In the first part of the paper, we present a short survey of the theory of
multipliers, or double centralisers, of Banach algebras and completely
contractive Banach algebras.  Our approach is very algebraic: this is a
deliberate attempt to separate essentially algebraic arguments from topological
arguments.  We concentrate upon the problem of how to extend module actions,
and homomorphisms, from algebras to multiplier algebras.  We then consider the
special cases when we have a bounded approximate identity, and when our algebra
is self-induced.  In the second part of the paper, we mainly concentrate upon
dual Banach algebras.  We provide a simple criterion for when a multiplier
algebra is a dual Banach algebra.  This is applied to show that the multiplier
algebra of the convolution algebra of a locally compact quantum group is always
a dual Banach algebra.  We also study this problem within the framework of
abstract Pontryagin duality, and show that we construct the same
weak$^*$-topology.  We explore the notion of a Hopf convolution algebra, and
show that in many cases, the use of the extended Haagerup tensor product can be
replaced by a multiplier algebra.


\end{abstract}

\tableofcontents

\section{Introduction}

Multipliers are a useful way of embedding a non-unital algebra into a unital
algebra: a problem which occurs often in algebraic analysis.  The theory has reached
maturity when applied to C$^*$-algebras (see, for example, \cite[Chapter~2]{WO})
where it is best studied in the context of Hilbert C$^*$-modules, \cite{lance}.
Indeed, one can also study ``unbounded operators'' for C$^*$-algebras, \cite{woro},
which form a vital tool in the study of quantum groups.
For Banach algebras with a bounded approximate identity, much of the theory
carries over (see \cite[Section~1.d]{johnmem} or \cite[Theorem~2.9.49]{dales})
although we remark that there seems to be no parallel to the unbounded theory.

This paper starts with a survey of multipliers; we start with some generality,
working with multipliers of modules, and not just algebras.  This material is
surely well-known to experts, but we are not aware of any particularly definitive
source.  For example, in \cite{ng}, Ng uses similar ideas (but for C$^*$-algebras,
working in the category of operator modules) motived by the study of cohomology
theories for Hopf operator algebras (that is, loosely speaking, quantum groups).
However, most of the proofs are left in an unpublished manuscript.  The particular
aspects of the theory which we develop is somewhat motivated by Ng's presentation.

We quickly turn to discussing Banach algebras, but we shall not (as is usually the
case) require a bounded approximate identity at this stage.  Instead, we proceed
in a very algebraic manner: the key point is that under some mild assumptions
on the ability to extend module actions, most of the theory can be developed
without worrying about \emph{how} such an extension can be found.

In Section~\ref{sec::bai}, we do consider the classical case of when we have
a bounded approximate identity.  There are two main ideas here: the use of
Cohen's factorisation theorem, and the use of Arens products.  Again, this
section is mostly a survey, although we, as usual, proceed with more generality
than usual.

In Section~\ref{dba_section}, we turn our attention to dual Banach algebras:
Banach algebras which are dual spaces, such that the multiplication is
separately weak$^*$-continuous.  Here we apply the idea of using the Arens
products to give a very short, algebraic proof of \cite[Theorem~5.6]{is}.

In Section~\ref{selfind}, we look at \emph{self-induced} algebras, see \cite{gronbaek},
which we argue form a larger, natural class of algebras where multipliers are
well-behaved.  We show how various extension problems for multipliers can
be solved in the self-induced case, but we do not give a complete survey.

Many algebras which arise in abstract harmonic analysis, for example the
Fourier algebra $A(G)$, are best studied in the category of Operator Spaces.
Rather than develop the theory twice, once for bounded maps, and then again
for completely bounded maps, we try to take a ``categorical'' approach
throughout, so that we can develop both theories in parallel.  For example,
we introduce the Arens products in Section~\ref{sec::bai} in an unusual way,
making more explicit links with the projective tensor product.  This seems
unnecessary in the Banach algebra case, but it does mean that our proofs work
\emph{mutatis mutandis} in the Operator space setting.  In Section~\ref{sec::ccba}
we quickly check that everything we have so far developed does work for
completely contractive Banach algebras.

We then look at the Fourier algebra in more depth: in particular, we show
that $A(G)$ is always self-induced, as a completely contractive Banach algebra.
Of course, $A(G)$ has a bounded approximate identity only when $G$ is amenable.
This provides some motivation for looking at the larger class of self-induced
algebras.

There has been considerable interest in multipliers as applied to abstract
(quantum) harmonic analysis, see \cite{nh, jnr, is, spronk, neufang} for example.
Part of our motivation for writing this paper is to argue that a slightly more
systematic approach to multipliers allows one to separate out the abstract Banach
algebra arguments from specific arguments, say from abstract harmonic analysis.
We take up this study seriously in Section~\ref{mult_dual_sec}, where we provide
a simple criterion (and construction) for showing that the multiplier algebra
of a Banach algebra is actually a dual Banach algebra (something one would not
actually expect to be true by analogy with the C$^*$-algebra setting).
We quickly check that our abstract result agrees with more concrete constructions
for $L^1(G)$ and $A(G)$.  We show that dual Banach algebras always have
multiplier algebras which are dual, and we make links with the case when we
have a bounded approximate identity: here we can make our construction more
concrete.

In Section~\ref{lcqg} we apply these ideas to the study of the convolution algebra of
a locally compact quantum group $\G$.  Indeed, we show that $M(L^1(\G))$ and
$M_{cb}(L^1(\G))$ are always dual Banach algebras: we need remarkably little
theory to show this!  We check that, again, our work carries over to the
Operator Space setting with little effort.  Multipliers and dual space
structures were consider by Kraus and Ruan in \cite{KR} for Kac algebras,
using the duality theory of Kac algebras.  We generalise (some of) their work
to locally compact quantum groups, and show that the resulting dual Banach
algebra structure on $M_{cb}(L^1(\G))$ agrees with that given by
our abstract construction.

In the final section, we look at the notion of a Hopf convolution algebra,
\cite{ERhc}.  The operator algebra approach to quantum groups usually starts with
a C$^*$-algebra or a von Neumann algebra which carries a co-product, hence turning the
dual or predual into a Banach algebra, which we term a convolution algebra.
However, in abstract harmonic analysis, one usually privileges the convolution
algebra as being the object of study.  Hence, can we give the convolution algebra
a coproduct?  By analogy with the C$^*$-algebra setup, we would expect the
coproduct to map into a multiplier algebra.  We explore this possibility,
and show that in many cases (including the Fourier algebra, and discrete and compact
quantum groups) this is possible.
We also develop an abstract theory of corepresentations in this setting.

A final word on notation.  We generally follow \cite{dales} for matters
related to Banach algebras, but we write $E^*$ for the dual space of a Banach
space $E$.  We write $\kappa_E$ for the canonical map $E\rightarrow E^{**}$ from
a Banach (or Operator) space to its bidual.

\section{Multipliers}\label{mults}

In this section, we shall present some general background theory about multipliers.
We shall develop the theory in a rather general context, namely for modules and
not just algebras.  This material, as applied to algebras, is well-known, but to
our knowledge, has not been systematically presented in this general context.
As such, we make no particular claim to originality, and we shall try to give
references, and sometimes just sketch proofs, where appropriate.  We take a little
care to present the material in a manner which clearly holds both for Banach spaces,
and for Operator spaces.

Multipliers seem to go back to work of Hochschild, Dauns \cite{dauns} and Johnson \cite{johnson}.
See \cite[Section~1.2]{palmer} for more historical remarks.
For C$^*$-algebras, all the standard texts cover
multipliers; both \cite{WO} and \cite[Section~II.7.3]{blackadar} are very readable.
An approach using double centralisers (see below) is taken by \cite{murphy}
while \cite{tak1} follows a bidual approach (compare with Theorem~\ref{arens} below),
and \cite{pedersen} explains the links between these two approaches.

Let $\mc A$ be a (complex) algebra, and let $E$ be an $\mc A$-bimodule.  We say that $E$
is \emph{faithful} if, for $x\in E$, we have that $a\cdot x\cdot b=0$ for all
$a,b\in\mc A$, then $x=0$.  We remark that, for C$^*$-algebras, the term
\emph{non-degenerate} is commonly used for this property.
Notice that $E_0=\{ x\in E : a\cdot x\cdot b=0 \ (a,b\in\mc A) \}$
is a submodule of $E$, and that $E/E_0$ is faithful.  Unless
otherwise stated, we shall always assume that modules are faithful, and, furthermore,
that $\mc A$ is faithful as module over itself.

A \emph{multiplier} of $E$ is a pair $(L,R)$ of maps $\mc A\rightarrow E$ such
that $a\cdot L(b) = R(a)\cdot b$ for $a,b\in\mc A$.  This notion (at least when $E=\mc A$)
is often called a \emph{centraliser} in the literature (see \cite{johnson} or \cite{jnr}).
We write $M(E)$ or $M_{\mc A}(E)$ for the collection of multipliers of $E$.
Each $x\in E$ induces a multiplier $(L_x,R_x)$ given by
\[ L_x(a) = x\cdot a, \quad R_x(a) = a\cdot x \qquad (a\in\mc A). \]
As $E$ is faithful, this gives an inclusion $E\rightarrow M(E)$.

We shall, occasionally, use the following notions.  Write $M_l(E)$ for the
collection of \emph{left multipliers} of $E$, that is, maps $L:\mc A \rightarrow E$
with $L(ab) = L(a)\cdot b$ for $a,b\in\mc A$.  Similarly, we define $M_r(E)$,
the collection of \emph{right multipliers}, those maps $R:\mc A\rightarrow E$
with $R(ab)=a\cdot R(b)$ for $a,b\in\mc A$.

\begin{lemma}
Let $(L,R)$ be a multiplier of $E$.  Then $L$ and $R$ are linear, $L$ is
a right module homomorphism (that is, a left multiplier), and $R$ is a
left module homomorphism (that is, a right multiplier).
When $\mc A$ is unital, $M(E) \cong E$.
\end{lemma}
\begin{proof}
When $\mc A=E$, this is well-known, compare \cite[Theorem~1.2.4]{palmer}.
For $a,b,c\in\mc A$ and $t\in\mathbb C$,
\[ a\cdot L(b+tc) = R(a)\cdot(b+tc) = R(a)\cdot b + t R(a)\cdot c
= a\cdot \big(L(b) + t  L(c)\big). \]
As $E$ is faithful, it follows that $L(b+tc) = L(b)+tL(c)$, so that $L$
is linear.  Furthermore,
\[ a\cdot L(bc) = R(a)\cdot bc = (R(a)\cdot b)\cdot c = (a\cdot L(b)) \cdot c
= a\cdot (L(b)\cdot c), \]
so that $L(bc) = L(b)\cdot c$, so that $L$ is a right module homomoprhisms.
The claims about $R$ follow analogously.

When $\mc A$ is unital, we see that
\[ L(a) = L(1a) = L(1)\cdot a, \quad R(a) = R(a1) = a\cdot R(1)
\qquad (a\in\mc A). \]
Furthermore, for $a,b\in\mc A$, we have
$a\cdot L(1)\cdot b = a\cdot L(b) = R(a)\cdot b = a \cdot R(1)\cdot b$, so
that $L(1)=R(1)$, and so $(L,R)$ is induced by $L(1)\in E$.
\end{proof}

When $E=\mc A$, we can turn $M(\mc A)$ into an algebra with the product
$(L,R)(L',R') = (LL',R'R)$.  In general, $M(E)$ is an $\mc A$-bimodule for the actions
\[ \begin{matrix} (a\cdot L)(b) = a\cdot L(b), \ \ (a\cdot R)(b) = R(ba), \\
(L\cdot a)(b) = L(ab), \ \ (R\cdot a)(b) = R(b)\cdot a \end{matrix}
\qquad (a,b\in\mc A, (L,R)\in M(E)). \]
These are well-defined, as, for example,
\[ a\cdot (a_0\cdot L)(b) = aa_0\cdot L(b) = R(aa_0)\cdot b
= (a_0\cdot R)(a) \cdot b \qquad (a,b,a_0\in\mc A, (L,R)\in M(E)). \]

If $E$ is a submodule of $F$, then the \emph{idealiser} of $E$ in $F$ is
$E_F = \{ x\in F : a\cdot x,x\cdot a\in E \ (a\in\mc A) \}$.  Then we have an
obvious map $E_F\rightarrow M(E); x\mapsto (L_x,R_x)$ where, as before,
$L_x(a) = x\cdot a, R_x(a) = a\cdot x$ for $a\in\mc A$.  If $F$ is faithful, this
map is injective; it is always an $\mc A$-bimodule homomorphism.

\begin{lemma}\label{idealiser}
$M(E)$ is a faithful $\mc A$-bimodule, and the idealiser of $E$ in
$M(E)$ is all of $M(E)$.
\end{lemma}
\begin{proof}
Let us consider the module actions on $M(E)$ in more detail.  Given $(L,R)\in M(E)$
and $a\in\mc A$, we have
\[ \begin{matrix} (a\cdot L)(b) = a\cdot L(b) = R(a)\cdot b = L_{R(a)}(b),\\
(a\cdot R)(b) = R(ba) = b \cdot R(a) = R_{R(a)}(b) \end{matrix}
\qquad (b\in\mc A), \]
so that $a\cdot (L,R) = (L_{R(a)}, R_{R(a)}) \in E$.  Similarly,
$(L,R)\cdot a = (L_{L(a)}, R_{L(a)})\in E$.
Thus the idealiser of $E$ in $M(E)$ is all of $M(E)$.

Furthermore, $a\cdot (L,R)\cdot b = (L_{R(a)\cdot b}, R_{R(a)\cdot b})$,
so if this equals $0$ for all $a,b\in\mc A$, then using the ``right'' map,
we see that $c\cdot R(a)\cdot b=0$ for all $a,b,c\in\mc A$.  As $E$ is
faithful, $R=0$, and hence $a\cdot L(b)=0$ for all $a,b\in\mc A$, so
that $L=0$.  Thus $M(E)$ is faithful.
\end{proof}

So, given any faithful module $F$ such that $E$ is a submodule of $F$ and $F$
idealises $E$, we have an injection $F\rightarrow M(E)$.  Given the previous
lemma, we can hence regard $M(E)$ as the ``largest'' faithful module containing $E$
as an idealised submodule.

When $E=\mc A$, these considerations take on a more familiar form.
The $\mc A$-module structure on $M(\mc A)$ is induced by
considering $\mc A$ as an ideal in $M(\mc A)$; that is,
$a\cdot (L,R) = (L_a,R_a)(L,R)$ and
$(L,R)\cdot a = (L,R)(L_a,R_a)$ for $a\in\mc A$ and $(L,R)\in M(\mc A)$.

If $\mc A$ is an ideal in an algebra $\mc B$, then $\mc B$ is an $\mc A$-bimodule,
and the notion of $\mc B$ being a faithful module corresponds to another notion.
For an algebra $\mc B$, we say that an ideal $I\subseteq\mc B$ is
\emph{thick} or \emph{essential} if whenever $J$ is an ideal, then $I\cap J=\{0\}$
implies that $J=\{0\}$.  The following lemma is surely folklore, but we give a sketch proof.

\begin{lemma}
Let $\mc B$ be an algebra and let $\mc A\subseteq\mc B$ be an ideal.
Consider the following properties:
\begin{enumerate}
\item\label{ess_one} Considering $\mc B$ as an $\mc A$-bimodule, $\mc B$ is faithful;
\item\label{ess_two} For $b\in\mc B$, if $aba'=0$ for all $a,a'\in\mc A$, then $b=0$;
\item\label{ess_three} $\mc A$ is essential.
\end{enumerate}
Then (\ref{ess_one})$\Leftrightarrow$(\ref{ess_two}) and
(\ref{ess_two})$\Rightarrow$(\ref{ess_three}).  If $\mc A$ is faithful over itself,
then (\ref{ess_three})$\Rightarrow$(\ref{ess_two}).
\end{lemma}
\begin{proof}
(\ref{ess_two}) is simply (\ref{ess_one}) written out in detail.
If (\ref{ess_two}) holds, then consider an ideal $J\subseteq\mc B$ with
$J\cap\mc A=\{0\}$.  For $b\in J$ and $a,a'\in\mc A$, we have that $aba'\in J$
as $J$ is an ideal, and $aba'\in \mc A$, as $\mc A$ is an ideal; so $aba'=0$.
Thus $b=0$, showing that $J=\{0\}$.  So $\mc A$ is essential.

If (\ref{ess_three}) holds then let $J=\{ b\in\mc B: aba'=0\ (a,a'\in\mc A)\}$.
For $b\in J$ and $c,c'\in\mc B$, we see that for $a,a'\in\mc A$, $acbc'a'
= (ac)b(c'a')=0$ as $\mc A$ is an ideal.  So $J$ is an ideal in $\mc B$.
If $\mc A$ is faithful over itself, then for $b\in J \cap \mc A$, we have that
$b=0$.  Thus $J=\{0\}$, showing (\ref{ess_two}).
\end{proof}

These ideas often appear in C$^*$-algebra theory, see for example
\cite[Chapter~III, Section~6]{tak1}.
Notice that if $\mc A$ is a C$^*$-algebra, then every ideal has a bounded
approximate identity, and hence is faithful over itself.

The above discussion hence shows that $M(\mc A)$ is the ``largest'' algebra which
contains $\mc A$ as an essential ideal.

To close this section, we introduce some notation.  We shall write a typical
element of $M(E)$ as $\hat x$, and will use the notation $\hat x=(L_{\hat x},R_{\hat x})$.
This notation is inspired the embedding of $E$ into $M(E)$.  For example, the calculations
in Lemma~\ref{idealiser} can be expressed as $a\cdot\hat x = R_{\hat x}(a),
\hat x\cdot a = L_{\hat x}(a)$ for $\hat x\in M(E), a\in\mc A$.  Hence we can write
the actions of the maps $L_{\hat x}, R_{\hat x}$ as module maps.  Similarly, when
$E = \mc A$, we write $\hat a=(L_{\hat a},R_{\hat a})$ for a typical element on $M(\mc A)$,
and then $L_{\hat a}(a) = \hat a a, R_{\hat a}(a) = a \hat a$ for $a\in\mc A$.

\subsection{For Banach algebras}

For a Banach algebra $\mc A$, it is customary to consider \emph{contractive}
modules, see \cite{dales} for example.  However, it will be convenient for us
to consider merely bounded modules.  We shall be careful to indicate procedures where
one starts with a contractive module but ends up with only a bounded module.
Furthermore, when $\mc A$ is a Banach algebra, by a (left/right/bi) $\mc A$-bimodule
$E$, unless otherwise stated, we will always mean that $E$ is a Banach space, and the
module actions are bounded.

An $\mc A$-bimodule $E$ is \emph{essential} when $\mc A\cdot E\cdot \mc A$ is
linearly dense in $E$.  Following Johnson, \cite{johnmem}, we shall say that $E$ is
\emph{neo-unital} if $E = \{ a\cdot x\cdot b : x\in E,a,b\in\mc A\}$.  Essential
to many of our arguments is the following result.  This, when $E=\mc A$, was
proved by Cohen in \cite{cohen}, and then extended to, essentially, the version
presented here by Hewitt in \cite{hewitt}, and independently by Curtis and
Fig\`a-Talamanca in \cite{cft}.

\begin{theorem}\label{cohen}
Let $\mc A$ be a Banach algebra with a bounded approximate identity with
bound $K>0$, and let $E$ be an essential left $\mc A$-module.  Then $E$ is neo-unital.
Indeed, for each $x\in E$ and $\epsilon>0$ there exists $y\in E$ and $a\in\mc A$
with $x=a\cdot y$, $\|x-y\|<\epsilon$ and $\|a\| \leq K$.  We can choose
$y$ in the closure of $\{ b\cdot x : b\in\mc A\}$.  A similar result holds
for right $\mc A$-modules.
\end{theorem}
\begin{proof}
See, for example, \cite[Chapter~11]{bd}, \cite[Theorem~32.22]{HR} or
\cite[Corollary~2.9.25]{dales}.
\end{proof}

Let $\mc A$ be a Banach algebra and let $E$ be an $\mc A$-bimodule.  It is natural
to consider $M(E)$ to be those multipliers $(L,R)$ such that $L,R\in\mc B(\mc A,E)$.
However, this is automatic from the Closed Graph Theorem.  Indeed, suppose that
$a_n\rightarrow a$ in $\mc A$ and that $L(a_n)\rightarrow x$ in $E$.  Then
\[ b\cdot x = \lim_n b\cdot L(a_n) = \lim_n R(b)\cdot a_n = R(b)\cdot a
= b\cdot L(a) \qquad (b\in\mc A), \]
so as $E$ is assumed faithful, $L(a)=x$, and we conclude that $L$ is bounded.
Similarly, $R$ is bounded.  We norm $M(E)$ by considering $M(E)$ as a subset of
$\mc B(\mc A,E)\oplus_\infty\mc B(\mc A,E)$, so that
$\|(L,R)\| = \max( \|L\|, \|R\| )$.
The \emph{strict topology} on $\mc B(\mc A,E)\oplus \mc B(\mc A,E)$ is defined by
the seminorms
\[ (L,R) \mapsto \|L(a)\|+\|R(a)\| \qquad (a\in\mc A). \]

\begin{proposition}\label{strict_closed}
$M(E)$ is a closed subspace of $\mc B(\mc A,E)\oplus_\infty \mc B(\mc A,E)$
in both the norm and strict topologies.  In particular, $M(E)$ is a Banach space.
\end{proposition}
\begin{proof}
Suppose that the net $(L_\alpha,R_\alpha)$ in $M(E)$ converges strictly to $(L,R)$
in $\mc B(\mc A,E)\oplus_\infty \mc B(\mc A,E)$.  For $a,b\in\mc A$, we have that
\[ a\cdot L(b) = \lim_\alpha a\cdot L_\alpha(b)
= \lim_\alpha R_\alpha(a)\cdot b = R(a)\cdot b, \]
so that $(L,R)\in M(E)$ as required.  As norm convergence implies strict convergence,
this completes the proof.
\end{proof}

In the rest of this section, we shall study various extension problems.
Almost all of these boil down to extending module actions to $M(\mc A)$,
with further extension problems following by purely algebraic methods.
We shall explore these methods here, deferring treatment of the original
module extension problem until later (where we study what extra properties of
$\mc A$, or the module in question, will ensure that such extensions exist).

Let us first consider the following problem, for a Banach algebra $\mc A$ and
an $\mc A$-bimodule $E$.  We showed above that $M(E)$ is an $\mc A$-bimodule.
Can we extend these module actions to turn $M(E)$ into an $M(\mc A)$-bimodule?

If $E$ is an $M(\mc A)$-bimodule, then we say that the module actions are
\emph{strictly continuous} if, whenever $\hat a_\alpha\rightarrow\hat a$ strictly
in $M(\mc A)$, we have that $\hat a_\alpha\cdot x \rightarrow \hat a\cdot x$,
$x\cdot\hat a_\alpha \rightarrow x\cdot\hat a$, in norm, for $x\in E$.

\begin{theorem}\label{mult_mod_is_mult_mod}
Let $\mc A$ be a Banach algebra, let $E$ be a $\mc A$-bimodule.  Suppose that
$E$ is also an $M(\mc A)$-bimodule, with actions extending those of $\mc A$.
Then there is a $M(\mc A)$-bimodule structure on $M(E)$ given by
\[ \begin{matrix}
L_{\hat a\cdot\hat x}(a) = \hat a \cdot L_{\hat x}(a),
\quad R_{\hat a\cdot\hat x}(a) = R_{\hat x}(a\hat a) \\
L_{\hat x\cdot\hat a}(a) = L_{\hat x}(\hat a a), \quad
R_{\hat x\cdot\hat a}(a) = R_{\hat x}(a) \cdot \hat a
\end{matrix} \qquad ( a\in\mc A, \hat a\in M(\mc A), \hat x\in M(E)). \]
These satisfy:
\begin{enumerate}
\item\label{mmimm:one} the module actions extend both those of $\mc A$ on $M(E)$
  and $M(\mc A)$ on $E$;
\item\label{mmimm:two} when the action of $M(\mc A)$ on $E$ is strictly continuous,
  the module action $M(\mc A)\times M(E)\rightarrow M(E)$ is
  strictly continuous in either variable; and analogously for
  $M(E)\times M(\mc A)\rightarrow M(E)$;
\end{enumerate}
With respect to condition (1), the definitions of $L_{\hat a\cdot\hat x}$ and
$R_{\hat x\cdot\hat a}$ are unique.  If $\mc A$ is essential over itself, then the
definitions of $R_{\hat a\cdot\hat x}$ and $L_{\hat x\cdot\hat a}$ are also unique.
\end{theorem}
\begin{proof}
These definitions are motivated by, and clearly extend, the module actions of $\mc A$
on $M(E)$.  For example, it is easy to see that then the pair
$(L_{\hat a\cdot\hat x}, R_{\hat a\cdot\hat x})$ is indeed a multiplier, and
similarly $\hat x\cdot\hat a$ is well-defined.  For $x\in E$ and $\hat a\in M(\mc A)$,
\[ (\hat a \cdot x)\cdot a = \hat a \cdot (x\cdot a)
= \hat a \cdot L_x(a) = L_{\hat a\cdot x}(a)
\qquad (a\in\mc A), \]
and so forth, showing that these actions extend those of $M(\mc A)$ on $E$.
We shall henceforth use fully the notation introduced at the end of
the previous section, and write, for example, the first definition as
$(\hat a\cdot \hat x)\cdot a = \hat a \cdot (\hat x\cdot a)$.

If $\hat a_\alpha\rightarrow \hat a$ strictly in $M(\mc A)$, then for $\hat x\in M(E)$
and $a\in\mc A$,
\[ (\hat a_\alpha\cdot\hat x)\cdot a = \hat a_\alpha \cdot (\hat x\cdot a)
\rightarrow \hat a\cdot (\hat x\cdot a) = (\hat a \cdot \hat x)\cdot a, \]
and similarly $a\cdot (\hat a_\alpha\cdot\hat x) \rightarrow a\cdot (\hat a\cdot\hat x)$.
Thus $\hat a_\alpha\cdot\hat x \rightarrow \hat a\cdot \hat x$ strictly in $M(E)$.
Now suppose that $\hat x_\alpha\rightarrow \hat x$ strictly in $M(E)$.
Then, for $\hat a\in M(\mc A)$,
\[ (\hat a\cdot\hat x_\alpha)\cdot a = \hat a\cdot(\hat x_\alpha\cdot a)
\rightarrow \hat a\cdot(\hat x\cdot a) = (\hat a\cdot \hat x)\cdot a
\qquad (a\in\mc A), \]
showing that $\hat a\cdot\hat x_\alpha \rightarrow \hat a\cdot\hat x$ strictly.
Analogously, these hold for the right module action.

Suppose now that we have some left-module action of $M(\mc A)$ on $M(E)$ satisfying
(1) and (2).  Let $\hat a\in M(\mc A)$ and $\hat x\in M(E)$, and let $(L,R) = \hat a
\cdot \hat x$.  Then
\[ a\cdot L(b) = (a\cdot L)(b) = (a\cdot \hat a\cdot \hat x)\cdot b
= a\hat a \cdot L_{\hat x}(b) \qquad (a,b\in\mc A). \]
As $E$ is faithful, we conclude that $L = L_{\hat a\cdot\hat x}$ as defined above.
Similarly, we see that
\[ R(ba) = (a\cdot R)(b) = b \cdot (a\cdot\hat a \cdot \hat x)
= R_{\hat x}(ba\hat a) \qquad (a,b\in\mc A), \]
so if products are dense in $\mc A$, then $R=R_{\hat a\cdot\hat x}$ as defined above.
The arguments for the right action are analogous.
\end{proof}

It might seem un-natural to first define $E$ as an $M(\mc A)$-bimodule: perhaps
it would be easier to extend the action of $\mc A$ on $E$ directly to an action
of $M(\mc A)$ on $M(E)$.  The following shows that, if we can do this, then in
many cases, $E$ will automatically be an $M(\mc A)$-submodule of $M(E)$.

\begin{proposition}
Let $E$ be an essential $\mc A$-bimodule, and suppose that $M(E)$ is an
$M(\mc A)$-bimodule satisfying condition (1) from the previous theorem.
Then $E$ is an $M(\mc A)$-bimodule, with the module actions extending those
of $\mc A$.  Furthermore, when $\mc A$ is essential over itself, the action
of $M(\mc A)$ on $M(E)$ is given by the definitions in the previous theorem.

If, further, $E =\{ a\cdot x\cdot b : a,b\in\mc A, x\in E \}$ then the action of
$M(\mc A)$ on $E$ is strictly continuous, and condition (2) holds.
\end{proposition}
\begin{proof}
Let $\hat a\in M(\mc A), x\in E$ and let $(L,R) = \hat a\cdot (L_x,R_x)\in M(E)$.
Suppose that $x=a\cdot y$ for some $a\in\mc A$ and $y\in E$, so that
$(L_x,R_x) = a\cdot (L_y,R_y)$ and thus $(L,R) = \hat a a\cdot (L_y,R_y) \in E$.
By density, as $E$ is essential, this shows that we have a module action
$M(\mc A)\times E\rightarrow E$ which extends the module action of $\mc A$.
Let $\hat a_\alpha\rightarrow\hat a$ in $M(\mc A)$.  Then
\[ \hat a_\alpha \cdot (a\cdot y) = (\hat a_\alpha a)\cdot y \rightarrow
(\hat a a)\cdot y = \hat a \cdot (a\cdot y) \qquad (a\in\mc A, y\in E), \]
so we have strict continuity, under the stronger condition on $E$ (notice that
we cannot assume that $(\hat a_\alpha)$ is bounded).

We now essentially reverse the uniqueness argument in the preeceding proof.
For $\hat x\in M(E), \hat a\in M(\mc A)$ and $a,b\in\mc A$, if $(L,R)=\hat a\cdot \hat x$,
then, as before, $R$ has a unique definition, and $a\cdot L(b) = a\hat a\cdot \hat x\cdot b
= a\hat a \cdot L_{\hat x}(b)$ for $a,b\in\mc A$.  Then $\hat a\cdot L_{\hat x}(b)\in E$
by the previous paragraph, and so $L(b) = \hat a\cdot L_{\hat x}(b)$ for $b\in\mc A$.

The arguments ``on the right'' follow analogously.  When $E$ satisfies the stronger
condition, uniqueness shows that (2) holds.
\end{proof}

We shall address the question of when $E$ \emph{is} an $M(\mc A)$-bimodule in later sections.

Another typical problem in the theory of multipliers is to extend (module) homomorphisms
to the level of multipliers.  At the level of modules, this is just algebra, as the following
proof shows.

\begin{theorem}\label{always_ext_mod_homos}
Let $E$ and $F$ be $\mc A$-modules, and let $\psi:E\rightarrow F$ be
an $\mc A$-bimodule homomorphism.  There exists a unique extension
$\tilde\psi:M(E)\rightarrow M(F)$ which is an $\mc A$-bimodule homomorphism.
Furthermore, $\tilde\psi$ is strictly continuous.

If $E$ and $F$ are also $M(\mc A)$-bimodules, with actions extending
those of $\mc A$, then use Theorem~\ref{mult_mod_is_mult_mod} to turn
$M(E)$ and $M(F)$ into $M(\mc A)$-bimodules.  Then $\tilde\psi$ is an
$M(\mc A)$-bimodule homomorphism.
\end{theorem}
\begin{proof}
We first define $\tilde\psi$ as follows.  For $\hat x \in M(E)$ define
$\tilde\psi(\hat x)=(L,R)$ where
\[ L(a) = \psi(\hat x\cdot a), \quad R(a) = \psi(a\cdot\hat x)
\qquad (a\in\mc A). \]
For $a,b\in\mc A$, $a\cdot L(b) = \psi(a\cdot\hat x\cdot b) = R(a)\cdot b$
as $\psi$ is a module homomorphism, and so $(L,R)\in M(F)$.  For $a,b\in\mc A$,
\[ \tilde\psi(a\cdot\hat x)\cdot b = \psi((a\cdot\hat x)\cdot b)
= \psi(a\cdot(\hat x\cdot b)) = a\cdot(\tilde\psi(\hat x)\cdot b), \]
and similarly $b\cdot\tilde\psi(a\cdot\hat x) = b\cdot(a\cdot\tilde\psi(\hat x))$
so that $\tilde\psi(a\cdot\hat x) = a\cdot\tilde\psi(\hat x)$.  Similarly
$\tilde\psi(\hat x\cdot a) = \tilde\psi(\hat x)\cdot a$, so that $\tilde\psi$
is an $\mc A$-module homomorphism.  Clearly $\tilde\psi$ is linear, is an extension
of $\psi$, and satisfies $\|\tilde\psi\|\leq \|\psi\|$.

If $\phi:M(E)\rightarrow M(F)$ is another extension, then for $a,b\in\mc A$,
and $\hat x\in M(E)$, we have that
\[ (a\cdot\phi(\hat x))\cdot b = \phi(a\cdot\hat x)\cdot b = \psi(a\cdot\hat x)\cdot b
= (a\cdot\tilde\psi(\hat x))\cdot b, \]
using that $\phi$ is an $\mc A$-module homomorphism, and that $a\cdot\hat x\in E$.
As $F$ is faithful, $a\cdot\phi(\hat x) = a\cdot\tilde\psi(\hat x)$, and a
similar argument establishes that $\phi(\hat x)\cdot a=\tilde\psi(\hat x)\cdot a$.
Thus $\phi=\tilde\psi$.

If $\hat x_\alpha \rightarrow \hat x$ in $M(E)$ then for $a\in\mc A$,
\[ \tilde\psi(\hat x_\alpha)\cdot a = \psi(\hat x_\alpha\cdot a)
\rightarrow \psi(\hat x\cdot a) = \tilde\psi(\hat x)\cdot a, \]
and similarly $a\cdot\tilde\psi(\hat x_\alpha)\rightarrow a\cdot\tilde\psi(\hat x)$.
Thus $\tilde\psi$ is strictly continuous.

Now suppose that $E$ and $F$ are $M(\mc A)$-bimodules, with actions extending those
of $\mc A$, and apply Theorem~\ref{mult_mod_is_mult_mod}.  Let $a,b\in\mc A$,
$\hat a\in M(\mc A)$ and $\hat x \in M(E)$.  Then
\begin{align*} b\cdot\big(\tilde\psi(\hat a\cdot\hat x) \cdot a\big)
&= b\cdot \psi ( (\hat a\cdot\hat x)\cdot a)
= b\cdot \psi( \hat a \cdot (\hat x\cdot a) 
= \psi( b\hat a \cdot (\hat x\cdot a) ) \\
&= b\hat a \cdot \psi(\hat x\cdot a)
= b\cdot\big( (\hat a \cdot \tilde\psi(\hat x))\cdot a \big ).
\end{align*}
We also have that
\begin{align*}
(\hat a \cdot \tilde\psi(\hat x))\cdot a
= \hat a \cdot (\tilde\psi(\hat x)\cdot a)
= \hat a \cdot \psi(\hat x\cdot a)
\end{align*}
As $F$ is faithful, it follows that $\tilde\psi(\hat a\cdot\hat x) \cdot a
= (\hat a \cdot \tilde\psi(\hat x))\cdot a$.  Similarly,
\[ a\cdot\tilde\psi(\hat a \cdot\hat x) = \psi( a\cdot(\hat a\cdot\hat x))
= \psi( a\hat a \cdot\hat x)
= a\hat a \cdot \tilde\psi(\hat x) = a\cdot(\hat a\cdot\tilde\psi(\hat x)). \]
We conclude that $\tilde\psi(\hat a\cdot\hat x) = \hat a \cdot \tilde\psi(\hat x)$.
Analogously, one can show that $\tilde\psi(\hat x\cdot\hat a) = \tilde\psi(\hat x)
\cdot \hat a$.
\end{proof}

Now suppose that $\mc B$ is a Banach algebra and that $\theta:\mc A\rightarrow M(\mc B)$
is a bounded homomorphism.  Then $\mc B$ becomes a bounded (but maybe not contractive!)
$\mc A$-bimodule for the actions
\[ a\cdot b = \theta(a)b, \quad b\cdot a= b\theta(a)
\qquad (a\in\mc A, b\in\mc B). \]
There appears to be no simple criterion on $\theta$ to ensure that $\mc B$ is then
a faithful $\mc A$-module.  If $\mc B$ is faithful, then we can apply the above theorem
to find an extension $\tilde\theta:M(\mc A)\rightarrow M_{\mc A}(\mc B)$ which is an
$\mc A$-bimodule homomorphism.
There is a linear contraction $M(\mc B)\rightarrow M_{\mc A}(\mc B)$ given by
$(L,R)\mapsto (L\theta,R\theta)$.
However, it is far from clear when $\tilde\theta$ maps into (the image of) $M(\mc B)$.

\subsection{Extending module actions and homomorphisms}

We saw in the previous section that the ability to extend the bimodule actions of
$\mc A$ on $E$ to $M(\mc A)$ actions is a sufficient (and often necessary)
condition for $M(E)$ to become an $M(\mc A)$-bimodule, in a natural way.
In this section, we shall see that extending the action on $E$ has a close
relation with the problem of extending homomorphisms between algebras.

Let $\mc A$ be a Banach algebra and let $E$ be a bimodule.  Recall that
$E$ is \emph{essential} if $\mc A\cdot E\cdot \mc A$ is linearly dense in $E$.
As an aside, we note that in the pure algebra setting, we would ask that the linear span
of $\mc A\cdot E\cdot \mc A$ be all of $E$.  In this setting, the proofs
are similar, compare with \cite[Appendix]{vandaele} for example.  We shall concentrate
upon the Banach algebra case.

Now suppose that $\mc B$ is a Banach algebra and that $\theta:\mc A\rightarrow M(\mc B)$
is a bounded homomorphism.  Then $\mc B$ becomes a bounded $\mc A$-bimodule as above.
Then $\mc B$ is an essential $\mc A$-bimodule if the linear span of
$\{ \theta(a_1) b \theta(a_2) : a_1,a_2\in\mc A, b\in\mc B \}$ is dense in $\mc B$.
This is often referred to as $\theta$ being \emph{non-degenerate}.  Notice that the
following does not need that $\mc B$ is a faithful $\mc A$-bimodule.

\begin{proposition}\label{homo_to_mod}
Let $\theta:\mc A\rightarrow M(\mc B)$ be a non-degenerate homomorphism.
Then the following are equivalent:
\begin{enumerate}
\item\label{htm:one} the module actions on $\mc B$ can be
extended to bounded $M(\mc A)$-module actions;
\item\label{htm:two} there is a bounded homomorphism
$\tilde\theta:M(\mc A)\rightarrow M(\mc B)$ extending $\theta$.
\end{enumerate}
We may replace ``bounded'' by ``contractive''.
The extensions, if they exist, are unique and strictly continuous.
\end{proposition}
\begin{proof}
If (\ref{htm:one}) holds then let $\hat a\in M(\mc A)$ and define $L,R:\mc B\rightarrow\mc B$
by $L(b) = \hat a\cdot b$ and $R(b) = b\cdot\hat a$ for $b\in\mc B$.  Let $b_1,b_2\in\mc B$
and $a_1,a_2\in\mc A$, so that
\[ b_1\theta(a_1) L(\theta(a_2)b_2) = b_1\big( a_1 \cdot \hat a \cdot a_2 \cdot b_2 \big)
= (b_1\cdot a_1\hat a a_2) \cdot b_2 = R(b_1\theta(a_1)) \theta(a_2)b_2. \]
As $\theta$ is non-degenerate, this is enough to show that $(L,R)\in M(\mc B)$.
Denote $(L,R)$ by $\tilde\theta(\hat a)$.  Then $\tilde\theta:M(\mc A)\rightarrow M(\mc B)$
is bounded and linear; if the module action is contractive, then so is $\tilde\theta$.
As $\tilde\theta$ is built from a module action, it is clear that $\tilde\theta$ is a
homomorphism, showing (\ref{htm:two}).

Conversely, if such a $\tilde\theta$ exists, then as above, we find that $\mc B$ is
a $M(\mc A)$-bimodule, and obviously these module actions extend those of $\mc A$.

Let us consider (\ref{htm:two}).  If $\tilde\theta$ exists, then
\[ \tilde\theta(\hat a) \theta(a) b = \theta(\hat a a) b, \quad
b \theta(a) \tilde\theta(\hat a) = b \theta(a\hat a)
\qquad (\hat a\in M(\mc A), a\in\mc A, b\in\mc B), \]
which as $\theta$ is non-degenerate, uniquely determines $\tilde\theta$.
Similarly, if $(\hat a_\alpha)$ is a net in $M(\mc A)$ converging strictly to $\hat a$,
then
\[ \lim_\alpha \tilde\theta(\hat a_\alpha) \theta(a)b =
\lim_\alpha \theta(\hat a_\alpha a) b = \theta(\hat a a) b
= \tilde\theta(\hat a) \theta(a)b \qquad (a\in\mc A,b\in\mc B), \]
and similarly ``on the right'', which shows that $\tilde\theta$ is
strictly continuous.  Similar remarks apply to the case of extending the module actions.
\end{proof}

Consequently, the problem of extending module actions is more general than extending
homomorphisms, at least if we restrict to essential modules and non-degenerate
homomorphisms.
For Hilbert C$^*$-modules, the framework of \emph{adjointable operators} provides
a way to pass between modules and algebras in a more seamless way.

The notion of non-degenerate is \emph{de rigueur} in C$^*$-theory, see
\cite[Chapter~2]{lance} for example.  Analogously, it is usual to consider essential modules
in Banach algebra theory.  By the above, we know that extensions will always be unique
under such conditions.  If our algebra has a bounded approximate
identity then we can construct extensions, see Theorem~\ref{exten_thm}.
Essentiallity seems like a reasonable minimal condition to consider,
but at least in principle, it would be interesting to consider wider classes of modules.
We shall not consider this problem, except when dealing with dual Banach algebras,
where the weak$^*$-topology can be used to deal with the non-essential setting,
see Section~\ref{dba_section}.

\subsection{Tensor products}\label{ten_prod_sec}

A motivating example for us is the following.  Let $\mc A$ be an algebra, and let $E$
be a vector space.  Consider the vector space $\mc A\otimes E$ (here and elsewhere,
an unadorned tensor product means the algebraic tensor product) which is an
$\mc A$-bimodule for the actions
\[ a\cdot (b\otimes x) = ab\otimes x, \quad
(b\otimes x)\cdot a = ba\otimes x
\qquad (a\in\mc A, b\otimes x\in \mc A\otimes E). \]
Indeed, it is not too hard to show that $M(\mc A)\otimes E$ is isomorphic to
$M(\mc A\otimes E)$ under the map which sends $(l,r)\otimes x$ to $(L,R)$
where $L(a) = l(a)\otimes x$ and $R(a)=r(a)\otimes x$ for $a\in\mc A$.
The Banach algebra case is somewhat more subtle!

To consider a suitable Banach algebra version, we have to consider completions
of tensor products.  Again, we work with a little generality.  Let $\proten$ denote
the \emph{projective tensor product}, so for Banach spaces $E$ and $F$,  $E\proten F$
is the completion of $E\otimes F$ under the norm
\[ \|\tau\| = \inf\Big\{ \sum_{k=1}^n \|x_k\|\|y_k\| : \tau = \sum_{k=1}^n
x_k\otimes y_k \Big\} \qquad (\tau\in E\otimes F). \]
Then $\proten$ has the universal property that if $\phi:E\times F\rightarrow G$
is a bounded bilinear map to some Banach space $G$, then there is a unique bounded
linear map $\tilde\phi:E\proten F\rightarrow G$ linearising $\phi$.  The dual
of $E\proten F$ can be identified with $\mc B(E,F^*)$ or $\mc B(F,E^*)$ by
\[ \ip{T}{x\otimes y} = \ip{T(x)}{y}, \quad
\ip{S}{x\otimes y} = \ip{S(y)}{x}, \]
where $x\otimes y\in E\proten F$, $T\in\mc B(E,F^*)$ and $S\in\mc B(F,E^*)$.

Let $\mc A$ be a Banach algebra and let $E$ be a Banach space.
We shall say that a norm $\alpha$ on $\mc A\otimes E$ is \emph{admissible} if
$\alpha(a\otimes x) = \|a\| \|x\|$ for $a\in\mc A$ and $x\in E$ (which says that
$\alpha$ is a \emph{cross-norm}) and the module actions of $\mc A$ are contractive.
The triangle inequality then shows that the projective tensor norm dominates $\alpha$.
Denote by $\mc A\proten_\alpha E$ the completion of $\mc A\otimes E$ under $\alpha$.
Then $\mc A\proten E\rightarrow \mc A\proten_\alpha E$ is norm-decreasing with
dense range.  The adjoint of this map thus identifies $(\mc A\proten_\alpha E)^*$
with a subspace of $\mc B(\mc A,E^*)$, which we shall denote by $\mc B_\alpha(\mc A,E^*)$.
We equip $\mc B_\alpha(\mc A,E^*)$ with the norm induced by $(\mc A\proten_\alpha E)^*$,
say $\|\cdot\|_\alpha$.  Thus
\[ \|T\|_\alpha \geq \|T\|, \quad \|a\cdot T\|_\alpha \leq \|a\|\|T\|_\alpha,
\quad \|T\cdot a\|_\alpha \leq \|a\|\|T\|_\alpha \qquad (a\in\mc A, T\in\mc B_\alpha(\mc A,E^*)). \]
Here $(a\cdot T)(b) = T(ba)$ and $(T\cdot a)(b) = T(ab)$ for $b\in\mc A$.

It is not obvious that $\mc A\proten_\alpha E$ will be faithful if $\mc A$ is;
however, \emph{henceforth we assume} that $\mc A\proten_\alpha E$ is a faithful module.
In particular, in examples, this will need to be checked.

For a given $E$, suppose for all choices of $\mc A$ we have an admissible norm $\alpha$
on $\mc A\otimes E$, and that if $T:\mc A\rightarrow\mc B$ is a contraction, then
$T\otimes I_E:\mc A\otimes E\rightarrow\mc B\otimes E$ is a contraction with respect
to $\alpha$.  Then we say that $\alpha$ is \emph{$E$-admissible}.

\begin{lemma}\label{lemma::tensor}
Let $\mc A$ be a Banach algebra, let $E$ be a Banach space, and let $\alpha$ be an admissible
norm.  There is a natural embedding $M(\mc A) \otimes E \rightarrow M(\mc A\proten_\alpha E)$
which is an $\mc A$-bimodule homomorphism.
If $\alpha$ is $E$-admissible, this extends to a contraction $M(\mc A)\proten_\alpha E
\rightarrow M(\mc A\proten_\alpha E)$.
\end{lemma}
\begin{proof}
Let $\hat a\in M(\mc A)$ and $x\in E$, and define
$L,R:\mc A\rightarrow \mc A\proten_\alpha E$ by
\[ L(a) = \hat a a \otimes x, \quad R(a) = a \hat a \otimes x
\qquad (a\in\mc A). \]
Then, for $a,b\in\mc A$, we have that $a\cdot L(b) = a \hat a b \otimes x = R(a) \cdot b$,
so that $(L,R)\in M(A\proten_\alpha E)$.  Write $\beta(\hat a\otimes x)$ for $(L,R)$, so
that $\beta:M(\mc A)\times E \rightarrow M(\mc A \proten_\alpha E)$ is bilinear.
Thus $\beta$ extends to a linear map on $M(\mc A)\otimes E$.  Suppose that
$\sum_{k=1}^n \hat a_k \otimes x_k$ is mapped to the zero multiplier.  We may
suppose that $\{x_k\}$ is linearly independent, so for each $a\in\mc A$,
$\hat a_ka=0$ for each $k$.  Hence $\hat a_k=0$ for each $k$, and we conclude that
$\beta$ is an injection.

For $\hat a\otimes x\in M(\mc A)\otimes E$ and $a\in\mc A$, we have
\[ \beta(a\cdot(\hat a\otimes x))\cdot b = (a\hat a)b \otimes x
= a\cdot (\beta(\hat a\otimes x)\cdot b) \qquad (b\in\mc A). \]
Similar calculations establish that $\beta$ is an $\mc A$-bimodule homomorphism.

If $\alpha$ is $E$-admissible, then fix $a\in\mc A$ with $\|a\|\leq 1$,
and define contractions $S,T:M(\mc A)\rightarrow\mc A$ by
\[ S(\hat a) = \hat a a, \quad T(\hat a) = a\hat a \qquad(\hat a\in M(\mc A)). \]
Then notice that for $\tau\in M(\mc A)\otimes E$, if $\beta(\tau) = (L,R)$, then
\[ L(a) = (S\otimes I_E)\tau, \quad R(a) = (T\otimes I_E)\tau. \]
Thus $\max(\|L(a)\|,\|R(a)\|) \leq \|\tau\|$.  As $a$ was arbitrary, we conclude
that $\|(L,R)\| \leq \|\tau\|$.  Thus $\beta$ is a contraction, and so extends by
continuity to $M(\mc A)\proten_\alpha E$.
\end{proof}

Let $\alpha$ be an $E$-admissible norm.  Given a non-degenerate homomorphism
$\theta:\mc A\rightarrow M(\mc B)$, we have the chain of maps
\[ \xymatrix{ \mc A\proten_\alpha E \ar[r]^{\theta\otimes I_E} &
M(\mc B)\proten_\alpha E \ar[r] & M_{\mc B}(\mc B\proten_\alpha E). } \]
The composition is an $\mc A$-bimodule homomorphism, if $\mc A$ acts on $\mc B$
(and hence on $\mc B\proten_\alpha E$) in the usual way induced by $\theta$.
Then we can apply Theorem~\ref{always_ext_mod_homos} to find a unique,
strictly-continuous extensions
\[ \theta\otimes I_E : M_{\mc A}(\mc A\proten_\alpha E) \rightarrow
M_{\mc B}(\mc B\proten_\alpha E). \]

Extending maps defined on $E$ is also easy, provided we have suitable tensor
norms.  Suppose for all $E$ we have an admissible norm $\alpha$ on $\mc A\otimes E$,
and that if $T:E\rightarrow F$ is a contraction, then $I_{\mc A}\otimes T$ extends to
a contraction $\mc A\proten_\alpha E\rightarrow \mc A\proten_\alpha F$.  Then we say that
$\alpha$ is \emph{uniformly admissible}.

\begin{proposition}\label{tensor_ext_easy}
Let $\mc A$ be a Banach algebra, $\alpha$ be a uniformly admissible norm, and
let $E$ and $F$ be Banach spaces.  Any $T\in\mc B(E,F)$ can be extended uniquely to
an $\mc A$-bimodule homomorphism $T_{\mc A}:M(\mc A\proten_\alpha E) \rightarrow
M(\mc A\proten_\alpha F)$ which satisfies $T_{\mc A}(\hat a\otimes x) = \hat a\otimes T(x)$
for $\hat a\in M(\mc A)$ and $x\in E$, and with $\|T_{\mc A}\| \leq \|T\|$.
\end{proposition}
\begin{proof}
Given $\hat x\in M(\mc A\proten_\alpha E)$, define $L,R\in\mc B(\mc A,\mc A\proten_\alpha F)$ by
$L = (I_{\mc A}\otimes T)\circ L_{\hat x}$ and $R = (I_{\mc A}\otimes T)\circ R_{\hat x}$.
Then $(L,R)$ is a multiplier; denote this by $T_{\mc A}(\hat x)$.  It is now easy to verify
that $T_{\mc A}$ has the stated properties.
\end{proof}

\section{When we have a bounded approximate identity}\label{sec::bai}

In this section, we shall consider the case when a Banach algebra $\mc A$ admits
a bounded approximate identity.  We shall see that we can form extensions to multiplier
algebras, a fact known since the start of the theory, see \cite[Section~1.d]{johnmem}.
We shall develop the theory by using the Arens products, in a similar way to
\cite[Theorem~2.9.49]{dales} and \cite{mck}.

The Arens products are discussed in \cite[Theorem~2.6.15]{dales} and
\cite[Section~1.4]{palmer}.
We shall define the Arens products in a slightly unusual way, but our construction
will self-evidently generalise to operator spaces.  Recall that we identify the dual
of $\mc A\proten\mc A$ with $\mc B(\mc A,\mc A^*)$ by
\[ \ip{T}{a\otimes b} = \ip{T(a)}{b} \qquad (a\otimes b\in\mc A\proten\mc A,
T\in\mc B(\mc A,\mc A^*)). \]
We then have two embeddings of $\mc A^{**} \proten \mc A^{**}$ into
$(\mc A\proten\mc A)^{**} = \mc B(\mc A,\mc A^*)^*$, say
\[ \Phi\otimes\Psi \mapsto \Phi\otimes_{\aone}\Psi \text{ and }
\Phi\otimes_{\atwo}\Psi \qquad (\Phi,\Psi\in\mc A^{**}), \]
which are defined by
\[ \begin{matrix} \ip{\Phi\otimes_{\aone}\Psi}{T} = \ip{T^{**}(\Phi)}{\Psi},\\
\ip{\Phi\otimes_{\atwo}\Psi}{T} = \ip{T^{***}\kappa_{\mc A}^{**}(\Phi)}{\Psi}
\end{matrix}
\qquad (\Phi,\Psi\in\mc A^{**}, T\in\mc B(\mc A,\mc A^*)). \]
For operator spaces, the following observation will be useful.  Let $\alpha:
\mc B(\mc A,\mc A^*)\rightarrow\mc B(\mc A^{**},\mc A^{***})$ be the map $\alpha(T)
=T^{**}$, and let $\beta:\mc B(\mc A^{**},\mc A^{***}) \rightarrow (\mc A^{**}\proten
\mc A^{**})^*$ be the usual isomorphism.  Then the map $\Phi\otimes\Psi
\mapsto \Phi\otimes_{\aone}\Psi$ is simply the pre-adjoint of $\beta\circ\alpha$.
Similar remarks apply to $\otimes_{\atwo}$.

Let $\pi:\mc A\proten\mc A\rightarrow\mc A$ be the product map.  For $\mu\in\mc A^*$,
the map $\pi^*(\mu)\in\mc B(\mc A,\mc A^*)$ is the map $a\mapsto\mu\cdot a$, for $a\in\mc A$.
Then we define maps $\aone,\atwo:\mc A^{**}\proten\mc A^{**}\rightarrow\mc A^{**}$ by
\[ \ip{\Phi\aone\Psi}{\mu} = \ip{\Phi\otimes_{\aone}\Psi}{\pi^*(\mu)}, \quad
\ip{\Phi\atwo\Psi}{\mu} = \ip{\Phi\otimes_{\atwo}\Psi}{\pi^*(\mu)}
\qquad (\Phi,\Psi\in\mc A^{**}, \mu\in\mc A^*). \]
These are contractive associative algebra products on $\mc A^{**}$, called
the \emph{Arens products}, which extend the bimodule actions of $\mc A$ on $\mc A^{**}$,
where we embed $\mc A$ into $\mc A^{**}$ in the canonical fashion.  More conventionally,
we turn $\mc A^*$ into an $\mc A$-bimodule is the usual way, and then define actions
of $\mc A^{**}$ on $\mc A^*$ by
\[ \ip{\Phi\cdot\mu}{a} = \ip{\Phi}{\mu\cdot a}, \quad
\ip{\mu\cdot\Phi}{a} = \ip{\Phi}{a\cdot\mu} \qquad
(\Phi\in\mc A^{**}, \mu\in\mc A^*, a\in\mc A). \]
Then $\aone$ and $\atwo$ satisfy
\[ \ip{\Phi\aone\Psi}{\mu} = \ip{\Phi}{\Psi\cdot\mu}, \quad
\ip{\Phi\atwo\Psi}{\mu} = \ip{\Psi}{\mu\cdot\Phi}
\qquad (\Phi,\Psi\in\mc A^{**}, \mu\in\mc A^*). \]

The following is stated for Banach algebras with a contractive approximate identity in
\cite[Theorem~2.9.49]{dales} and is modeled on \cite{mck}.  Notice
that if $\mc A$ has a bounded approximate identity, then $\mc A$ is essential over itself.

\begin{theorem}\label{arens}
Let $\mc A$ be a Banach algebra with a bounded approximate identity $(e_\alpha)$.
Let $\Phi_0\in\mc A^{**}$ be a weak$^*$-accumulation point of $(e_\alpha)$.
For an essential $\mc A$-bimodule $E$ we have that:
\begin{enumerate}
\item\label{arens_one} $E$ is a closed submodule of $M(E)$ which is strictly dense;
\item\label{arens_two} $\theta:M(E)\rightarrow E^{**}$, defined by
$(L,R)\mapsto L^{**}(\Phi_0)$, is an $\mc A$-bimodule homomorphism, and
an isomorphism onto its range, with $\theta(x)=\kappa_E(x)$ for $x\in E$;
\item\label{arens_three} the image of $\theta$ is contained in the idealiser of $E$ in $E^{**}$.
Indeed, $R(a) = a \cdot L^{**}(\Phi_0)$
and $L(a) = L^{**}(\Phi_0)\cdot a$ for $a\in\mc A$ and $(L,R)\in M(E)$.
\end{enumerate}
When $E=\mc A$, the map $\theta$ is a homomorphism $M(\mc A)\rightarrow (\mc A^{**},\atwo)$.
\end{theorem}
\begin{proof}
For (\ref{arens_one}), let $(L,R)\in M(E)$ and $a\in\mc A$, so that
$L(a) = \lim_\alpha L(e_\alpha a) = \lim_\alpha L(e_\alpha)\cdot a$.
As $E$ is essential, we see that $e_\alpha\cdot x\rightarrow x$ for each $x\in E$.
Thus $R(a) = \lim_\alpha R(a)\cdot e_\alpha = \lim_\alpha a\cdot L(e_\alpha)$.
We conclude that $L(e_\alpha) \rightarrow (L,R)$ strictly.

By passing to a subnet, we may suppose that $e_\alpha\rightarrow\Phi_0$
in the weak$^*$-topology on $E^{**}$.  It follows that $L(e_\alpha) \rightarrow
L^{**}(\Phi_0)$ weak$^*$ in $E^{**}$, for each $(L,R)\in M(E)$.  Thus, from above,
\[ L(a) = L^{**}(\Phi_0)\cdot a, \quad R(a) = a \cdot L^{**}(\Phi_0)
\qquad (a\in\mc A), \]
showing (\ref{arens_three}).  Notice that for $a\in\mc A$, $\|L(a)\|
= \|L^{**}(\Phi_0)\cdot a\| \leq \|L^{**}(\Phi_0)\| \|a\|$, so that $\|L\| \leq
\|L^{**}(\Phi_0)\| = \|\theta(L,R)\|$.  Similarly, $\|R\| \leq \|\theta(L,R)\|$,
and so $\theta$ is norm-increasing.  For $x\in E$, we have that
\[ \theta(x) = L_x^{**}(\Phi_0) = \lim_\alpha \kappa_EL_x(e_\alpha)
= \lim_\alpha \kappa_E(x\cdot e_\alpha) = \kappa_E(x), \]
again using that $E$ is essential.  In particular, if $(x_n)$ is a sequence
in $E$ converging to $(L,R)$ in norm in $M(E)$, then $(\theta(x_n))$ is
a Cauchy sequence in $E^{**}$, which means that $(\kappa_E(x_n))$ is Cauchy,
that is, $(x_n)$ is Cauchy.  So $(L,R)\in E$ and $E$ is closed in $M(E)$,
which completes showing (\ref{arens_one}).

Finally, for $(L,R)\in M(E)$ and $a,b\in\mc A$,
\[ a\cdot\theta(L,R) = a\cdot L^{**}(\Phi_0) = \lim_\alpha a\cdot L(e_\alpha)
= \lim_\alpha R(a)\cdot e_\alpha = R(a) = a\cdot(L,R), \]
from Lemma~\ref{idealiser}, showing that $\theta$ is a left $\mc A$-module
homomorphism.  It follows similarly that $\theta$ is a right $\mc A$-module
homomorphism, which completes showing (\ref{arens_two}).

When $E=\mc A$, for $(L_1,R_1),(L_2,R_2)\in M(\mc A)$ and $\mu\in\mc A^*$, we have that
\begin{align*} \ip{L_1^{**}(\Phi_0) \atwo L_2^{**}(\Phi_0)}{\mu}
&= \lim_\alpha \ip{\mu\cdot L_1^{**}(\Phi_0)}{L_2(e_\alpha)}
= \lim_\alpha \lim_\beta \ip{L_2(e_\alpha)\cdot\mu}{L_1(e_\beta)} \\
&= \lim_\alpha \lim_\beta \ip{\mu}{L_1\big( e_\beta L_2(e_\alpha) \big)}
= \lim_\alpha \ip{\mu}{L_1L_2(e_\alpha)} \\
&= \ip{(L_1L_2)^{**}(\Phi_0)}{\mu}.
\end{align*}
This shows that $\theta$ is a homomorphism for $\atwo$.
\end{proof}

Note that we cannot, in general, identify the image of $\theta$ with the idealiser
of $E$ in $E^{**}$ as $E^{**}$ may not be faithful.  However, we could instead
work with the canonical faithful quotient of $E^{**}$.  We shall explore a more
general idea shortly.

Extending module actions is rather easy in this setting, once we have
Theorem~\ref{cohen}.  We can then apply Theorem~\ref{mult_mod_is_mult_mod}
and Proposition~\ref{homo_to_mod} to find further extensions.

\begin{theorem}\label{exten_thm}
Let $\mc A$ be a Banach algebra with a bounded approximate identity, and let $E$ be
an essential $\mc A$-bimodule.  Then $E$ carries a unique bounded $M(\mc A)$-bimodule
structure extending the $\mc A$-bimodule structure.  If $\mc A$ has a contractive
approximate identity, then $E$ becomes a (contractive) $M(\mc A)$-bimodule.
\end{theorem}
\begin{proof}
Johnson showed this in \cite[Section~1.d]{johnson}, so we only give a sketch.
By Theorem~\ref{cohen}, each $x\in E$ has the form $a\cdot y\cdot b$ for some
$a,b\in\mc A$ and $y\in E$.  Then we define
\[ \hat a\cdot x = (\hat a a)\cdot y\cdot b, \quad
x\cdot\hat a = a\cdot y \cdot (b \hat a) \qquad ( \hat a\in M(\mc A) ). \]
By using that $E$ is faithful, we may check that these are well-defined actions;
clearly they extend the module actions of $\mc A$.  Uniqueness follows from
Proposition~\ref{homo_to_mod}.

If $\mc A$ has a contractive approximate identity, then Theorem~\ref{cohen}
gives that for $x\in E$, we can write $x=a\cdot z$ where $z$ is arbitrarily close
to $x$, and $\|a\|\leq 1$.  Similarly, we can write $z = y\cdot b$ with $y$
arbitrarily close to $z$, and with $\|b\|\leq 1$.  Thus
\[ \|\hat a\cdot x\| = \|\hat a a \cdot y \cdot b\|
\leq \|\hat a a \| \|y\| \|b\| \leq \|\hat a\| \|y\|, \]
which can be made arbitrarily close to $\|\hat a\| \|x\|$.  Hence the left
module action of $M(\mc A)$ is contractive; similarly on the right.
\end{proof}

\subsection{Smaller submodules}

Theorem~\ref{arens} identifies $M(E)$ with a subspace of $E^{**}$.  However, $E^{**}$
can both be very large, and can fail to be faithful.  In applications, it is often
useful to work with a smaller $\mc A$-bimodule.

Let $F\subseteq E^*$ be a closed $\mc A$-submodule of $E^*$.  Denote by
$q_F$ the quotient map $E^{**}\rightarrow F^*=E^{**}/F^\perp$, which is an
$\mc A$-bimodule homomorphism.  Let $\theta_F=q_F\theta:M(E)
\rightarrow F^*$, where $\theta$ is given by Theorem~\ref{arens}.
Let $\iota_F:E\rightarrow F^*$ be the natural map, which is $q_F\kappa_E$.

\begin{theorem}\label{submod_rep_thm}
Let $\mc A$ be a Banach algebra with a bounded approximate identity, let $E$ be an
essential $\mc A$-bimodule, and let $F\subseteq E^*$ be a closed submodule such that
$F\cdot\mc A$ or $\mc A\cdot F$ is dense in $F$.  In particular, this holds if $F$
is essential.  Then:
\begin{enumerate}
\item\label{dbr:zero} for $\hat x\in M(E)$ and $a\in\mc A$, we have
  $\theta_F(\hat x)\cdot a = \iota_F(\hat x\cdot a)$ and
  $a\cdot\theta_F(\hat x) = \iota_F(a\cdot\hat x)$;
\item\label{dbr:one} $\theta_F:M(E)\rightarrow F^*$ is strictly-weak$^*$ continuous;
\item\label{dbr:two} $\iota_F:E\rightarrow F^*$ is injective
  if and only if $\theta_F$ is injective;
\item\label{dbr:three} $\iota_F$ is bounded below if and only if $\theta_F$ is bounded below;
\item\label{dbr:four} when $\theta_F$ is injective, the image of $\theta_F$ is the
  idealiser of $E$ in $F^*$.
\end{enumerate}
\end{theorem}
\begin{proof}
With reference to Theorem~\ref{arens}, let $\Phi_0\in\mc A^{**}$ be the
weak$^*$-limit of the bounded approximate identity $(e_\alpha)$, so that
$\theta_F(L,R) = q_F L^{**}(\Phi_0)$ for $(L,R)\in M(E)$.
Thus, for $a\in\mc A$ and $\hat x\in M(E)$,
\[ \theta_F(\hat x) \cdot a = q_F( \theta(\hat x)\cdot a )
= q_F \kappa_E (\hat x\cdot a) = \iota_F(\hat x\cdot a). \]
Similarly, $a\cdot\theta_F(\hat x) = \iota_F(a\cdot\hat x)$, showing (\ref{dbr:zero}).

Suppose that $F\cdot\mc A$ is dense in $F$, so by Theorem~\ref{cohen}, a typical
element of $F$ has the form $\mu\cdot a$ for some
$a\in\mc A$ and $\mu\in F$.  Suppose that $(L_\alpha,R_\alpha) \rightarrow (L,R)$
strictly in $M(E)$.  Then
\begin{align*} \lim_\alpha \ip{L_\alpha^{**}(\Phi_0)}{\mu\cdot a}
&= \lim_\alpha \lim_\beta \ip{\mu\cdot a}{L_\alpha(e_\beta)}
= \lim_\alpha \lim_\beta \ip{\mu}{R_\alpha(a)\cdot e_\beta}
= \lim_\alpha \ip{\mu}{R_\alpha(a)} \\
&= \ip{\mu}{R(a)} = \ip{L^{**}(\Phi_0)}{\mu\cdot a},
\end{align*}
using that $E$ is essential.  It follows that
$q_FL_\alpha^{**}(\Phi_0) \rightarrow q_FL^{**}(\Phi_0)$ weak$^*$ in $F^*$, as required.
The case when $\mc A\cdot F$ is dense in $F$ is similar, so we've shown (\ref{dbr:one}).

We have seen before that as $E$ is faithful, for $\hat x\in M(E)$, we have that
$\hat x=0$ if and only if $\hat x\cdot a=0$ for all $a\in\mc A$.
Let $\iota_F$ be injective, and suppose that $\theta_F(\hat x)=0$.  From
(\ref{dbr:zero}), it follows that $\iota_F(\hat x\cdot a) = \theta_F(\hat x)\cdot a = 0$
for all $a\in\mc A$, so that $\hat x=0$.  So $\theta_F$ injects.  Conversely,
as $\theta_F(x) = q_F \kappa_E(x) = \iota_F(x)$ for $x\in E$, if $\theta_F$ injects,
then certainly $\iota_F$ injects, so (\ref{dbr:two}) holds.

By Theorem~\ref{arens}, the inclusion $E\rightarrow M(E)$ is an isomorphism
onto its range.  As $\iota_F = \theta_F|_E$, it follows that if $\theta_F$ is
bounded below, then so is $\iota_F$.  
If $\iota_F$ is bounded below by $\delta>0$, then
\[ \|\theta_F(\hat x)\| \|a\| \geq \|\theta_F(\hat x)\cdot a\|
= \|\iota_F(\hat x\cdot a)\| \geq \delta \|L_{\hat x}(a)\| \qquad (a\in\mc A). \]
Thus we have that $\|\theta_F(\hat x)\| \geq \|L_{\hat x}\|$, and similarly,
$\|\theta_F(\hat x)\| \geq \|R_{\hat x}\|$.  So (\ref{dbr:three}) holds.

For (\ref{dbr:four}), we first note that by Theorem~\ref{arens}, $\theta_F$ maps
into the idealiser of $E$.  Conversely, if $\Phi\in F^*$ is such that
$\mc A\cdot\Phi, \Phi\cdot\mc A\subseteq \iota_F(E)$, then there exist
linear maps $L,R:\mc A \rightarrow E$ with
\[ \iota_F L(a) = \Phi\cdot a, \quad \iota_F R(a) = a\cdot\Phi
\qquad (a\in\mc A). \]
As $\iota_F$ is injective, it follows that $(L,R)\in M(E)$.  Suppose that
$\mc A\cdot F$ is dense in $F$ (the other case is similar) so that a typical
element of $F$ has the form $\mu\cdot a$ for some $a\in\mc A$
and $\mu\in F$.  Then
\begin{align*} \ip{\theta_F(L,R)}{\mu\cdot a}
&= \lim_\alpha \ip{\mu\cdot a}{L(e_\alpha)}
= \lim_\alpha \ip{\mu}{R(a)\cdot e_\alpha} = \ip{\mu}{R(a)} \\
&= \ip{a\cdot\Phi}{\mu} = \ip{\Phi}{\mu\cdot a}, \end{align*}
as $E$ is essential.  Thus $\theta_F(L,R) = \Phi$, as required.
The case when $\mc A\cdot F$ is dense in $F$ is similar.
\end{proof}

The power of this result is illustrated by the following example.  Let $G$ be a locally
compact group, and consider when $\mc A=E=L^1(G)$.  Then $\mc A$ has a contractive
bounded approximate identity, so Theorem~\ref{arens} applies, and we can consider
$M(\mc A)$ as a subalgebra of $L^1(G)^{**}$.  This, however, is a very large space!
Instead, consider $F=C_0(G)$ which is, see \cite[Theorem~3.3.23]{dales},
an essential submodule of $\mc A^*$.
Furthermore, the natural map $\mc A\rightarrow F^*$ is an isometry in this case.
Thus we may (isometrically) identify $M(\mc A)$ with the idealiser of $\mc A$ in $F^*$.
In this case, $F^*=M(G)$ the measure algebra, and so $M(\mc A)=F^*=M(G)$ (and
we have essentially reproved Wendel's Theorem, compare \cite[Theorem~3.3.40]{dales}).

We finish this section by proving a ``dual'' version of the above, which is an
easier result.

\begin{proposition}\label{dba_dual_mod}
Let $\mc A$ be a Banach algebra with a bounded approximate identity, let $E$ be an
essential $\mc A$-bimodule, and let $F\subseteq E^*$ be a closed submodule.  Then
there is a unique bounded $\mc A$-module homomorphism $\phi_F$ from $M(F)$ into the idealiser
of $F$ in $E^*$ which extends the inclusion $F\rightarrow E^*$.  Furthermore, $\phi_F$
satisfies:
\begin{enumerate}
\item $\hat x \cdot a = \phi_F(\hat x)\cdot a$ and $a\cdot\hat x = a\cdot\phi_F(\hat x)$
for $a\in\mc A$ and $\hat x\in M(F)$;
\item $\phi_F$ is strictly-weak$^*$ continuous;
\end{enumerate}
\end{proposition}
\begin{proof}
Let $(e_\alpha)$ be a bounded approximate identity for $\mc A$, and define
\[ \phi_F(\hat x) = \lim_\alpha \hat x\cdot e_\alpha \qquad(\hat x\in M(F)), \]
with the limit taken in the weak$^*$-topology on $E^*$.  Then, for $\hat x\in M(F)$
and $a\in\mc A$,
\[ \ip{\hat x\cdot a}{t} = \lim_\alpha \ip{\hat x\cdot(e_\alpha a)}{t}
= \lim_\alpha \ip{\hat x\cdot e_\alpha}{a\cdot t}
= \ip{\phi_F(\hat x)\cdot a}{t} \qquad (t\in E), \]
showing that $\hat x\cdot a = \phi_F(\hat x)\cdot a$.  Similarly,
\[ \ip{a\cdot\hat x}{b\cdot t} = \ip{\hat x\cdot b}{t\cdot a}
= \ip{\phi_F(\hat x)\cdot b}{t\cdot a}
= \ip{a\cdot\phi_F(\hat x)}{b\cdot t} \qquad (t\in E, b\in\mc A), \]
which shows that $a\cdot\hat x = a\cdot\phi_F(\hat x)$, using that $E$ is essential.

Let $\mu\in F$, so for $a\in\mc A$ and $t\in E$,
\[ \ip{\phi_F(\mu)}{a\cdot t} = \ip{\phi_F(\mu)\cdot a}{t}
= \ip{\mu\cdot a}{t} = \ip{\mu}{a\cdot t}, \]
so as $E$ is essential, $\phi_F(\mu) = \mu$ for $\mu\in F$.
Similarly, for $\hat x\in M(F)$ and $a,b\in\mc A$,
\[ \phi_F(a\cdot\hat x)\cdot b = (a\cdot\hat x)\cdot b = a\cdot(\hat x\cdot b)
= (a \cdot \phi_F(\hat x))\cdot b, \]
which shows that $\phi_F(a\cdot\hat x) = a\cdot\phi_F(\hat x)$.  Similarly,
$\phi_F(\hat x\cdot a) = \phi_F(\hat x)\cdot a$, so that $\phi_F$ is an $\mc A$-bimodule
homomorphism.

If $\phi:M(F)\rightarrow E^*$ is another extension of the inclusion $F\rightarrow E^*$
which is an $\mc A$-bimodule homomorphism, then for $\hat x\in M(F)$,
\[ \ip{\phi(\hat x)}{a\cdot t} = \ip{\phi(\hat x\cdot a)}{t} = \ip{\hat x\cdot a}{t}
= \ip{\phi_F(\hat x)}{a\cdot t} \qquad (a\in\mc A,t\in E), \]
so that $\phi(\hat x) = \phi_F(\hat x)$.  Hence $\phi_F$ is unique.

Let $\hat x_\alpha\rightarrow\hat x$ strictly in $M(F)$, so that
\[ \lim_\alpha \ip{\phi_F(\hat x_\alpha)}{a\cdot t} = \lim_\alpha \ip{\hat x_\alpha\cdot a}{t}
= \ip{\hat x\cdot a}{t} = \ip{\hat x}{a\cdot t} \qquad (a\in\mc A,t\in E), \]
and similarly for $t\cdot a$, showing that $\phi_F(\hat x_\alpha)\rightarrow \phi_F(\hat x)$
weak$^*$ in $E^*$, again, as $E$ is essential.
\end{proof}

Finally, we apply these ideas to extend module homomorphisms which map into dual modules.

\begin{proposition}\label{wsclosed_ext}
With the same hypotheses, use Theorem~\ref{exten_thm} to turn $E$, and hence also $E^*$,
into an $M(\mc A)$-bimodule.  If $F$ is weak$^*$-closed, then $F$ is an $M(\mc A)$-submodule,
and $\phi_F$ is an $M(\mc A)$-bimodule homomorphism.
\end{proposition}
\begin{proof}
Let $\hat a\in M(\mc A)$ and $\mu\in F$, and let $\lambda$ be a weak$^*$-limit point of
$(\mu\cdot \hat a e_\alpha)$.  A typical member of $E$ is $a\cdot t$ for some $a\in\mc A$
and $t\in E$.  Then
\[ \ip{\lambda}{a\cdot t} = \lim_\alpha \ip{\mu}{\hat a e_\alpha a\cdot t}
= \ip{\mu}{\hat a a\cdot t} = \ip{\mu\cdot\hat a}{a\cdot t}, \]
which shows that $\lambda = \mu\cdot\hat a$.  As $E$ is weak$^*$-closed, it follows that
$\mu\cdot\hat a\in E$; similarly, $\hat a\cdot\mu\in E$.

It is now easy to show that $\phi_F$ is an $M(\mc A)$-bimodule homomorphism, using
property (1) established above in Proposition~\ref{dba_dual_mod}.
For example, for $\hat a\in M(\mc A), \hat x\in M(F)$ and $a\in\mc A$,
\[ \phi_F(\hat x\cdot\hat a)\cdot a = (\hat x\cdot\hat a)\cdot a
= \hat x\cdot \hat a a = \phi_F(\hat x)\cdot \hat a a, \]
and similarly $a\cdot\phi_F(\hat x\cdot\hat a) = a\cdot\phi_F(\hat x)\cdot\hat a$,
so that $\phi_F(\hat x\cdot\hat a) = \phi_F(\hat x)\cdot\hat a$.
\end{proof}

\begin{theorem}\label{dual_mod_extension}
Let $\mc A$ be a Banach algebra with a bounded approximate identity of bound $K>0$,
and let $E$ and $F$ be $\mc A$-bimodules, with one of $E$ or $F$ being essential.
An $\mc A$-bimodule homomorphism $\psi:E\rightarrow F^*$ has an
extension $\tilde\psi:M(E)\rightarrow F^*$ such that:
\begin{enumerate}
\item if $F$ is essential, then $\tilde\psi$ is uniquely defined,
  strictly-weak$^*$-continuous, and satisfies $\|\tilde\psi\| \leq K\|\psi\|$;
\item if both $E$ and $F$ are essential, then $\tilde\psi$ is also an
  $M(\mc A)$-bimodule homomorphism.
\end{enumerate}
\end{theorem}
\begin{proof}
Suppose first that $F$ is essential.  By Theorem~\ref{always_ext_mod_homos},
there is a strictly-continuous $\psi_0:M(E)\rightarrow M(F^*)$ which extends $\psi$.
Then consider the map $\phi_{F^*}:M(F^*)\rightarrow F^*$ constructed by
Proposition~\ref{dba_dual_mod}.  Let $\tilde\psi = \phi_{F^*} \psi_0$;
the estimate $\|\tilde\psi\| \leq K\|\psi\|$ follows easily from the proof of
Proposition~\ref{dba_dual_mod}.  For $x\in E$, as $\psi(x)\in F^*$, we have
that $\tilde\psi(x) = \phi_{F^*} \psi_0(x) = \phi_{F^*}\psi(x) = \psi(x)$, so
that $\tilde\psi$ is an extension.  As $\phi_{F^*}$ is strictly-weak$^*$
continuous, it follows that $\tilde\psi$ is as well.

If $\phi:E\rightarrow F^*$ is another extension, then for $\hat x\in M(E)$
and $a\in\mc A$, we have that $a\cdot\phi(\hat x) = \phi(a\cdot\hat x)
= \psi(a\cdot\hat x) = a\cdot\tilde\psi(\hat x)$.  As $F$ is essential, this
is enough to show that $\phi(\hat x) = \tilde\psi(\hat x)$, so that
$\tilde\psi$ is unique.

In the case when both $E$ and $F$ are essential, we can turn $E$, $F$ and hence
$F^*$ into $M(\mc A)$-bimodule, by Theorem~\ref{exten_thm}.
By Proposition~\ref{wsclosed_ext}, $\phi_{F^*}$ is an $M(\mc A)$-bimodule
homomorphism, as is $\psi_0$, and hence also $\tilde\psi$.

If $F$ is not essential, but $E$ is essential, then we can adapt a technique which goes
back to Johnson, see \cite[Proposition~1.8]{johnmem}.  Let $F_0 = \mc A\cdot F\cdot\mc A$,
which by Theorem~\ref{cohen}, is a closed essential submodule of $F$.  
Define a map $\iota:F_0^*\rightarrow F^*$ by
\[ \ip{\iota(\mu)}{x} = \lim_\alpha \ip{\mu}{e_\alpha\cdot x\cdot e_\alpha}
\qquad (\mu\in F_0^*, x\in F). \]
Let $q:F_*\rightarrow F_0^*$ be the restriction map.  Then $q\iota$ is the
identity, and $\iota q$ is a projection.
By the previous result, we can extend $q\psi$ to a map $\psi_0:M(E) \rightarrow
F_0^*$.  Let $\tilde\psi = \iota \psi_0:M(E)\rightarrow F^*$.
If $E$ is essential, then a typical element of $E$ is of the form $x=a\cdot x\cdot b$
for some $y\in E$ and $a,b\in\mc A$.  Then, for $t\in F$,
\begin{align*} \ip{\tilde\psi(x)}{t} &= \lim_\alpha \ip{\psi_0(x)}{e_\alpha\cdot t\cdot e_\alpha}
= \lim_\alpha \ip{\psi(x)}{e_\alpha\cdot t\cdot e_\alpha}
= \lim_\alpha \ip{a\cdot\psi(y)\cdot b}{e_\alpha\cdot t\cdot e_\alpha} \\
&= \ip{\psi(y)}{b\cdot t\cdot a}
= \ip{\psi(x)}{t}, \end{align*}
so that $\tilde\psi$ does extend $\psi$.  However, it is not now clear that
$\tilde\psi$ is uniquely defined or strictly-weak$^*$-continuous.
\end{proof}

\section{Dual Banach algebras}\label{dba_section}

Following \cite{runde, daws}, we say that a Banach algebra $\mc A$ which is
the dual of a Banach space $\mc A_*$ is a \emph{dual Banach algebra} if
multiplication in $\mc A$ is separately weak$^*$-continuous.  This is equivalent
to the canonical image of $\mc A_*$ in $\mc A^*$ being an $\mc A$-submodule, that
is, that $\mc A$ is a dual $\mc A$-bimodule.  In \cite[Theorem~5.6]{is}, it is
shown that in the presence of a bounded approximate identity, we can always extends
homomorphisms which map into a dual Banach algebra.  We shall extend this result to
multipliers of modules, and also show how the result really follows from algebraic
considerations, and Theorem~\ref{exten_thm}.  Firstly, we shall explore connections
with weakly almost periodic functionals, which we shall return to when considering
when multiplier algebras are themselves dual, in Section~\ref{mult_dual_sec} below.

As in \cite{runde2}, given an $\mc A$-bimodule $E$, we shall write $\wap(E)$ for the
collection of elements $x\in E$ such that $R_x,L_x:\mc A\rightarrow E$ are weakly compact.
It is easy to see that $\wap(E)$ is a closed $\mc A$-submodule of $E$.
When $E=\mc A^*$, we recover the usual definition of $\wap(\mc A^*)$ (which some authors
write as $\wap(\mc A)$) as, for $\mu\in\mc A^*$, $R_\mu^*\kappa_{\mc A} = L_\mu$ and 
$L_\mu^*\kappa_{\mc A} = R_\mu$ so $R_\mu$ is weakly compact if and only if $L_\mu$ is.

\begin{lemma}\label{dba_lemma}
Let $\mc A$ be a Banach algebra with a bounded approximate identity, let $E$ be an
essential $\mc A$-bimodule, and let $F \subseteq \wap(E^*)$ be a closed
$\mc A$-submodule.  Then $F = \{ a \cdot\mu\cdot b: a,b\in\mc A, \mu\in F\}$;
in particular, $F$ is essential.
\end{lemma}
\begin{proof}
Let $(e_\alpha)$ be a bounded approximate identity for $\mc A$.  For $\mu\in F$,
by weak compactness, the net $(e_\alpha\cdot\mu)$ has a weakly convergent subnet, whose
limit must be $\mu$, as
\[ \lim_\alpha \ip{e_\alpha\cdot\mu}{x\cdot a} = \lim_\alpha \ip{\mu}{x\cdot ae_\alpha} =
\ip{\mu}{x\cdot a} \qquad (x\in E, a\in\mc A), \]
and using that $E$ is essential.  As the norm closure and weak closure of a convex
set agree, by taking convex combinations, it follows that
$\mu$ is in the norm closure of $\{ a\cdot\mu: a\in\mc A \}$.  By the Theorem~\ref{cohen},
it follows that $F = \{ a\cdot\lambda : a\in\mc A, \lambda\in F\}$.
Then repeat the argument on the other side.
\end{proof}

In particular, we can apply Theorem~\ref{submod_rep_thm} for any closed submodule
$F\subseteq\wap(E^*)$.

We now consider the case of homomorphisms; in this case, Lemma~\ref{dba_lemma} becomes
more powerful.  For a Banach algebra $\mc A$, let $F=\wap(A^*)$, let
$\kappa_w = q_F\kappa_{\mc A}:\mc A\rightarrow\wap(\mc A^*)^*$, and let
$\theta_w = \theta_F:M(\mc A)\rightarrow\wap(\mc A^*)^*$.  Now, $\theta_w =
q_F \theta$, and $\theta:M(\mc A)\rightarrow\mc A^{**}$ is a homomorphism for the
second Arens product.  As $q_F:\mc A^{**}\rightarrow\wap(A^*)^*$ is a homomorphism
for either Arens product, it follows that $\theta_w$ is a homomorphism.

\begin{theorem}\label{dba_bai_exten}
Let $\mc A$ be a Banach algebra with a bounded approximate identity of
bound $K>0$, and let $(\mc B,\mc B_*)$ be dual Banach algebra.  A homomorphism
$\psi:\mc A \rightarrow \mc B$ has a unique extension to a homomorphism
$\tilde\psi:M(\mc A)\rightarrow\mc B$ with $\|\tilde\psi\|\leq K\|\psi\|$,
and such that $\tilde\psi$ is strictly-weak$^*$-continuous.
\end{theorem}
\begin{proof}
By \cite[Theorem~4.10]{runde2} there exists a unique weak$^*$-continuous homomorphism
$\psi_0:\wap(\mc A^*)^* \rightarrow \mc B$ which extends $\psi$ in the sense that
$\psi_0\kappa_w = \psi$, and with $\|\psi_0\| \leq \|\psi\|$.  Indeed, to show
this, observe that $\psi^*(\mc B_*) \subseteq \wap(\mc A^*)$, so we may define
$\psi_0 = (\psi^*|_{\mc B*})^*$.

Then let $\tilde\psi = \psi_0 \theta_w$, so that for $a\in\mc A$, we have that $\tilde\psi(a)
= \psi_0 \kappa_w(a) = \psi(a)$.  Then $\tilde\psi$ is strictly-weak$^*$-continuous by
Theorem~\ref{submod_rep_thm}, and as $\|\theta_w\|\leq K$, it follows that
$\|\tilde\psi\| \leq K\|\psi\|$.  Uniqueness follows as $\mc A$ is strictly
dense in $M(\mc A)$.
\end{proof}

In \cite[Theorem~5.6]{is} this result is proved, using a completely different
method, but \emph{a priori} with two differences:
\begin{itemize}
\item The modification that $M(\mc A)$ is given the \emph{right multiplier
topology}, determined by the seminorms $(L,R)\mapsto \|R(a)\|$ for $a\in\mc A$.
However, if we examine the proof of Theorem~\ref{submod_rep_thm}, then
we actually only used the right multiplier topology.
\item The extension is defined from $\mc A_0$, a Banach algebra which contains
$\mc A$ as a closed ideal.  Then we have a natural contraction
$\mc A_0 \rightarrow M(\mc A)$, so really, it is enough to work with $M(\mc A)$.
\end{itemize}

We explore below, in Proposition~\ref{si_dba_ext_new} and the remark thereafter,
a more algebraic way to prove this result.

We shall see in Section~\ref{mult_dual_sec} that often $M(\mc B)$ is a dual Banach algebra.
Thus, given \emph{any} homomorphism $\psi:\mc A\rightarrow\mc B$, we can consider
$\psi$ as a homomorphism $\mc A\rightarrow M(\mc B)$, and hence use the above
theorem to find an extension $\tilde\psi:M(\mc A)\rightarrow M(\mc B)$.  This
hence gives a stronger extension result than that given by Proposition~\ref{homo_to_mod}
and Theorem~\ref{exten_thm} (but only gives
a strictly-weak$^*$-continuous extension, not a strictly-strictly-continuous
extension).

\section{Self-induced Banach algebras}\label{selfind}

We have seen that having a bounded approximate identity allows us to perform most
of the operations which we might wish, as regards multipliers.  A larger class of
algebras with which multipliers interact well is the class of \emph{self-induced}
algebras, which we shall explore in this section.  The theory becomes very algebraic,
and indeed, \emph{most} of what we saw in previous sections could have been proved
by observing that any algebra with a bounded approximate identity is
self-induced (see Proposition~\ref{bai_si} and references, below).  However, this
would have been unconventional, and we achieved greater generality by waiting as
long as possible before exploring how we might extend module actions to algebras
of multipliers.

Let $\mc A\proten\mc A$ be the projective tensor product of $\mc A$ with itself,
and let $N$ be the closed linear span of elements of the form
$ab\otimes c - a\otimes bc$, for $a,b,c\in\mc A$.  Then we define
$\mc A\proten_{\mc A} \mc A := \mc A\proten\mc A / N$.
Let $\pi:\mc A\proten\mc A\rightarrow\mc A$ be the product map, $\pi(a\otimes b)=ab$.
Then clearly $N \subseteq \ker\pi$, so $\pi$ induces a map $\mc A\proten_{\mc A}\mc A
\rightarrow\mc A$.  If this map is an isomorphism, that is, $\ker\pi = N$,
then $\mc A$ is said to be \emph{self-induced}.

This idea was explored by Gronbaek in \cite{gronbaek} in the context of Morita
equivalence, although the idea goes back at least to work of Rieffel in \cite{rieffel}.
Similar ideas have also been explored in the context of Banach cohomology theory, see
for example \cite{sel}.  We shall argue that self-induced algebras form a
natural setting to consider multipliers in.  We shall prove some general results,
as the proofs will later be useful when we consider completely contractive Banach algebras.

Given a Banach algebra $\mc A$, let $\module-\mc A$ be the class of right
$\mc A$-modules.  We similarly define $\mc A-\module$ and $\mc A-\module-\mc A$.
Given another Banach algebra $\mc B$, let $\mc A-\module-\mc B$ be the class
of left $\mc A$-modules which are also right $\mc B$-module, with commuting actions.

For $E\in\module-\mc A$ and $F\in\mc A-\module$, we let $E\proten_{\mc A}F =
E\proten F / N$ where $N$ is the closed linear span of elements of the form
$x\cdot a \otimes y - x\otimes a\cdot y$ for $x\in E, y\in F$ and $a\in\mc A$.
If $E\in\mc A-\module-\mc A$ and $F\in\mc A-\module-\mc B$, then it is easy to see
that $E\proten_{\mc A}F \in \mc A-\module-\mc B$.

For Banach spaces $E$ and $F$, we identify $(E\proten F)^*$ with $\mc B(E,F^*)$.
Then $(E\proten_{\mc A}F)^* = N^\perp$ where $N^\perp = \{ T\in(E\proten F)^* : \ip{T}{\tau}=0
\ (\tau\in N) \}$.  It is easy to see that in our case, $N^\perp = \mc B_{\mc A}(E,F^*)$,
the space of right $\mc A$-module homomorphisms.

The following is \cite[Theorem~3.19]{rieffel}, but we give a proof as we shall
wish to generalise this (to operator spaces) later.

\begin{lemma}\label{assoc}
Let $\mc A$ and $\mc B$ be Banach algebras.  Let $E\in\module-\mc A$,
$F\in\mc A-\module-\mc B$ and $G\in\mc B-\module$.  The identity map 
$E\otimes F\otimes G \rightarrow E\otimes F\otimes G$ induces an isometric
isomorphism $(E\proten_{\mc A}F)\proten_{\mc B}G \cong 
E\proten_{\mc A}(F\proten_{\mc B}G)$.
\end{lemma}
\begin{proof}
We shall first show that there is a natural isomorphism
\[ \alpha: \mc B_{\mc B}(E\proten_{\mc A}F, G^*) \cong \mc B_{\mc A}\big(E,(F\proten_{\mc B}G)^*\big)
\cong \mc B_{\mc A}\big(E,\mc B_{\mc B}(F,G^*)\big), \]
the main claim then following by duality.  As $F\proten_{\mc B}G\in\mc A-\module$,
by duality, $\mc B_{\mc B}(F,G^*)=(F\proten_{\mc B}G)^*\in\module-\mc A$.
To be explicit, the module action is
\[ (S\cdot a)(y) = S(a\cdot y)
\qquad (S\in\mc B_{\mc B}(F,G^*), a\in\mc A, y\in F). \]

For $T\in\mc B_{\mc B}(E\proten_{\mc A}F, G^*)$, define
\[ \alpha(T)\in\mc B_{\mc A}\big(E,\mc B_{\mc B}(F,G^*)\big),
\quad \alpha(T)(x)(y) = T(x\otimes y)
\qquad (x\in E, y\in F). \]
Then for fixed $T$ and $x$, clearly $\alpha(T)(x) \in \mc B(F,G^*)$
with $\|\alpha(T)(x)\| \leq \|T\| \|x\|$.  For
$y\in F, z\in G$ and $b\in\mc B$, we have that
\[ \ip{\alpha(T)(x)(y\cdot b)}{z} = \ip{T(x\otimes y\cdot b)}{z}
= \ip{T(x\otimes y)\cdot b}{z} = \ip{\alpha(T)(x)(y) \cdot b}{z}, \]
as $T$ is a right $\mc B$-module homomorphism.  Thus $\alpha(T)(x) \in \mc B_{\mc B}(F,G^*)$.
Then obviously $x\mapsto \alpha(T)(x)$ is linear and bounded.  For
$y\otimes z\in F\proten_{\mc B} G$, we see that
\begin{align*} \ip{\alpha(T)(x\cdot a)}{y\otimes z} &= \ip{T(x\cdot a\otimes y)}{z}
= \ip{T(x\otimes a\cdot y)}{z} \\ &= \ip{\alpha(T)(x)}{a\cdot y\otimes z}
= \ip{\alpha(T)(x)\cdot a}{y\otimes z}, \end{align*}
as $x\cdot a\otimes y = x\otimes a\cdot y$ in $E\proten_{\mc A}F$.
Hence $\alpha(T)$ is a right $\mc A$-module homomorphism, as claimed.  Finally,
similar arguments show that $\alpha$ is indeed an isometric isomorphism.

The adjoint of $\alpha$ induces an isometric isomorphism
\[ \big( (E\proten_{\mc A}F)\proten_{\mc B}G \big)^{**} \cong 
\big( E\proten_{\mc A}(F\proten_{\mc B}G) \big)^{**}. \]
Let $\kappa$ be the canonical map from $(E\proten_{\mc A}F)\proten_{\mc B}G$
to its bidual, and similarly let $\iota$ be the canonical map
from $E\proten_{\mc A}(F\proten_{\mc B}G)$ to its bidual.  Then it is easy
to see that
\[ \alpha^* \kappa\big( (x\otimes y)\otimes z \big) = \iota\big( x\otimes (y\otimes z) \big)
\qquad (x\in E, y\in F, z\in G). \]
By continuity, $\alpha^*\kappa$ takes values in the image of $\iota$, and
$(\alpha^{-1})^*\iota$ takes values in $\kappa$, from which the claim
immediately follows.
\end{proof}

As in the previous section, we now consider the problem of extending homomorphisms
by way of modules.

\begin{proposition}\label{si_mod_ext}
Let $E\in\mc A-\module-\mc A$, so that $E_0 = \mc A\proten_{\mc A}E\proten_{\mc A}\mc A$
is also an $\mc A$-bimodule.  Then $E_0$ is an $M(\mc A)$-bimodule, with the module
actions extending those of $\mc A$.  If $E$ is essential over itself, these extensions
are unique.
\end{proposition}
\begin{proof}
For $(L,R)\in M(\mc A)$, $a_1,a_2\in\mc A$ and $x\in E$, we define
\[ (L,R)\cdot (a_1\otimes x\otimes a_2) = L(a_1)\otimes x\otimes a_2, \quad
(a_1\otimes b\otimes a_2)\cdot(L,R) = a_1\otimes b\otimes R(a_2). \]
As $L$ is a right module homomorphism, and $R$ is a left module homomorphism,
it follows that these actions respect the quotient map $\mc A\proten\mc B\proten\mc A
\rightarrow E_0$, and simple checks show that these are bimodule actions.

If $\mc A$ is essential over itself, then $E_0$ is essential, and so uniqueness
follows by Theorem~\ref{mult_mod_is_mult_mod}.
\end{proof}

Given $E\in\mc A-\module$, if the product map induces an isomorphism $\mc A\proten_{\mc A} E
\cong E$, then we say that $E$ is \emph{induced}.  Similar remarks apply to right
modules and bimodules.  Hence, if $E\in\mc A-\module-\mc A$ is induced, then
we can always extend the module actions to $M(\mc A)$.
This allows us to immediately reprove Theorem~\ref{exten_thm}, given the following,
which was first shown by Rieffel in \cite[Theorem~4.4]{rieffel} (again, we give a
different proof, exploiting duality, as we wish to generalise this later).

\begin{proposition}\label{bai_si}
Let $\mc A$ be a Banach algebra with a bounded approximate identity, and let
$E\in\mc A-\module$ be essential.  Then $E$ is induced.  Similar remarks apply to
right- and bi-modules.
\end{proposition}
\begin{proof}
Let $\pi_E:\mc A\proten E\rightarrow E; a\otimes x\mapsto a\cdot x$ be
the product map.  We shall show that $\ker\pi = N$, where $N$ is the closed
linear span of elements of the form $aa'\otimes x - a\otimes a'\cdot x$, for
$a,a'\in\mc A$ and $x\in E$.  As $E$ is essential, $\pi_E$ has dense range,
so it is enough to show that $\pi_E^*:E^*\rightarrow(\mc A\proten E)^* =
\mc B_{\mc A}(\mc A,E^*)$ surjects.

Let $T\in\mc B(\mc A,E^*)$ be a right $\mc A$-module map, and let
$(e_\alpha)$ is a bounded approximate identity for $\mc A$.  By moving to a subnet
we may suppose that $T(e_\alpha)$ converges weak$^*$ to $\mu\in E^*$.  Then,
for $a\in\mc A$ and $x\in E$,
\[ \ip{T(a)}{x} = \lim_\alpha \ip{T(e_\alpha a)}{x}
= \lim_\alpha \ip{T(e_\alpha) \cdot a}{x}
= \ip{\mu}{a\cdot x} = \ip{\pi^*(\mu)(a)}{x}. \]
Thus $T = \pi^*(\mu)$, and we are done.

The argument on the right follows similarly, and the bimodule case
follows by using Lemma~\ref{assoc}.
\end{proof}

Let $\mc A$ be self-induced, let $\mc B$ be a Banach algebra, and let
$\theta:\mc A\rightarrow\mc B$ be a homomorphism.
Let $\pi_{\mc A}:\mc A\proten_{\mc A}\mc A\rightarrow\mc A$ and
$\pi_{\mc B}:\mc A\proten_{\mc A}\mc B\rightarrow\mc B$ be the product maps.
Then \cite[Proposition~2.13]{gronbaek} tells us that
${}_{\mc A}\mc B(\mc A,\mc B) \cong {}_{\mc A}\mc B(\mc A,\mc A\proten_{\mc A}\mc B)$;
in particular, $\theta$ induces a unique map $\hat\theta:\mc A\rightarrow
\mc A\proten_{\mc A}\mc B$ such that $\pi_{\mc B} \hat\theta = \theta$.
Indeed, we have that $\hat\theta = (\id\otimes\theta)\pi_{\mc A}^{-1}$.  Similarly, we
can work on the right, and so find a homomorphism $\theta_0:\mc A\rightarrow
\mc A \proten_{\mc A}\mc B \proten_{\mc A} \mc A$.

\begin{lemma}
$\mc A\proten_{\mc A}\mc B$ becomes a Banach algebra for the product
\[ (a\otimes b)(a'\otimes b') = a \otimes b\theta(a')b' \qquad (a,a'\in\mc A, b,b'\in\mc B). \]
\end{lemma}
\begin{proof}
The claimed product can be written as
\[ \sigma\tau = \sigma \cdot \pi_{\mc B}(\tau)
\qquad (\sigma, \tau \in \mc A\proten_{\mc A}\mc B). \]
It is hence clear that this is a well-defined, contractive bilinear map.
Notice that $\pi_{\mc B}(\sigma\cdot\pi_{\mc B}(\tau)) = \pi_{\mc B}(\sigma)\pi_{\mc B}(\tau)$ for 
$\sigma, \tau \in \mc A\proten_{\mc A}\mc B$.  Hence the product is associative, as
\[ \omega(\sigma\tau) = \omega\cdot\pi_{\mc B}(\sigma\cdot\pi_{\mc B}(\tau))
= \omega\cdot\pi_{\mc B}(\sigma)\pi_{\mc B}(\tau) = (\omega\sigma)\tau
\qquad (\sigma, \tau, \omega \in \mc A\proten_{\mc A}\mc B). \]
\end{proof}

Similarly, $\mc B\proten_{\mc A}\mc A$ becomes a Banach algebra for the product defined by
$(b\otimes a)(b'\otimes a') = b\theta(a)b'\otimes a'$.  Combining these observations,
we see that $\mc A\proten_{\mc A}\mc B\proten_{\mc A}\mc A$ becomes an algebra for the product
\[ (a\otimes b\otimes c)(a'\otimes b'\otimes c')
= a \otimes b\theta(ca')b' \otimes c' \qquad
(a,a',c,c'\in\mc A, b,b'\in\mc B). \]
By Lemma~\ref{assoc}, it is easy to see that
$\mc A \proten_{\mc A}\mc B \proten_{\mc A} \mc A$ is induced as an $\mc A$-bimodule.

\begin{proposition}\label{si_ext_homo}
Let $\mc A$ be a self-induced Banach algebra, and let $\mc B$ be Banach algebra.
Let $\theta:\mc A\rightarrow\mc B$ be a homomorphism, and use this to induce a
Banach algebra structure on
$\mc C = \mc A \proten_{\mc A} \mc B \proten_{\mc A} \mc A$.
There is a unique extension of
$\theta_0:\mc A \rightarrow \mc C$ to $\tilde\theta_0:M(\mc A)\rightarrow M(\mc C)$
with $\|\tilde\theta_0\| \leq \|\theta_0\|$.
\end{proposition}
\begin{proof}
Let $x=(L',R')\in M(\mc A)$ and define $L\in\mc B(\mc C)$ by $L(c) = x\cdot c$ for $c\in\mc C$.
Then, for $a_i\in\mc A$ for $1\leq i\leq 4$ and $b_1,b_2\in\mc B$, with
reference to the proof of Proposition~\ref{si_mod_ext} above, we have
\begin{align*} L\big( & (a_1 \otimes b_1 \otimes a_2)(a_3\otimes b_2\otimes a_4)\big) =
L( a_1 \otimes b_1\theta(a_2a_3)b_2 \otimes a_4) \\
&= L'(a_1) \otimes b_1\theta(a_2a_3)b_2 \otimes a_4
= L\big( a_1 \otimes b_1 \otimes a_2\big) (a_3\otimes b_2\otimes a_4). \end{align*}
So $L\in M_l(\mc C)$.  Similarly, we define $R\in M_r(\mc C)$ by $R(c)=c\cdot x$
for $c\in\mc C$.  Then $(L,R)\in M(\mc C)$, and so we have defined
$\tilde\theta_0:M(\mc A)\rightarrow M(\mc C); x\mapsto (L,R)$ as required.
Uniqueness follows as the linear span of elements of the form
$a\cdot c$, for $a\in\mc A,c\in\mc C$, are dense in $\mc C$, as $\mc A$ is self-induced.
\end{proof}

\begin{corollary}
Let $\mc A$ be a self-induced Banach algebra, let $\mc B$ be a Banach algebra,
and let $\theta:\mc A\rightarrow\mc B$ be a homomorphism such that $\mc B$
becomes induced as an $\mc A$-bimodule.  There is a unique extension
$\tilde\theta:M(\mc A)\rightarrow M(\mc B)$.
\end{corollary}

\begin{example}
Let $\mc A = \ell^1(I)$ for some index set $I$, with the pointwise product.
It is easy to see that $M(\mc A) \cong \ell^\infty(I)$, again acting pointwise.
Furthermore, $\ell^1(I)$ is self-induced.  This follows, as $\ell^1(I)\proten\ell^1(I)
=\ell^1(I\times I)$, and then $N$ is the closure of the linear span of
\[ \big\{ \delta_i\delta_j\otimes\delta_k - \delta_i\otimes\delta_j\delta_k : i,j,k\in I \big\}. \]
This is the same as the closed linear span of $\{ \delta_i \otimes \delta_k : i\not=k \}$.
It now follows immediately that $\ell^1(I)$ is self-induced.

Now consider $\ell^1(\mathbb Z)$, and consider $A(\mathbb Z)$ the Fourier algebra
on $\mathbb Z$.  This is defined below, or as $\mathbb Z$ is abelian, we can consider
$A(\mathbb Z)$ to be the Fourier transform of $L^1(\mathbb T)$.  Then the formal identity
map gives a contractive homomorphism $\ell^1(\mathbb Z)\rightarrow A(\mathbb Z)$.
Suppose, towards a contradiction, that we can extend this to a continuous homomorphism
$\psi:\ell^\infty(\mathbb Z) = M(\ell^1(\mathbb Z)) \rightarrow M(A(\mathbb Z))
= B(\mathbb Z) \cong M(\mathbb T)$,
the measure algebra on $\mathbb T$, here identified with $B(\mathbb Z)$ the Fourier-Stieltjes
algebra, again by the Fourier transform.  Let $c_{00}(\mathbb Z)$ be the collection of
finitely supported functions, so $c_{00}(\mathbb Z)$ is a subalgebra of $\ell^1(\mathbb Z)$
and $B(\mathbb Z)$, and so $\psi(a) = a$ for $a\in c_{00}(\mathbb Z)$.  As $\psi$ is
continuous, it follows that $\psi(a)=a$ for $a\in c_0(\mathbb Z)$, implying that
$c_0(\mathbb Z) \subseteq B(\mathbb Z)$, a contradiction.

Thus, as we might expect, the canonical homomorphism $\ell^1(\mathbb Z) \rightarrow
A(\mathbb Z)$ is not inducing.

Notice also that $B(\mathbb Z)$ is a dual Banach algebra, and so there can be no
naive extension of Theorem~\ref{dba_bai_exten} to the self-induced case.

Finally, we note that a similar calculation to that in the first paragraph shows
that $\ell^1(\mathbb Z) \proten_{\ell^1(\mathbb Z)} A(\mathbb Z) \cong \ell^1(\mathbb Z)$.
\end{example}

To close this section, we consider a self-induced version of Theorem~\ref{dba_bai_exten}.

\begin{proposition}\label{si_dba_ext_new}
Let $\mc A$ be a self-induced Banach algebra, let $(\mc B,\mc B_*)$ be a dual Banach
algebra, and let $\theta:\mc A\rightarrow\mc B$ be a homomorphism.  Then $\mc B$,
and so $\mc B_*$, becomes an $\mc A$-bimodule.  If $\mc B_*$ is induced, then there
is a unique extension $\tilde\theta:M(\mc A)\rightarrow \mc B$ which is bounded and
strictly-weak$^*$-continuous.
\end{proposition}
\begin{proof}
Let $\pi:\mc A\proten_{\mc A}\mc B_*\rightarrow\mc B_*$ be the product map,
$\pi(a\otimes\mu) = \theta(a)\cdot\mu$, which by assumption is an isomorphism.
The adjoint is $\pi^*:\mc B\rightarrow\mc B_{\mc A}(\mc A,\mc B)$ where
$\mc B_{\mc A}(\mc A,\mc B)$ denotes the space of maps $T:\mc A\rightarrow\mc B$
with $T(aa') = T(a)\theta(a')$ for $a,a'\in\mc A$, and $\pi^*(b)$ is the map
$a\mapsto b\theta(a)$, for $a\in\mc A$ and $b\in\mc B$.

Let $(L,R)\in M(\mc A)$, and consider $\theta L:\mc A\rightarrow\mc B$, which is
a member of $\mc B_{\mc A}(\mc A,\mc B)$.  Let $b=(\pi^*)^{-1}(\theta L)$, so that
$b\theta(a) = \theta L(a)$ for $a\in\mc A$.  Analogously, we can work with the
product map $\mc B_* \proten_{\mc A} \mc A\rightarrow \mc B_*$, which leads to
$b'\in\mc B$ such that $\theta R(a) = \theta(a)b'$ for $a\in\mc A$.
For $a,a'\in\mc A$, we have that $aL(a')=R(a)a'$, and hence
\[ \ip{b-b'}{\theta(a')\cdot\mu\cdot\theta(a)}
= \ip{\theta(a)\theta(L(a')) - \theta(R(a))\theta(a')}{\mu} = 0 \qquad (\mu\in\mc B_*). \]
As $\mc B_*$ is induced, this shows that $b=b'$.  Notice that $\|b\| \leq \|\theta\|
\|\pi^{-1}\| \|L\|$.

It is now easy to verify that $\tilde\theta:M(\mc A)\rightarrow\mc B; (L,R)
\mapsto b$ is a bounded homomorphism which extends $\theta$.  Uniqueness follows as
if $\psi:M(\mc A)\rightarrow\mc B$ also extends $\theta$, then for $(L,R)\in M(\mc A)$,
$a\in\mc A$ and $\mu\in\mc B_*$,
\[ \ip{\psi(L,R)}{\theta(a)\cdot\mu} = \ip{\theta(L(a))}{\mu}
= \ip{\tilde\theta(L,R)\theta(a)}{\mu}
= \ip{\tilde\theta(L,R)}{\theta(a)\cdot\mu}. \]
As $\mc B_*$ is induced, it follows that $\psi = \theta$.
Finally, if $(L_\alpha,R_\alpha) \rightarrow (L,R)$ strictly in $M(\mc A)$, then,
again using that $\mc B_*$ is induced, it follows that
$\tilde\theta(L_\alpha,R_\alpha) \rightarrow \tilde\theta(L,R)$ weak$^*$ in $\mc B$.
\end{proof}

If $\mc A$ has a bounded approximate identity $(e_\alpha)$, and $\mc B_*$ is
essential, then $\mc B_*$ is induced, and so the above gives another proof of
Theorem~\ref{dba_bai_exten}.  If $\mc B_*$ is not induced, then we can use the
following ``cut-down'' technique.  Let $e\in\mc B$ be a weak$^*$-limit point of
$(\theta(e_\alpha))$, so that $e^2=e$.  Let $\mc C=e\mc B e$, a closed subalgebra
of $\mc B$ which contains the image of $\theta$.
It is easy to see that $\mu\in {}^\perp\mc C$ if and only if
$e\cdot \mu\cdot e=0$.  Let $b\in ({}^\perp\mc C)^\perp$, and let $\mu\in\mc B_*$.
Then $e\cdot(\mu-e\cdot\mu\cdot e)\cdot e=0$, so $0 = \ip{b}{\mu-e\cdot\mu\cdot e}
= \ip{b-ebe}{\mu}$.  Thus $b=ebe$, and we conclude that $\mc C$ is weak$^*$-closed.
Thus $\mc C$ is a dual Banach algebra with predual $\mc C_* = \mc B_* / {}^\perp\mc C$.

We now observe that $\mc C_*$ is essential as an $\mc A$-bimodule.  Indeed,
we have that ${}^\perp\mc C = \{ \mu-e\cdot\mu\cdot e : \mu\in\mc B_* \}$, and so
a typical element of $\mc C_*$ is $e\cdot \mu\cdot e + {}^\perp\mc C$.  Let
$\lambda=e\cdot\mu\cdot e + {}^\perp\mc C \in \mc C_*$, and let $c=ece\in\mc C$, so
\[ \lim_\alpha \ip{c}{\theta(e_\alpha)\cdot\lambda}
= \lim_\alpha \lim_\beta \ip{c\theta(e_\beta)\theta(e_\alpha)}{\lambda}
= \lim_\alpha \ip{c\theta(e_\alpha)}{\lambda}
= \ip{ce}{\lambda} = \ip{c}{\lambda}. \]
Thus $\theta(e_\alpha)\cdot\lambda \rightarrow \lambda$ weakly, and, by taking
convex combinations, also $\lambda$ is in the norm closure of $\mc A\cdot\lambda$.
Similarly, $\lambda$ is in the norm closure of $\lambda\cdot\mc A$, and so $\mc C_*$
is essential.

To finish, we apply the above proposition (or Theorem~\ref{dba_bai_exten})
to the map $\theta:\mc A\rightarrow\mc C$ to find an extension $\tilde\theta:
M(\mc A)\rightarrow\mc C \subseteq\mc B$, as required.

\section{Completely contractive Banach algebras}\label{sec::ccba}

In this section, we adapt the results of the previous sections to the setting
of operator spaces and completely contractive Banach algebras.  As explained
in the introduction, we shall exploit duality arguments (essentially, the operator
space version of the Hahn-Banach theorem) to avoid lengthly matrix level
calculations, where possible.  We refer the reader to \cite{ER} for details.

We shall overload notation, and write $\proten$ for the operator space projective
tensor product; we shall not consider the Banach space projective tensor product
of operator spaces!  We write $\mc{CB}(E,F)$ for the space of completely bounded
maps between operator spaces $E$ and $F$.  A \emph{completely contractive Banach
algebra} (CCBA) is a Banach algebra $\mc A$ which is an operator space such that the
product map extends to a complete contraction $\mc A\proten\mc A\rightarrow \mc A$.
Similarly, for example, a completely contractive left $\mc A$-module is an operator
space $E$ such that the product map induces a complete contraction
$\mc A\proten E\rightarrow E$.  We continue to write $E\in\mc A-\module$, and so forth.

One, important, caveat is that the Open Mapping Theorem has no analogue for
operator spaces, so it is possible for $T:E\rightarrow F$ be a completely bounded
bijection, but for $T^{-1}$ to only be bounded.  Fortunately, we shall see that
we can usually find an explicit estimate for the completely bounded norm of $T^{-1}$:
indeed, often $T$ might be a complete isometry, in which case there is no problem.

So, for example, if $\mc A$ is a completely contractive Banach algebra, then $\mc A$
is \emph{self-induced} if the product map induces an isomorphism $\mc A\proten_{\mc A}
\mc A\rightarrow\mc A$; it is not enough that the product map induce a bijection
$\mc A\proten_{\mc A}\mc A\rightarrow\mc A$.  However, see, for example, the proof
of Theorem~\ref{fourier_si} below.

For a CCBA $\mc A$, we write $M_{cb}(\mc A)$ for the subalgebra of $M(\mc A)$
consisting of those pairs $(L,R)$ with $L$ and $R$ completely bounded.  We give
$M_{cb}(\mc A)$ an operator space structure by embedding it in $\mc{CB}(\mc A)
\oplus_\infty \mc{CB}(\mc A)$ (see \cite[Section~2.6]{pisier}), so that
\[ \big\| (L,R) \big\|_n = \max\big( \|L\|_n, \|R\|_n \big)
\qquad (L,R\in\mathbb M_n(M_{cb}(\mc A)), n\geq1). \]

Everything in Section~2 translates to the completely bounded setting.  For example,
Theorem~\ref{mult_mod_is_mult_mod}
says that if $E$ is an $M_{cb}(\mc A)$-bimodule, then so is $M_{cb}(E)$.
For example, consider
\[ L_{\hat a\cdot\hat x}(a) = \hat a\cdot L_{\hat x}(a)
\qquad (a\in\mc A, \hat a\in M_{cb}(\mc A), \hat x\in M_{cb}(E)). \]
Assume that $E$ is a completely \emph{contractive} $M_{cb}(\mc A)$-bimodule, so the
product map $M_{cb}(\mc A)\proten E \rightarrow E$ is a complete contraction.
This is equivalent, see \cite[Proposition~7.1.2]{ER} to the map $\lambda:M_{cb}(\mc A)
\rightarrow \mc{CB}(E)$ being a complete contraction, where $\lambda(\hat a)(x) =
\hat a\cdot x$, for $\hat a\in M_{cb}(\mc A)$ and $x\in E$.  Then
$L_{\hat a\cdot\hat x} = \lambda(\hat a) L_{\hat x}$, and so $L_{\hat a\cdot\hat x}$
is completely bounded.  Furthermore, the resulting map $M_{cb}(\mc A)\proten M_{cb}(E)
\rightarrow M_{cb}(E)$ is
\[ \hat a\otimes\hat x \mapsto \lambda(\hat a) L_{\hat x}
\qquad (\hat a\in M_{cb}(\mc A), \hat x\in M_{cb}(E)), \]
and is hence clearly a complete contraction.  Similar remarks apply to
$R_{\hat a\cdot\hat x}$ and the definition of $\hat x\cdot\hat a$.

Similarly, the construction in Theorem~\ref{always_ext_mod_homos} is really
given by composition of various maps, and hence readily extends to the completely
bounded case.  Lemma~\ref{lemma::tensor} follows through if we work at the matrix
level.  All other results in Section~2 are really just algebra, carried out in
the approach category (that is, either bounded maps, or completely bounded maps).

We turn now to Section~\ref{sec::bai}.  From our presentation of the Arens products,
it is clear that when $\mc A$ is a CCBA, so is $\mc A^{**}$ for either $\aone$ or $\atwo$.
For example, the map $\otimes_{\aone}$ is the adjoint of $\beta\circ\alpha$, and
both $\beta$ and $\alpha$ are complete isometries, so $\otimes_{\aone}$ is
completely contractive.  Then $\aone$ is the composition with $\pi^{**}$,
showing that $\aone:\mc A^{**}\proten\mc A^{**}\rightarrow\mc A^{**}$ is a
complete contraction.  We now quickly show the completely bounded analogue of
Theorem~\ref{arens}.

\begin{theorem}\label{arens_op}
Let $\mc A$ be a CCBA with a bounded approximate identity $(e_\alpha)$, and
let $\Phi_0\in\mc A^{**}$ be a weak$^*$-accumulation point of $(e_\alpha)$.  Then:
\begin{enumerate}
\item $M_{cb}(\mc A)\subseteq\mc{CB}(\mc A)\times\mc{CB}(\mc A)$ is closed in the strict topology;
\item $\mc A$ is a closed ideal in $M_{cb}(\mc A)$ which is strictly dense;
\item $\theta:M_{cb}(\mc A)\rightarrow (\mc A^{**},\atwo)$, defined by
$(L,R)\mapsto L^{**}(\Phi_0)$, is an algebra homomorphism and a complete isomorphism onto its range,
with $\theta(a)=a$ for $a\in\mc A$.
\end{enumerate}
\end{theorem}
\begin{proof}
(1) follows by the arguments as used in Proposition~\ref{strict_closed}, and (2) is exactly as
in the bounded case.  Similarly, for (3), we can follow the bounded case to see
that $\theta$ is a homomorphism with $\theta(a)=a$ for $a\in\mc A$.  For any operator
spaces $E$ and $F$, given $x_0\in E$, the map $\mc{CB}(E,F)\rightarrow F; T\mapsto T(x_0)$
is easily seen to be completely bounded with bound at most $\|x_0\|$.
It follows that $\theta$ is completely bounded (even a complete contraction if
$(e_\alpha)$ is a contractive approximate identity).

Let $\pi_l:\mc A^{**}\rightarrow\mc{CB}(\mc A^{**})$ be the left-regular representation
for $\atwo$, so that $\pi_l$ is a complete contraction.  Then, as $\kappa_{\mc A}L(a) = 
\pi_l(L^{**}(\Phi_0))(\kappa_{\mc A}(a))$ for $a\in\mc A$, it follows that
\[ \|L(a)\|_{nm} \leq \|L^{**}(\Phi_0)\|_m \|a\|_n \qquad ((L,R)\in\mathbb M_m(M_{cb}(\mc A)),
a\in\mathbb M_n(\mc A), n,m\geq1). \]
Similarly, if $\pi_r$ denotes the right-regular representation, then as
$\kappa_{\mc A}R(a) = \pi_r(L^{**}(\Phi_0))(\kappa_{\mc A}(a))$, we have 
\[ \|R(a)\|_{nm} \leq \|L^{**}(\Phi_0)\|_m \|a\|_n \qquad ((L,R)\in\mathbb M_m(M_{cb}(\mc A)),
a\in\mathbb M_n(\mc A), n,m\geq1). \]
It follows that $\theta$ is a complete isomorphism onto its range, as claimed.
\end{proof}

A curious corollary of this observation is the following, shown in
\cite[Proposition~3.1]{KR} by another method.

\begin{theorem}\label{arens_ccba}
Let $\mc A$ be a CCBA with a contractive approximate identity.
Then $M(\mc A) = M_{cb}(\mc A)$ with equal norms.
If $\mc A$ only has a bounded approximate identity, then $M(\mc A)=M_{cb}(\mc A)$
with equivalent norms.
\end{theorem}
\begin{proof}
Clearly $M_{cb}(\mc A)$ contractively injects into $M(\mc A)$.
Conversely, let $\hat a=(L,R)\in M(\mc A)$, so by Theorem~\ref{arens}
we can find $\Phi\in\mc A^{**}$ with $\|\Phi\|\leq\|\hat a\|$ and such that
\[ \kappa_{\mc A}L(a) = \Phi\atwo\kappa_{\mc A}(a), \quad
\kappa_{\mc A}R(a) = \kappa_{\mc A}(a)\atwo\Phi \qquad (a\in\mc A). \]
As $\kappa_{\mc A}$ is a complete isometry, and $(\mc A^{**},\atwo)$ a CCBA,
it follows that $L$ and $R$ are completely bounded, with $\|L\|_{cb}
\leq \|\Phi\|$ and $\|R\|_{cb} \leq \|\Phi\|$.

The case when $\mc A$ only has a bounded approximate identity is similar.
\end{proof}

The proof of Theorem~\ref{exten_thm} could be translated to the completely
bounded setting, but instead we shall take a detour via the theory of
self-induced algebras.  

Theorem~\ref{submod_rep_thm} translates, except for condition (\ref{dbr:four}),
for which we need a stronger hypothesis; we use the same notation as before.

\begin{proposition}
Let $\mc A$ be a CCBA with a bounded approximate identity, and let $E$ and
$F$ be as in Theorem~\ref{submod_rep_thm}.  If $\theta_F$ is a complete isomorphism
onto its range (which is equivalent to $\iota_F$ being a complete isomorphism onto its
range) then the image of $\theta_F$ is the idealiser of $E$ in $F^*$.
\end{proposition}
\begin{proof}
The proof of Theorem~\ref{submod_rep_thm} shows that if $\Phi\in F^*$ idealises
$E$, then there exists $(L,R)\in M(E)$ with $\theta_F(L,R)=\Phi$.  Thus
\[ \Phi\cdot a = \iota_F L (a), \quad a\cdot\Phi = \iota_F R(a) \qquad (a\in\mc A). \]
The map $\mc A\rightarrow F^*, a\mapsto \Phi\cdot a$ is completely bounded, with
bounded at most $\|\Phi\|$.  As $\iota_F$ is a complete isomorphism onto its range,
it follows that $L$ is completely bounded.  Similarly $R$ is completely bounded.
\end{proof}

To translate the proof of Proposition~\ref{dba_dual_mod}
we simply proceed as in the proof of Theorem~\ref{arens_ccba} above.
The rest of Section~\ref{sec::bai} carries over without issue.  The same applies to
Section~\ref{dba_section}.

The results of Section~\ref{selfind} similarly translate without
issue, using standard results about operator spaces.
The typical idea which we exploit is illustrated by the proof of Lemma~\ref{assoc}.
Here we wish to show that $X$ and $Y$ are complete isometric (say), which
we do by finding a completely isometric isomorphism $\alpha:Y^*\rightarrow X^*$
such that $\alpha^*\kappa_X(X) \subseteq\kappa_Y(Y)$.  Then (the proof of)
Lemma~\ref{weakstar_lemma} shows that there is a completely isometric isomorphism
$\beta:X\rightarrow Y$ such that $\alpha=\beta^*$.  The ``matrix calculations''
are all hidden in the standard fact that $\kappa_X$ is a complete isometry,
and so forth.

An exception is Proposition~\ref{bai_si}, which we check does translate;
this was shown in \cite[Proposition~3.3]{fls}, but in the interests of
completeness, we give a proof here (which we think is shorter).

\begin{proposition}\label{cb_bai_si}
Let $\mc A$ be a completely contractive Banach algebra with a bounded approximate
identity, and let $E \in \mc A-\module-\mc A$ be essential.
Then $E$ is induced.
\end{proposition}
\begin{proof}
We follow the proof of Proposition~\ref{bai_si}.  In particular, we see that
$\pi^*:E^*\rightarrow\mc{CB}_{\mc A}(\mc A,E^*)$ surjects, the inverse
being given by $T\mapsto \mu$ where $\mu$ is the weak$^*$-limit of $T(e_\alpha)$.
We shall show that this inverse is completely bounded, which will show that
$\pi^*$ is a complete isomorphism, and hence also that $\pi$ is
(see \cite[Corollary~4.1.9]{ER}).

Let the bounded approximate identity for $\mc A$ have bound $K>0$, and let
$\Phi\in\mc A^{**}$ be a weak$^*$-limit of $(e_\alpha)$, so that $\|\Phi\|\leq K$.
For $T \in \mc{CB}_{\mc A}(\mc A,E^*)$, we have that the weak$^*$-limit of
$T(e_\alpha)$ is $\kappa_E^* T^{**}(\Phi)$.  Now, the map
\[ \mc{CB}_{\mc A}(\mc A,E^*) \rightarrow \mc{CB}(\mc A^{**},E^*); \quad
T \mapsto \kappa_E^* T^{**} \]
is a complete contraction, see \cite[Chapter~3]{ER}, and
\[ \mc{CB}(\mc A^{**},E^*)\rightarrow E^*; S \mapsto S(\Phi), \]
is completely bounded, with bound at most $\|\Phi\|\leq K$.  The
composition is then $(\pi^*)^{-1}$, so we are done.
\end{proof}

Finally, we come to our alternative proof of Theorem~\ref{exten_thm} for
CCBAs.  Indeed, we simply combine the previous result with
(the completely bounded version of) Proposition~\ref{si_mod_ext}.
Notice that if $\mc A$ only have a bounded approximate identity, then
$\mc A \proten_{\mc A} E \proten_{\mc A} \mc A$ will only be completely
isomorphic (and not isometric) to $E$, which explains why $E$ might only
become a completely bounded (not contractive) $M_{cb}(\mc A)$-bimodule.

\subsection{For the Fourier algebra}

Let $G$ be a locally compact group, and let $\lambda$ be the left-regular
representation of $G$ on $L^2(G)$, given by
\[ \lambda(s)\xi : t\mapsto \xi(s^{-1}t) \qquad (s,t\in G, \xi\in L^2(G)). \]
Let $VN(G)$ be the von Neumann algebra generated by the operators $\{\lambda(s):
s\in G\}$.  This carries a \emph{coproduct}, a normal unital $*$-homomorphism
$\Delta:VN(G) \rightarrow VN(G)\vnten VN(G)$ given by
\[ \Delta:\lambda(s) \mapsto \lambda(s)\otimes\lambda(s). \]
It is not immediately obvious that such a homomorphism exists, but if we define
a unitary $W$ on $L^2(G\times G)$ by $W\xi(s,t) = \xi(ts,t)$ for $\xi\in L^2(G\times G)$
and $s,t\in G$ then we see that
\[ \Delta(x) = W^*(1\otimes x)W \qquad (x\in VN(G)) \]
obviously defines a normal $*$-homomorphism which satisfies $\Delta\lambda(s)
=\lambda(s)\otimes\lambda(s)$ (and hence $\Delta$ does map into $VN(G)\vnten
VN(G)$).  Denote by $A(G)$ the predual of $VN(G)$.
Then the pre-adjoint of $\Delta$ defines a complete contraction
$\Delta_*: A(G) \proten A(G)\rightarrow A(G)$.  It is not hard to show that
this is an associative product, see \cite[Section~16.2]{ER} for further
details, for example.  The resulting commutative
algebra is the \emph{Fourier algebra} as defined and studied by Eymard in
\cite{eymard}.  See also \cite[Section~3, Chapter~VII]{tak2}.

We may identify
$A(G)$ with a (in general, not closed) subalgebra of $C_0(G)$, where $a\in A(G)$
is the function $s\mapsto \ip{\lambda(s)}{a}$ (note that this is a different
convention to that chosen in \cite{tt2}).

\begin{theorem}\label{fourier_si}
For any locally compact group $G$, the Fourier algebra $A(G)$ is self-induced
as a completely contractive Banach algebra.
\end{theorem}
\begin{proof}
As $\Delta$ is an injective $*$-homomorphism, it is a complete isometry,
and so by \cite[Corollary~4.1.9]{ER}, $\Delta_*$ is a complete quotient map, so
in particular, is surjective (Wood proves this in \cite{wood} using a more
complicated method).  If we can show that $\ker\Delta_* = N$ where $N$ is the
closed linear span of elements of the form $ab\otimes c - a\otimes bc$ for
$a,b,c\in A(G)$, then $\Delta_*$ will induce a surjective complete isometry
$A(G) \proten_{A(G)} A(G) \rightarrow A(G)$.  In particular, the inverse will
also be a complete isometry, and so $A(G)$ will be self-induced.

Following \cite[Section~4]{eymard}, we define the \emph{support} of $x\in VN(G)$
to be the collection of $s\in G$ with the property that if $a\cdot x = 0$ for some
$a\in A(G)$, then $a(s)=0$.  Let $D(G) = \{ (s,s) : s\in G \} \subseteq G\times G$.
Let $x\in VN(G)\vnten VN(G) = VN(G\times G)$ annihilate $N$.  We claim that the
support of $x$ is contained in $D(G)$.  Indeed, given $s\not=t$ in $G$, we
can find compact sets $K,L$ and open sets $U,V$ with $s\in K\subseteq U$ and
$t\in L\subseteq V$ and $U\cap V=\emptyset$.  By \cite[Lemma~3.2]{eymard},
there exists $a\in A(G)$ with $a(r)=1$ for $r\in K$, and $a(r)=0$ for $r\not\in U$.
Similarly, there exists $b\in A(G)$ with $b(r)=1$ for $r\in L$, and $b(r)=0$ for
$r\not\in V$.  Thus $ab=0$.  Then, for any $c,d\in A(G)$, we see that
\[ \ip{(a\otimes b)\cdot x}{c\otimes d} = \ip{x}{ca \otimes db}
= \ip{x}{c\otimes abd} = 0, \]
as $x$ annihilates $N$ and $A(G)$ is commutative.  Thus $(a\otimes b)\cdot x=0$,
so as $(a\otimes b)(s,t) = 1$, we conclude that $(s,t)$ is not in the support of $x$,
as required.

By \cite[Theorem~3]{tt2} it follows that such an $x$ is in the von Neumann algebra
generated by $\{\lambda(s,s) : s\in G\}$, that is, in $\Delta(VN(G))$.  So $x=\Delta(y)$
for some $y\in VN(G)$.  So the annihilator of $N$ is equal to the image of $\Delta$,
from which it follows that $N = \ker\Delta_*$ as required.
\end{proof}

This result is interesting, as $A(G)$ has a bounded approximate identity if and
only if $G$ is amenable \cite{leptin}.  It would be interesting to know if this
result holds for a general locally compact quantum group (see below).  Obviously
it holds for the group convolution algebra $L^1(G)$, as $L^1(G)$ always has a bounded
approximate identity.

Let us think further about the Fourier algebra.  In \cite{isp}, the completely bounded
homomorphisms between $A(G)$ and $A(H)$ are classified in terms of piecewise affine
maps, at least if $G$ is amenable.  However, there is no reason why all such maps
$A(G)\rightarrow A(H)$ should be non-degenerate, while \cite[Corollary~3.9]{isp}
easily implies that we do have an extension $M_{cb}(A(G))\rightarrow M_{cb}(A(H))$.
We can of course (following \cite{is}) apply the completely contractive version of
Theorem~\ref{dba_bai_exten}, as $M_{cb}(A(H))$ is a dual, completely contractive
Banach algebra (compare Section~\ref{cbmult_lcqg} below).

\section{When multiplier algebras are dual}\label{mult_dual_sec}

In this section, we provide a simple criterion for when $M(\mc A)$ is a dual Banach algebra.
Notice that for a C$^*$-algebra $\mc A$, it is relatively rare for $M(\mc A)$ to be dual
(that is, a von Neumann algebra).  However, multiplier algebras which appear in abstract
harmonic analysis do often seem to be dual spaces.  Our result allows us to show that,
in particular, $M(L^1(\G))$ (and its completely-bounded counterpart) are
dual Banach algebras, for a locally compact quantum group $\G$.  Our ideas are influenced
by \cite{sel}.

\begin{theorem}\label{mult_is_dba}
Let $\mc A$ be a Banach algebra such that $\{ ab : a,b\in\mc A\}$ is linearly dense in $\mc A$.
Let $(\mc B,\mc B_*)$ be a dual Banach algebra, let $\iota:\mc A \rightarrow \mc B$ be an isometric
homomorphism such that $\iota(\mc A)$ is an essential ideal in $\mc B$.  Suppose that the induced map
$\mc B\rightarrow M(\mc A)$ is injective.  Then there is a weak$^*$-topology on $M(\mc A)$
making $M(\mc A)$ a dual Banach algebra.
\end{theorem}
\begin{proof}
Given Banach spaces $E$ and $F$, let $E\oplus_1 F$ be the direct sum of $E$ and $F$ with the
norm $\|(x,y)\| = \|x\| + \|y\|$ for $x\in E$ and $y\in F$.  Then $(E\oplus_1 F)^* =
E^* \oplus_\infty F^*$, which has the norm $\|(\mu,\lambda)\| = \max(\|\mu\|,\|\lambda\|)$
for $\mu\in E^*$ and $\lambda\in F^*$.

Consider $(\mc A\proten B_*) \oplus_1 (\mc A\proten B_*)$ which has dual space
$\mc B(\mc A,\mc B) \oplus_\infty \mc B(\mc A,\mc B)$.  Consider
\[ X = \lin\{ (b\otimes\mu\cdot\iota(a))\oplus(-a\otimes \iota(b)\cdot\mu) :
a,b\in\mc A, \mu\in\mc B_* \} \subseteq A\proten B_* \oplus_1 A\proten B_*. \]
Then $X^\perp \subseteq \mc B(\mc A,\mc B) \oplus_\infty \mc B(\mc A,\mc B)$ is
a weak$^*$-closed subspace, and we calculate that $(T,S) \in X^\perp$ if and only if
\[ \ip{\iota(a) T(b)}{\mu} = \ip{S(a)\iota(b)}{\mu} \qquad (a,b\in\mc A, \mu\in\mc B_*), \]
that is, $\iota(a) T(b) = S(a)\iota(b)$ for $a,b\in\mc A$.  So, if $(T,S)\in X^\perp$,
then $\iota(a) T(bc) = S(a) \iota(bc) = \iota(a) T(b) \iota(c)$ for $a,b,c\in\mc A$.
As $\mc B$ injects into $M(\mc A)$, and as $\mc A$ is always assumed faithful, it
follows that $T(bc) = T(b)\iota(c)$ for $b,c\in\mc A$.
As products are dense in $\mc A$, and as $\iota(\mc A)$ is a closed ideal in $\mc B$,
it follows that $T(\mc A) \subseteq \iota(\mc A)$.  A similar argument shows that
$S(\mc A) \subseteq \iota(\mc A)$.  Consequently, there are $L,R\in\mc B(\mc A)$
with $T = \iota L$ and $S = \iota R$.  Then, for $a,b\in\mc A$, $\iota(a)T(b)
=\iota(a L(b)) = S(a)\iota(b) = \iota(R(a)b)$.  We conclude that $(L,R)\in M(\mc A)$.

We have thus shown that $M(\mc A)$ is isomorphic to $X^\perp$, and so $M(\mc A)$
is a dual Banach space.  The weak$^*$-topology is given by the embedding
$M(\mc A) \rightarrow (A\proten B_* \oplus_1 A\proten B_*)^*$ given by
\[ \ip{(L,R)}{(a\otimes\mu)\oplus(b\otimes\lambda)} =
\ip{\iota L(a)}{\mu} + \ip{\iota R(b)}{\lambda}, \]
for $L,R\in\mc B(\mc A)$, $a,b\in\mc A$ and $\mu,\lambda\in\mc B_*$.

Next, notice that the linear span of $\{ \mu\cdot\iota(a) : \mu\in B_*, a\in\mc A \}$
is dense in $\mc B_*$.  Indeed if $b\in\mc B$ is such that $\ip{b}{\mu\cdot\iota(a)}=0$
for $a\in\mc A$ and $\mu\in\mc B_*$, then $\iota(a)b=0$ for all $a\in\mc A$.  Again,
as $\iota(\mc A)$ is an ideal in $\mc B$, and $\mc A$ is faithful, it follows that
$b$ induces the zero multiplier on $\mc A$, and so by assumption, $b=0$.
Similarly, the linear span of $\{ \iota(a) \cdot\mu: \mu\in B_*, a\in\mc A \}$
is dense in $\mc B_*$

Suppose now that $(L_\alpha,R_\alpha)$ is a bounded net in $M(\mc A)$ converging
weak$^*$ to $(L,R)$.  Let $(L',R')\in M(\mc A)$, let $a,b,c\in\mc A$ and let
$\mu,\lambda\in\mc B_*$.  Then
\begin{align*} \lim_\alpha & \ip{(L_\alpha,R_\alpha)(L',R')}{(a\otimes\mu)\oplus(b\otimes\iota(c)\cdot\lambda)} \\
&= \lim_\alpha \ip{\iota(L_\alpha L(a))}{\mu} + \ip{\iota(R'R_\alpha(b))}{\iota(c)\cdot\lambda} \\
&= \lim_\alpha\ip{\iota(L_\alpha L(a))}{\mu} + \ip{\iota(R_\alpha(b) L'(c))}{\lambda} \\
&= \lim_\alpha \ip{(L_\alpha,R_\alpha)}{(L(a)\otimes\mu)\oplus(b\otimes L'(c)\cdot\lambda)} \\
&= \ip{(L,R)}{(L(a)\otimes\mu)\oplus(b\otimes L'(c)\cdot\lambda)} \\
&= \ip{(L,R)(L',R')}{(a\otimes\mu)\oplus(b\otimes\iota(c)\cdot\lambda)}.
\end{align*}
Thus, by the preecding paragraph, it follows that $(L_\alpha,R_\alpha)(L',R')
\rightarrow (L,R)(L',R')$ weak$^*$.  A similar argument establishes that
$(L',R')(L_\alpha,R_\alpha) \rightarrow (L',R')(L,R)$ weak$^*$.  We conclude that
$M(\mc A)$ is a dual Banach algebra for this weak$^*$ topology.
\end{proof}

If products are not linearly dense in $\mc A$, then, following \cite{sel}, one could
instead consider the unitisation of $\mc A$.  In our applications, this will not be needed.

We next show that, under a natural assumption, this weak$^*$ topology is unique.

\begin{theorem}\label{dual_mult_wstar_top}
Let $\mc A$ and $\mc B$ be as above, and let $\theta:\mc B\rightarrow M(\mc A)$ be the
induced map.  There is one and only one weak$^*$ topology on $M(\mc A)$ such that:
\begin{itemize}
\item $M(\mc A)$ is a dual Banach algebra;
\item for a bounded net $(b_\alpha)$ in $\mc B$ and $b\in\mc B$, we have that
  $b_\alpha\rightarrow b$ weak$^*$ in $\mc B$ if and only if
  $\theta(b_\alpha)\rightarrow\theta(b)$ weak$^*$ in $M(\mc A)$.
\end{itemize}
\end{theorem}
\begin{proof}
We first show that the previously constructed weak$^*$-topology on $M(\mc A)$ has this
property.  For $b\in\mc B$, write $\theta(b) = (L_b,R_b)\in M(\mc A)$.
If $b_\alpha\rightarrow b$ weak$^*$ in $\mc B$, then for $a,c\in\mc A$ and $\mu,\lambda\in\mc B_*$,
\begin{align*} \lim_\alpha &\ip{(L_{b_\alpha},R_{b_\alpha})}{(a\otimes\mu)\oplus(c\otimes\lambda)}
= \lim_\alpha \ip{b_\alpha\iota(a)}{\mu} + \ip{\iota(c)b_\alpha}{\lambda} \\
&\hspace{5ex}= \ip{b\iota(a)}{\mu} + \ip{\iota(c)b}{\lambda}
= \ip{(L_b,R_b)}{(a\otimes\mu)\oplus(c\otimes\lambda)}, \end{align*}
showing that $\theta(b_\alpha)\rightarrow\theta(b)$ weak$^*$.  Conversely, from the
previous proof, we know that elements of the form $\iota(a)\cdot\mu$ and $\lambda\cdot\iota(c)$
are linearly dense in $\mc B_*$.  Thus we can reverse the argument to show that if
$\theta(b_\alpha) \rightarrow \theta(b)$ weak$^*$, then $b_\alpha\rightarrow b$ weak$^*$.

Now let $\sigma$ be some other weak$^*$ topology on $M(\mc A)$ with the stated properties.
Notice that for $(L,R)\in M(\mc A)$ and $a\in\mc A$, we have that, by the calculations
of Lemma~\ref{idealiser},
\[ (L,R) \theta(\iota(a)) = (L L_{\iota(a)}, R_{\iota(a)} R)
= (L_{L(a)}, R_{L(a)}) = \theta(\iota(L(a))), \]
as $L_{\iota(a)} = L_a$ and so forth.  Let $(L_\alpha,R_\alpha)$ be a bounded net in
$M(\mc A)$ which converges in $\sigma$ to $(L,R)$.  Hence, for $a\in\mc A$,
\[ \lim_\alpha \theta(\iota(L_\alpha(a))) =
\lim_\alpha (L_\alpha, R_\alpha) \theta(\iota(a)) = 
(L,R) \theta(\iota(a)) = \theta(\iota(L(a))), \]
with respect to $\sigma$.  Thus $\iota(L_\alpha(a)) \rightarrow \iota(L(a))$ weak$^*$
in $\mc B$, for any $a\in\mc A$.  Thus, for $a,b,c\in\mc A$ and $\mu,\lambda\in\mc B_*$,
\begin{align*} \lim_\alpha \ip{(L_\alpha,R_\alpha)}{(a\otimes\mu)\oplus(b\otimes\iota(c)\cdot\lambda)}
&= \lim_\alpha \ip{\iota(L_\alpha(a))}{\mu} + \ip{\iota(R_\alpha(b))}{\iota(c)\cdot\lambda} \\
&= \lim_\alpha \ip{\iota(L_\alpha(a))}{\mu} + \ip{\iota(b L_\alpha(c))}{\lambda} \\
&= \ip{\iota(L(a))}{\mu} + \ip{\iota(b L(c))}{\lambda} \\
&= \ip{(L,R)}{(a\otimes\mu)\oplus(b\otimes\iota(c)\cdot\lambda)}.
\end{align*}
Again, this is enough to show that $(L_\alpha,R_\alpha) \rightarrow (L,R)$
in the weak$^*$ topology on $M(\mc A) = X^\perp$, as in the proof of the previous theorem.

Hence the map $(M(\mc A),\sigma) \rightarrow X^\perp$ is an (isometric) isomorphism
such that weak$^*$-convergent, bounded nets are sent to weak$^*$-convergent nets.
So by Lemma~\ref{weakstar_lemma}, the two weak$^*$-topologies agree, as required.
\end{proof}

In connection with the condition in the previous theorem, the next lemma is useful.

\begin{lemma}\label{weakstarlemma}
Let $(\mc B,\mc B_*)$ and $(\mc C, \mc C_*)$ be dual Banach algebras, and let
$\theta:\mc B\rightarrow\mc C$ be a bounded linear map.  The following properties
are equivalent:
\begin{enumerate}
\item For a bounded net $(b_\alpha)$ in $\mc B$ and $b\in\mc B$,
we have that $b_\alpha\rightarrow b$ weak$^*$ in $\mc B$ if
and only if $\theta(b_\alpha)\rightarrow\theta(b)$ weak$^*$ in $\mc C$;
\item $\theta$ is weak$^*$-continuous, and the preadjoint $\theta_*:\mc C_*
\rightarrow\mc B_*$ has dense range.
\end{enumerate}
\end{lemma}
\begin{proof}
Lemma~\ref{weakstar_lemma} shows that (1) implies that $\theta$ is weak$^*$-continuous.
Suppose that (1) holds, but that $\theta_*$ does not have dense range.  Then we can
find a non-zero $b\in\mc B$ with $\ip{b}{\theta_*(\mu)}=0$ for each $\mu\in\mc C_*$.
Thus $\theta(b)=0$.  Then the constant net $\theta(0)$ converges weak$^*$ in $\mc C$
to $\theta(b)$, so (1) shows that $b=0$, a contradiction.  Hence (1) implies (2).

If (2) holds, then clearly $b_\alpha\rightarrow b$ weak$^*$ implies that
$\theta(b_\alpha)\rightarrow\theta(b)$.  As $\theta_*$ has dense range, and
$(b_\alpha)$ is bounded, it follows that $\theta(b_\alpha)\rightarrow\theta(b)$
weak$^*$ implies that $b_\alpha\rightarrow b$, so that (1) holds.
\end{proof}

\begin{example}
Consider $\mc A=L^1(G)$ for a locally compact group $G$.  Then $M(\mc A) = M(G) = C_0(G)^*$.
It seems natural to let $\mc B=M(G)$ in the above, so that $\theta$ is just the identity
map, under the natural identifications.  The previous theorem then shows that our abstract
construction does construct the canonical weak$^*$ topology on $M(\mc A)$.
\end{example}

\begin{example}\label{fourier_example_one}
Consider now the Fourier algebra $A(G)$, for some locally compact group $G$.
As $A(G)$ is a regular Banach algebra of functions on $G$, it is not hard to show
that we can identify $M(A(G))$ with the collection of bounded continuous functions
$m:G\rightarrow\mathbb C$ such that $ma\in A(G)$ for any $a\in A(G)$.  See
\cite[Section~4.1]{spronk} or \cite{losert} for further details, or compare
with Example~\ref{dba_non_unital} below.

Let $C^*(G)$ be the full group C$^*$-algebra, whose dual is $B(G)$, the
Fourier-Stieljtes algebra.  By \cite[Proposition~3.4]{eymard}, the natural
map $A(G)\rightarrow B(G)$ is an isometry, and $A(G)$ is an ideal in $B(G)$.
Remember that by \cite{leptin}, we have that $M(A(G)) = B(G)$ only when $G$ is
amenable.  Nevertheless, we can apply our theorem to construct a weak$^*$-topology
on $M(A(G))$.  This follows as $A(G)$ separates the points of $G$, and so
the map $B(G)\rightarrow M(A(G))$ is injective.

In \cite{DcH}, it is shown that $M(A(G))$ is a dual Banach space, the predual being
$X$, which is the completion of $L^1(G)$ under the norm
\[ \|f\|_X = \sup\Big\{ \Big| \int_G f(s) m(s) \ ds \Big| : m\in M(A(G)),
\|m\|_{M(A(G))} \leq 1 \Big\}. \]
Given $m\in M(A(G))$, $m$ induces a member of $X^*$ by
\[ \ip{m}{f} = \int_G f(s) m(s) \ ds \qquad (f\in L^1(G)), \]
furthermore (and of course, this requires a proof) all of $X^*$ arises in this way.
It is then easy to see that $M(A(G))$ becomes a dual Banach algebra.

Let $\omega:L^1(G)\rightarrow C^*(G)$ be the universal representation.  Treating
$b\in B(G)=C^*(G)^*$ as a continuous function $G\rightarrow\mathbb C$, we have that
\[ \ip{b}{\omega(f)} = \int_G b(s) f(s) \ ds \qquad (f\in L^1(G)). \]
It follows easily that the weak$^*$ topology on $M(A(G))$ induced by $X$
satisfies the conditions of Theorem~\ref{dual_mult_wstar_top}, and so this
weak$^*$ topology agrees with that constructed by Theorem~\ref{mult_is_dba}.
\end{example}

Later, we shall extend this idea to locally compact quantum groups, as well as considering
operator space issues.

\subsection{For dual Banach algebras}

Suppose that we start with a dual Banach algebra $(\mc A,\mc A_*)$.  Firstly, we
show that $\mc A$ being faithful implies a number of useful properties.

\begin{proposition}
Let $(\mc A,\mc A_*)$ be a dual Banach algebra which is faithful.  Then:
\begin{enumerate}
\item Both $\{ a\cdot\mu : a\in\mc A,\mu\in\mc A_* \}$ and
$\{\mu\cdot a:a\in\mc A,\mu\in\mc A_*\}$ are linearly dense in $\mc A_*$;
\item Given $(L,R)\in M(\mc A)$, we have that $L$ and $R$ are weak$^*$-continuous;
\end{enumerate}
\end{proposition}
\begin{proof}
The proof of (1) is exactly the same as the analogous statement in the proof of
Theorem~\ref{mult_is_dba}.  For (2), let $(a_\alpha)$ be a bounded net in $\mc A$ which
converges weak$^*$ to $a$.  For $b\in\mc A$ and $\mu\in\mc A_*$, we have that
\[ \lim_\alpha \ip{L(a_\alpha)}{\mu\cdot b} = \lim_\alpha \ip{R(b)a_\alpha}{\mu}
= \ip{R(b)a}{\mu} = \ip{L(a)}{\mu\cdot b}. \]
By (1), this is enough to show that $L(a_\alpha)\rightarrow L(a)$ weak$^*$,
so we conclude (see Lemma~\ref{weakstar_lemma}) that $L$ is weak$^*$ continuous.
Similar arguments apply to $R$.
\end{proof}

\begin{theorem}
Let $(\mc A,\mc A_*)$ be a faithful, dual Banach algebra.  Then we can construct
a predual for $M(\mc A)$ which turns $M(\mc A)$ into a dual Banach algebra.  This
weak$^*$-topology on $M(\mc A)$ is the unique one such that, for a bounded
net $(a_\alpha)$ in $\mc A$, and $a\in\mc A$, we have that $a_\alpha\rightarrow a$
weak$^*$ in $\mc A$ if and only if $a_\alpha\rightarrow a$ weak$^*$ in $M(\mc A)$.
\end{theorem}
\begin{proof}
If products are dense in $\mc A$, then we can immediately apply Theorem~\ref{mult_is_dba},
with $\mc B=\mc A$.
Indeed, if we follow the proof of Theorem~\ref{mult_is_dba}, then the only point
at which we use this density is the ensure that a map $T:\mc A\rightarrow\mc B$
actually maps into $\iota(\mc A)$, but clearly this is automatic in the current
situation.
\end{proof}

By Lemma~\ref{weakstarlemma}, we can equivalently say that we can construct a predual
for $M(\mc A)$, say $M(\mc A)_*$, such that the inclusion $\mc A\rightarrow M(\mc A)$
has a preadjoint $M(\mc A)_* \rightarrow \mc A_*$ which has dense range.

\begin{example}\label{dba_non_unital}
Examples of non-unital dual Banach algebras arise in abstract harmonic analysis.
For example, let $G$ be a locally compact group, let $C^*_r(G)$ be the reduced
group C$^*$-algebra of $G$, and let $B_r(G)$ be its dual.  This is a commutative
Banach algebra which can be identified as an algebra of functions on $G$.  As
$C^*_r(G)$ is weak$^*$-dense in $VN(G)$, it follows that $A(G)$ embeds isometrically
into $B_r(G)$.  Conversely, given a continuous function $f:G\rightarrow\mathbb C$,
we have that $f\in B_r(G)$ if and only if there is a constant $c$, such that for
each compact set $K\subseteq G$, there exists $a\in A(G)$ with $a|_K = f|_K$ and
with $\|a\|\leq c$.  See \cite{cowling} for further details (and more generality).

We claim that we can identify $M(B_r(G))$ with an algebra of functions on $G$.
Indeed, let $\mc A$ be any algebra of functions on $G$
such that for each $s\in G$, there exists $a\in\mc A$ with $a(s)=1$.  Fix
$(L,R)\in M(\mc A)$.  As $\mc A$ is commutative, so is $M(\mc A)$, with
$L=R$.  For $s\in G$, let $a(s)=1$, and define $f(s) = L(a)(s)$.  This is
well-defined, for if also $a'(s)=1$, then $f(s) = f(s) a'(s) = L(a)(s) a'(s)
= L(aa')(s) = L(a')(s) a(s) = L(a')(s)$.  So we find a function
$f:G\rightarrow\mathbb C$.  For $b\in\mc A$, we see that $L(b)(s) = L(b)(s)a(s)
= L(a)(s) b(s) = f(s) b(s)$.  So $L(b) = fb$, and hence
\[ M(\mc A) = \big\{ f:G\rightarrow\mathbb C : fa\in\mc A \ (a\in\mc A) \big\}. \]

For $B_r(G)$, we can say a little more.  As $A(G)\subseteq B_r(G)$, for any
$s\in G$ we can find an open neighbourhood $U$ of $s$ and $a\in A(G)$ with
$a|_U=1$.  Then, for $f\in M(B_r(G))$, and any net $(s_\alpha)$ converging to $s$
in $G$, we see that
\[ \lim_\alpha f(s_\alpha) = \lim_\alpha f(s_\alpha) a(s_\alpha)
= \lim_\alpha (fa)(s_\alpha) = (fa)(s) = f(s) a(s) = f(s), \]
as $fa\in B_r(G)$ and so $fa$ is continuous.  So $f$ is continuous.  Similarly,
we can show that $f$ must be bounded.

By \cite[Proposition~3.4]{eymard} or \cite{cowling}, if $a\in B_r(G)$ has compact
support, then actually $a\in A(G)$.  It follows that if $a\in A(G)$ has compact support,
then for $f\in M(B_r(G))$, we see that $fa\in B_r(G)$ has compact support, so
$fa\in A(G)$.  As such $a$ are dense in $A(G)$, and $f:A(G)\rightarrow B_r(G)$ is
bounded, it follows that actually $f$ maps $A(G)$ to $A(G)$.  So $f\in M(A(G))$.

The arguments explored in Section~\ref{lcqg} below will show that the
weak$^*$-topology induced on $M(B_r(G))=M(A(G))$ by the previous theorem
agrees with that constructed in Example~\ref{fourier_example_one}.
\end{example}

\subsection{When we have a bounded approximate identity}

We now return to the case when $\mc A$ has a bounded approximate identity,
and make links with Section~\ref{dba_section}.  Suppose that $\mc A$ is a C$^*$-algebra,
so that $\wap(\mc A^*)=\mc A^*$ and $\mc A$ has a contractive approximate identity.
It follows that $\theta_w:M(\mc A)\rightarrow\wap(\mc A^*)^* = \mc A^{**}$ is
an isometry onto its range.  As $M(\mc A)$ is not, in general, a dual space, we cannot,
in general, expect that $\theta_w$ has weak$^*$-closed range.

\begin{theorem}\label{mult_dba_cai}
Let $\mc A$ be a Banach algebra with a contractive approximate identity.
Then the following are equivalent:
\begin{enumerate}
\item\label{mdc_one} $M(\mc A)$ is a dual Banach algebra for some predual;
\item\label{mdc_two} there exists a closed $\mc A$-submodule $Z$ of $\wap(\mc A^*)$
such that $\theta_0:M(\mc A)\rightarrow Z^*$ is an isometric isomorphism,
where $\theta_0$ is the composition of $\theta_w:M(\mc A)\rightarrow\wap(\mc A^*)^*$
with the quotient map $\wap(\mc A^*)^*\rightarrow Z^*$.
\end{enumerate}
When (\ref{mdc_one}) holds, we can choose $Z$ in (\ref{mdc_two}) such that $\theta_0:
M(\mc A)\rightarrow Z^*$ is weak$^*$ continuous. Thus all possible dual Banach algebra
weak$^*$-topologies which can occur on $M(\mc A)$ arise by the construction of (\ref{mdc_two}).
\end{theorem}
\begin{proof}
If (\ref{mdc_two}) holds, then observe that by \cite[Proposition~2.4]{daws}, the quotient
map $\wap(\mc A^*)^*\rightarrow Z^*$ is an algebra homomorphism turning $Z^*$
into a dual Banach algebra.  Hence $\theta_0$ induces a dual Banach algebra
structure on $M(\mc A)$.

If (\ref{mdc_one}) holds, then choose a dual Banach algebra $(\mc B,\mc B_*)$
and an embedding $\iota:\mc A\rightarrow\mc B$ as in Theorem~\ref{mult_is_dba}.
Indeed, we can choose $\mc B=M(\mc A)$.  Let $\iota_*:\mc B_*\rightarrow
\mc A^*$ be the map given by $\ip{\iota_*(\mu)}{a} = \ip{\iota(a)}{\mu}$ for $a\in\mc A$
and $\mu\in\mc B_*$.  Following, for example, \cite[Theorem~4.10]{runde2}, it
is not hard to show that $\iota_*$ maps into $\wap(\mc A^*)$.  Then $\hat\iota =
(\iota_*)^* : \wap(\mc A^*)^* \rightarrow \mc B$ is a homomorphism which extends $\iota$.

Define $\phi:\mc A\proten B_* \oplus_1 \mc A\proten B_* \rightarrow \wap(\mc A^*)$ by
\[ \phi\big( (a\otimes\mu) \oplus (b\otimes\lambda) \big)
= a\cdot\iota_*(\mu) + \iota_*(\lambda)\cdot b \qquad (a,b\in\mc A,\mu,\lambda\in\mc B_*), \]
and linearity and continuity.  Then $\phi$ is a contraction.
Let $X \subseteq (\mc A\proten\mc B_*) \oplus_1 (\mc A\proten\mc B_*)$ be as in
the proof of Theorem~\ref{mult_is_dba}.  It follows easily that $\phi$ sends $X$
to $\{0\}$, and so we have an induced map $\tilde\phi : (\mc A\proten B_* \oplus_1
\mc A\proten B_*) / X \rightarrow \wap(\mc A^*)$.  Let $Z$ be the closure of the image
of this map.  It is easy to see that $Z$ is an $\mc A$-submodule of $\wap(\mc A^*)$,
so as above, $Z^*$ is a dual Banach algebra.  From the proof of Theorem~\ref{mult_is_dba}
it follows that $Z$ is simply the closure of the image of $\iota_*$.

Let $\theta_0:M(\mc A)\rightarrow Z^*$ be the composition of the map
$\theta_w:M(\mc A)\rightarrow\wap(\mc A^*)^*$ and
the quotient map $\wap(\mc A^*)^* \rightarrow Z^*$.  As $\mc A$ has a contractive
approximate identity, $\theta_0$ is a contraction.  Then, with reference to
Theorem~\ref{arens}, for $a,b\in\mc A,\mu,\lambda\in\mc B_*$ and $(L,R)\in M(\mc A)$,
\begin{align*} & \ip{\phi^*\theta_0(L,R)}{(a\otimes\mu) \oplus (b\otimes\lambda)}
= \ip{L^{**}(\Phi_0)}{a\cdot\iota_*(\mu) + \iota_*(\lambda)\cdot b} \\
&= \ip{\iota_*(\mu)}{L(a)} + \ip{\iota_*(\lambda)}{R(a)}
= \ip{(L,R)}{(a\otimes\mu) \oplus (b\otimes\lambda)}, \end{align*}
where the final dual pairing is as in the proof of Theorem~\ref{mult_is_dba}.
Hence $\phi^* \theta_0:M(\mc A)\rightarrow X^\perp$
is the canonical map, which is an isometric isomorphism.  Hence $\theta_0:M(\mc A)
\rightarrow Z^*$ must be an isometry, and we see that $\tilde\phi^*$ is an isometric
isomorphism between the image of $\theta_0$ and $X^\perp$.

It follows that $\tilde\phi$ is an isometry (and hence an isometric isomorphism
onto $Z$).  Indeed, for $\tau \in \mc A\proten B_* \oplus_1 \mc A\proten B_*$,
we can find $T\in X^\perp$ with $\|T\|=1$ and $\ip{T}{\tau} = \|\tau\|$, the norm
in the quotient $(\mc A\proten B_* \oplus_1 \mc A\proten B_*)/X$.  Then there exists
$\Phi\in Z^*$ in the image of $\theta_0$ with $\tilde\phi^*(\Phi) = T$ and $\|\Phi\|=1$.
Then $\|\tau\| = \ip{T}{\tau} = \ip{\Phi}{\tilde\phi(\tau)} \leq \|\tilde\phi(\tau)\|
\leq \|\tau\|$, so we must have equality throughout.

Hence $\tilde\phi^* : Z^* \rightarrow X^\perp$ is also an isometric isomorphism.
We conclude that $\theta_0:M(\mc A)\rightarrow Z^*$ must surject, and is hence
an isometric isomorphism, as required.
\end{proof}

\begin{example}
Consider $\mc A=L^1(G)$ for a locally compact group $G$.  Then $\wap(\mc A^*)$ can
be identified with a C$^*$-subalgebra of $C(G) \subseteq L^\infty(G)=\mc A^*$ which
contains $C_0(G)$.  Then $M(\mc A) = M(G)$, and the map $\tilde\theta$ is the
restriction of the canonical map $M(G) \rightarrow C(G)^*$ given by integration.
As $1\in\wap(\mc A^*)$, it is easy to see that $\theta_w$ is not weak$^*$-continuous.
Hence, in the previous theorem, we cannot in general take $Z$ to be all of $\wap(\mc A^*)$.
Indeed, in this case, we have $Z = C_0(G)$.
\end{example}

\begin{example}\label{ex::one}
Now consider the Fourier algebra $A(G)$.  Then $A(G)$ has a bounded approximate
identity if and only if $G$ is amenable, in which case it has a contractive approximate
identity, \cite{leptin}.  Then $B(G) = B_r(G)$, and $M(A(G)) = B(G) = B_r(G)$.
We have that $C^*_r(G) \subseteq \wap(VN(G))$ (see \cite{dr})
and so $M(A(G)) = C^*_r(G)^*$.  Hence $Z = C^*_r(G)$ in the above theorem.
\end{example}

So far, we have not discussed the ``non-isometric'' case.  That is, we have been considering
a Banach algebra $\mc A$ to be a dual Banach algebra if $\mc A$ is \emph{isometrically}
isomorphic to a dual space, such that the product is separately weak$^*$-continuous.
However, one can weaken this to just asking for $\mc A$ to be isomorphic to a
dual space (sometimes this gives the same notion of weak$^*$ topology, see for example
\cite[Section~4]{daws}).  For example, in Theorem~\ref{mult_is_dba},
we can weaken $\iota$ from being an isometry to being an isomorphism onto its range.
By following the proof through, we see that now $M(\mc A)$ is only isomorphic (but
not isometric) to $X^\perp$.  Similarly Theorem~\ref{dual_mult_wstar_top} also
works in this setting.  Then we can adapt the previous theorem to the case when
$\mc A$ only has a bounded, but maybe not contractive, approximate identity,
to find $Z \subseteq \wap(\mc A^*)$ with $M(\mc A)$ isomorphic to $Z^*$.

\section{Application to locally compact quantum groups}\label{lcqg}

We shall quickly sketch the theory of locally compact quantum groups, in the
sense of Kustermans and Vaes, \cite{kus,kus2,kus3}.  This is a very short
overview, but it is worth mentioning that actually we need remarkably little
of the theory in order to apply the work of the previous section.

In the von Neumann algebra setting, 
we have a von Neumann algebra $M$ together with a unital normal $*$-homomorphism
$\Delta:M\rightarrow M\vnten M$ which is coassociative in the sense that
$(\id\otimes\Delta)\Delta = (\Delta\otimes\id)\Delta$.  Furthermore, we 
have left and right invariant weights $\varphi, \psi$.  For a weight $\varphi$,
write $\mf n_\varphi = \{ x\in M : \varphi(x^*x)<\infty \}$,
$\mf m_\varphi = \mf n_\varphi^* \mf n_\varphi$ and $\mf p_\varphi = \mf m_\varphi \cap
M^+$ (see \cite{tak2} for further details, for example).  Then invariant means that
\[ \varphi((a\otimes\id)\Delta(x)) = \varphi(x) a(1), \quad
\psi((\id\otimes a)\Delta(y)) = \psi(y) a(1)
\qquad (x\in \mf p_\varphi, y\in\mf p_\psi, a\in M_*^+). \]

Let $(H,\pi,\Lambda)$ be the GNS construction for $\varphi$.  We shall identify
$M$ with $\pi(M) \subseteq \mc B(H)$.  Then $M$ is in \emph{standard form}
on $H$, see \cite[Chapter~IX, Section~1]{tak2}, and so, in particular, for each
$\omega\in M_*$, there exist $\xi,\eta\in H$ with $\omega=\omega_{\xi,\eta}$,
so that $\ip{x}{\omega} = (x(\xi)|\eta)$ for $x\in M$.
There is a \emph{multiplicative unitary}
$W\in\mc B(H\otimes H)$ such that $W^*(1\otimes x)W = \Delta(x)$ for $x\in M$.
Let $A$ be the norm closure of $\{ (\iota\otimes\omega)W : \omega\in \mc B(H)_*\}$.
Then $A$ is a C$^*$-algebra and $\Delta$ restricts to a map $A\rightarrow M(A\otimes A)$.
Furthermore, $\varphi$ and $\psi$ restrict to give KMS weights on $A$, see \cite{kus}.
Then $A$ is a \emph{reduced} C$^*$-algebraic locally compact quantum group.

As $\Delta$ is normal, its predual $\Delta_*$ induces a Banach algebra structure
on $M_*$.  Similarly, the adjoint of $\Delta$ induces a Banach algebra structure
on $A^*$.  As we have identified $A$ with a subalgebra of $M$, we see that we have
a natural map $M_* \rightarrow A^*$.  As discussed in \cite[Pages~913-4]{kus},
we define $L^1(A)$ to be the closed linear span of functionals $x \varphi y^*$,
where $x,y\in\mf n_\varphi$.  Here $\ip{x\varphi y^*}{z} = \ip{\varphi}{y^*zx}$
for $z\in A$, which makes sense as $\mf n_\varphi$ is a left ideal.  Then
\[ \varphi(y^*zx) = \big( z \Lambda(x) \big| \Lambda(y) \big)
= \ip{\omega}{z} \qquad (z\in A), \]
where $\omega = \omega_{\Lambda(x),\Lambda(y)} \in B(H)_*$.  An application
of Kaplansky's Density Theorem thus allows us to isometrically identify $M_*$ with
$L^1(A)$.  Then \cite[Pages~913-4]{kus} shows that $L^1(A)$ is an ideal in $A^*$
(compare with Lemma~\ref{ideal_in_uni_dual} below).
We also see that $L^1(A)$ norms $A$, and so $L^1(A)$ is weak$^*$-dense in $A^*$.

Also \cite[Proposition~4.2]{runde1} shows that $A^*$ is a dual Banach algebra.
Furthermore, \cite[Proposition~1]{nh} shows that $M_*$ is faithful.  As $\Delta$ is
a complete isometry, it follows that $\Delta_*$ is a complete quotient map, in
particular it is surjective, so certainly $\{\omega_1\omega_2:\omega_1,\omega_2\in
L^1(\G)\}$ is linearly dense in $L^1(\G)$.

We naturally have actions of $M$ on $M_*$ and of $A$ on $A^*$, which we shall
write by juxtaposition to avoid confusion with the actions of the Banach algebras
$M_*$ on $M$ and $A^*$ on $A$.
Finally, we shall follow \cite{nh,jnr,runde1} and use some notation due to Ruan.
We write $\G$ for a locally compact quantum group, and set $L^\infty(\G) = M,
L^1(\G) = M_*$, $C_0(\G) = A$ and $M(\G)=A^*$.

\begin{theorem}\label{lcqg_predual_mult_dba}
Let $\G$ be a locally compact quantum group.  Then $M(L^1(\G))$ is
a dual Banach algebra, and the resulting dual Banach algebra weak$^*$-topology
is unique such that the map $M(\G)\rightarrow M(L^1(\G))$ satisfies the
conditions of Theorem~\ref{dual_mult_wstar_top}.
\end{theorem}
\begin{proof}
We simply need to verify the conditions of Theorem~\ref{mult_is_dba}.
As discussed above, we naturally have a map $\iota:L^1(\G)\rightarrow M(\G)$
which is an isometric homomorphism with $\iota(L^1(\G))$ being an ideal.
So we need only verify that the induced map $M(\G) \rightarrow M(L^1(\G))$
is injective.  Suppose not, so that there exists $\mu_0\in M(\G)$ with
$\mu_0 a = 0$ for $a\in L^1(\G)$.  Observe that if instead $a\mu_0=0$ for $a\in L^1(\G)$,
then $a\mu_0 b=0$ for $a,b\in L^1(\G)$ so as $L^1(\G)$ is faithful, $\mu_0 b=0$ for
$b\in L^1(\G)$.  Consequently
\[ \ip{\mu_0 \otimes a}{\Delta(x)} = 0 \qquad (a\in L^1(\G), x\in C_0(\G)). \]
However, we claim that $\{ (\id\otimes a)\Delta(x) : a\in L^1(\G), x\in C_0(\G) \}$
is linearly dense in $C_0(\G)$, from which it follows that $\mu_0=0$, as required.

To prove the claim, we first note that \cite[Corollary~6.11]{kus} shows that
$\{ \Delta(x)(1\otimes y) : x,y\in C_0(\G) \}$ is linearly dense in $C_0(\G)\otimes C_0(\G)$.
By taking the adjoint, and using that $L^1(\G)$ is weak$^*$-dense in $C_0(\G)^*$,
it follows that
\[ \{ (\id\otimes ya)\Delta(x) : x,y\in C_0(\G), a\in L^1(\G) \} \]
is linearly dense in $C_0(\G)$.  As $ya\in L^1(\G)$ for $a\in L^1(\G)$ and $y\in C_0(\G)
\subseteq L^\infty(\G)$, it follows immediately that
$\{ (\id\otimes a)\Delta(x) : a\in L^1(\G), x\in C_0(\G) \}$
is linearly dense in $C_0(\G)$, as required.
\end{proof}

Let us return to the example of the Fourier algebra, Example~\ref{fourier_example_one}.
We used the embedding $A(G)\rightarrow C^*(G)^* = B(G)$, which seemed natural
in light of \cite{eymard} (where the Fourier algebra is basically defined to be a
certain ideal in $B(G)$).  However, the above theorem considers the reduced setting,
which means in this case considering $C^*_r(G)$ and hence $A(G)\rightarrow B_r(G)$.
As mentioned above in Example~\ref{ex::one} we have that $B_r(G) = B(G)$ only when $G$
is amenable.  We shall now show that using $B_r(G)$ does indeed give the same
weak$^*$-topology on $M(A(G))$.  Indeed, we shall show the quantum version of this.

Firstly, we need to say what the quantum analogue of $B(G)$ is.
In \cite[Section~11]{kus1} Kustermanns gives the notion of a \emph{universal
C$^*$-algebraic quantum group}; let us very quickly sketch this.
A \emph{C$^*$-algebraic quantum group} is essentially
as described above, but without assuming that the left and right invariant
weights are faithful.  Given such an object, we can form a $*$-algebra $\mc A$,
called the \emph{coefficient $*$-algebra of $A$}.  Given such an $\mc A$, we can
form a maximal C$^*$-algebra $A_u$ which contains $\mc A$ densely.  $A_u$ can
be given the structure of a C$^*$-algebra quantum group, that is, $\Delta_u:
A_u\rightarrow M(A_u\otimes A_u)$ and left and right invariant weights $\varphi_u,
\psi_u$.  We call $A_u$ the \emph{universal enveloping algebra of $\mc A$}.
Then we can find a surjective $*$-homomorphism $\pi:A_u\rightarrow A$ with
\[ \Delta\pi = (\pi\otimes\pi)\Delta_u, \quad
\varphi \pi = \varphi_u, \quad \psi\pi = \psi_u. \]
All of this generalises the passage of $C^*_r(G)$ to $C^*(G)$ and back again.

From now on, fix a locally compact quantum group $\G$, let $A=C_0(\G)$, and let $A_u$
be the universal C$^*$-algebraic quantum group associated with $A$.  As
$\pi:A_u\rightarrow A$ is a surjective $*$-homomorphism, it is a quotient map,
and so $\pi^*:A^*\rightarrow A_u^*$ is an isometry onto its range.
As $\Delta\pi = (\pi\otimes\pi)\Delta_u$, it follows that $\pi^*$ is a homomorphism
between the Banach algebras $A^*$ and $A_u^*$.  As $L^1(\G)$ is identified with
a closed ideal of $A^*$, we identify $L^1(\G)$ as a closed subalgebra of $A_u^*$.
Let $\iota:L^1(\G)\rightarrow A_u^*$ be the map thus constructed.

\begin{lemma}
With notation as above, $(A_u,A_u^*)$ is a dual Banach algebra.
\end{lemma}
\begin{proof}
We adapt the proof of \cite[Proposition~4.3]{runde1}.  Indeed, it is enough
to show that $\Delta_u(x)(1\otimes y)\in A_u\otimes A_u$ for $x,y\in A_u$, which
follows from \cite[Proposition~6.1]{kus1}.
\end{proof}

\begin{proposition}\label{ideal_in_uni_dual}
With notation as above, $L^1(\G)$ is an ideal in $A_u^*$.  Furthermore,
the induced map $A_u^* \rightarrow M(L^1(\G))$ is injective.
\end{proposition}
\begin{proof}
We adapt the argument given in \cite[page~914]{kus}.  Let $\omega\in A_u^*$ be
a state, and let $(K,\theta,\xi)$ be a cyclic GNS construction for $\omega$.
Let $x,y\in \mf n_\varphi$ and let $a = \omega_{\Lambda(x),\Lambda(y)}\in L^1(\G)$.
Let $H=L^2(\G)$, and recall that we identify $A$ with a subalgebra of $\mc B(H)$.
Then $\ip{\iota(a)}{z} = ( \pi(z) \Lambda(x) | \Lambda(y) )$ for $z\in A_u$.
Let $\mc B_0(H)$ be the compact operators on $H$, and recall that $M(\mc B_0(H)) = \mc B(H)$.
From \cite[Proposition~5.1]{kus1} and \cite[Proposition~6.2]{kus1}, there exists
$\mc V \in M(A_u\otimes\mc B_0(H))$ such that $(\id\otimes\pi)\Delta_u(z) =
\mc V^*(1\otimes\pi(z))\mc V$ in $M(A_u\otimes\mc B_0(H))$, for $z\in A_u$.  Let
$P = (\theta\otimes\id)(\mc V) \in M(\mc B(K)\otimes \mc B_0(H)) \subseteq \mc B(K\otimes H)$.
Then, for $z\in A_u$,
\begin{align*} \ip{\omega\iota(a)}{z} &= \ip{\omega\otimes\iota(a)}{\Delta(z)}
= \ip{\omega}{(\id\otimes a)\big((\id\otimes\pi)\Delta_u(z)\big)} \\
&= \ip{\omega}{(\id\otimes\omega_{\Lambda(x),\Lambda(y)}) \mc V^*(1\otimes\pi(z))\mc V} \\
&= \big( P^* (1\otimes\pi(z)) P (\xi\otimes\Lambda(x)) \big| \xi\otimes\Lambda(y) \big),
\end{align*}
from which it follows that there exists $\omega_0\in \mc B(H)_*$ with
$\ip{\omega\iota(a)}{z} = \ip{\pi(z)}{\omega_0}$.  It is now immediate that
$L^1(\G)$ is a left ideal in $A_u^*$.

Then \cite[Proposition~7.2]{kus1} shows the
existence of an anti-$*$-automorphism $R_u$ of $A_u$ with $\pi R_u = R \pi$,
with $R$ being the unitary antipode on $A$.  As $\chi(R_u\otimes R_u)\Delta_u =
\Delta_u R_u$, it follows that $R_u^*$ is an anti-homomorphism on $A_u^*$, and as
$R$ leaves $L^1(\G)$ invariant, the same is true for $R_u$.  Hence $L^1(\G)$ is
also a right ideal in $A_u^*$, and hence an ideal, as claimed.

To show that the map $A_u^*\rightarrow M(L^1(\G))$ is injective, as in the proof
of Theorem~\ref{lcqg_predual_mult_dba}, it is enough to show that
$\{ (\id\otimes\iota(a))\Delta_u(z) : a\in L^1(\G), z\in A_u \}$ is linearly
dense in $A_u$.  By \cite[Proposition~6.1]{kus1}, we have that
$\{ (1\otimes z_1)\Delta_u(z_2) : z_1,z_2 \in A_u \}$ is linearly dense in
$A_u\otimes A_u$.  As $\iota(a)z = \iota(a\pi(z))$ for $a\in L^1(\G), z\in A_u$,
it follows that
\[ \lin\{ \big(\id\otimes \iota(a\pi(z))\big)\Delta_u(w) : z,w\in A_u, a\in L^1(\G) \} \]
is dense in $A_u$, which completes the proof.
\end{proof}

\begin{theorem}
With notation as above, we use the inclusion $\iota:L^1(\G)\rightarrow A_u^*$
to induce a weak$^*$-topology on $M(L^1(\G))$.  This topology agrees with that
given by Theorem~\ref{lcqg_predual_mult_dba}, that is, when using $A^*$.
\end{theorem}
\begin{proof}
This follows essentially because $\iota$ factors through $\pi^*$.  Indeed,
let $\iota_A:L^1(\G)\rightarrow A^*$ be the map considered in
Theorem~\ref{lcqg_predual_mult_dba}, and recall the construction in
Theorem~\ref{dual_mult_wstar_top}.  We have that a net $(L_\alpha,R_\alpha)$
in $M(L^1(\G))$ converges weak$^*$ to $(L,R)$ when 
\[ \lim_\alpha \ip{\iota L_\alpha(a)}{x} + \ip{\iota R_\alpha(b)}{y}
= \ip{\iota L(a)}{x} + \ip{\iota R(b)}{y} \qquad (a,b\in L^1(\G), x,y\in A_u). \]
As $\iota = \pi^* \iota_A$, this is equivalent to
\[ \lim_\alpha \ip{\iota_A L_\alpha(a)}{\pi(x)} + \ip{\iota_A R_\alpha(b)}{\pi(y)}
= \ip{\iota_A L(a)}{\pi(x)} + \ip{\iota_A R(b)}{\pi(y)}. \]
As $\pi$ is surjective, this is equivalent to $(L_\alpha,R_\alpha)$
converging weak$^*$ to $(L,R)$ in the weak$^*$-topology induced by $\iota_A$.
\end{proof}

Given that we have introduced $A_u$, now seems a good place to apply some
of the results from section~\ref{sec::bai}.

\begin{theorem}
The algebra $L^1(\G)$ has a bounded approximate identity if and only if the
natural map $M(\G)\rightarrow M(L^1(\G))$ is an isomorphism.
When $L^1(\G)$ has a bounded approximate identity,
both $M(\G)$ and $A_u^*$ are isometrically isomorphic to $M(L^1(\G))$.
\end{theorem}
\begin{proof}
By  \cite[Theorem~3.1]{bt}, we know that $L^1(\G)$ has a bounded approximate identity
if and only if $M(\G)$ is unital.  So, if the map $M(\G) \rightarrow M(L^1(\G))$
is surjective, then $M(\G)$ is unital, and so $L^1(\G)$ has a bounded approximate
identity.

If $L^1(\G)$ has a bounded approximate identity, then \cite[Theorem~2]{nh} shows that
we can even choose the bounded approximate identity to be contractive, and to be a net
of states in $L^1(\G)$.  
By \cite[Theorem~4.4]{runde1}, we have that $C_0(\G) \subseteq \wap(L^1(\G))$.
So by Lemma~\ref{dba_lemma}, we can apply Theorem~\ref{submod_rep_thm} with
$F=C_0(\G)\subseteq L^1(\G)^*$.  Then $\iota_F:L^1(\G)\rightarrow M(\G)$ is
an isometry, and so $\theta_F:M(L^1(\G)) \rightarrow M(\G)$ is an 
isometry, whose range is the idealiser of $L^1(\G)$ in $M(\G)$.  As
$L^1(\G)$ is an ideal in $M(\G)$, it follows that $\theta_F$ is an isometric
isomorphism, as required.

Then exactly the same argument applies to $A_u^*$, given the work above.
\end{proof}

We remark that it would be interesting to know if $M(L^1(\G))=A_u^*$ is
equivalent to $L^1(\G)$ having a bounded approximate identity.  Of course, even
in the co-commutative case, when $L^1(\G)=A(G)$, this is rather hard, see
\cite{losert}.

\subsection{Completely bounded multipliers}\label{cbmult_lcqg}

We now turn to the operator space setting.  Let us just record the completely
bounded version of our work in Section~\ref{mult_dual_sec}.  As shown in \cite{lemerdy},
it is possible for a Banach space $E$ to be such that $E^*$ has an operator space
structure which is not the dual space structure of any operator space structure on
$E$.  Consequently, we have to be a little careful when it comes to duality arguments
(but we are okay, essentially by an application of Lemma~\ref{weakstar_lemma}).

\begin{theorem}\label{op_sp_mult_is_dba}
Let $\mc A$ be a CCBA with dense products, let $(\mc B,\mc B_*)$ be a dual CCBA, and let
$\iota:\mc A \rightarrow \mc B$ be a complete isometry with $\iota(\mc A)$
an ideal in $\mc B$.  Suppose further than the induced map $\theta:\mc B\rightarrow M_{cb}(\mc A)$
is injective.  Then there is a unique operator space $X$ such that $M_{cb}(\mc A)$ is
completely isometrically isomorphic to $X^*$, turning $M_{cb}(\mc A)$ into a dual CCBA,
and such that for a bounded net $(b_\alpha)$ in $\mc B$, $b_\alpha\rightarrow b$ weak$^*$
in $\mc B$ if and only if $\theta(b_\alpha)\rightarrow\theta(b)$ in $M_{cb}(\mc A)$.
\end{theorem}
\begin{proof}
We simply follow the construction of Theorem~\ref{mult_is_dba}.  We note that $\oplus_\infty$
and $\oplus_1$ have operator space analogue (see \cite[Section~2.6]{pisier}).
Similarly, we work with the operator space projective tensor product and so forth.

For uniqueness, we similarly adapt the proof of Theorem~\ref{dual_mult_wstar_top}
making use of the completely bounded version of Lemma~\ref{weakstar_lemma}.
\end{proof}

\begin{example}\label{fourier_example_two}
Consider again the Fourier algebra $A(G)$.  As $C^*(G)\rightarrow C^*_r(G)$
is a $*$-homomorphism, it is a complete quotient map, and so the adjoint,
$C^*_r(G)^* \rightarrow B(G)$ is a complete isometry.
Consider $C^*_r(G)^{**}$ the universal enveloping von Neumann algebra of $C^*_r(G)$.
Then, see \cite[Chapter~III, Section~2]{tak1}, there is a unique normal
surjective $*$-homomorphism $C^*_r(G)^{**} \rightarrow VN(G)$, which is hence
a complete quotient map.  Its preadjoint is hence a complete isometry
$A(G)\rightarrow C^*_r(G)$.  It is not hard to check that the composition
$A(G)\rightarrow B(G)$ is the canonical map, which is thus a complete isometry.

Hence, exactly as in Example~\ref{fourier_example_one}, we can use our theorem to
construct a weak$^*$-topology on $M_{cb}(A(G))$.  This was first done in
\cite[Proposition~1.10]{DcH}, at the Banach space level, in exactly the same way
as detailed in Example~\ref{fourier_example_one}, with $M(A(G))$ replaced with $M_{cb}(A(G))$.
Again, it follows that we get the same weak$^*$-topology from our abstract theorem.

Of course, our theorem actually turns $M_{cb}(A(G))$ into a dual CCBA.  An operator
space predual for $M_{cb}(A(G))$ was constructed by Spronk, \cite[Section~6.2]{spronk}.
Indeed, let $K$ be the closure of $\{ \sum_i f_i\otimes g_i \in L^1(G)\otimes L^1(G)
: \sum_i f_i*g_i=0 \}$ in $L^1(G) \otimes^h L^1(G)$.  Here we use the usual convolution
product on $L^1(G)$, and $\otimes^h$ is the completed Haagerup tensor product
(see Section~\ref{haatenprod} below).  Then
\[ Q(G) = (L^1(G) \otimes^h L^1(G))/K \]
is an operator space with $Q(G)^*$ isometrically isomorphic to $M_{cb}(A(G))$.
The dual pairing is given by
\[ \ip{(f\otimes g)+K}{a} = \int_G (f*g)(s) a(s) \ ds \qquad (f,g\in L^1(G), a\in M_{cb}(A(G))). \]
It is shown in \cite[Corollary~6.6]{spronk} that, as a Banach space, $Q(G)$ is
isometric to the predual constructed in \cite[Proposition~1.10]{DcH}.
As such, the same argument as in Example~\ref{fourier_example_one}
shows that $Q(G)$ is completely isometrically isomorphic to the predual
constructed by Theorem~\ref{op_sp_mult_is_dba}.
\end{example}

\subsection{Duality and multipliers}

For a locally compact group $G$, consider the Fourier algebra $A(G)$.  The multiplier
algebra $MA(G)$ can be (compare with Example~\ref{dba_non_unital} above) identified with
\[ \{ f\in C^b(G) = M(C_0(G)) : fa\in A(G) \ (a\in A(G)) \}. \]
Similarly, consider $L^1(G)$, with multiplier algebra $M(L^1(G)) = M(G)$.  The
left regular representation $\lambda:L^1(G)\rightarrow C^*_r(G)$ extends to a homomorphism
$\lambda:M(G)\rightarrow VN(G)$, which actually maps into $M(C^*_r(G))$.  Indeed,
$M(G)$ can be identified with
\[ \{ x\in M(C^*_r(G)) : x\lambda(f), \lambda(f)x \in \lambda(L^1(G)) \ (f\in L^1(G)) \}. \]

We wish to explain this in the language of locally compact quantum groups, but we
need to define the analogue of the left regular representation.  Given a locally compact
quantum group $\G$, we form $L^2(\G)$ and the multiplicative unitary $W$, as before.
Then we may define $\lambda:L^1(\G)\rightarrow\mc B(L^2(\G))$ by
\[ \lambda(\omega) = (\omega\otimes\iota)(W) \qquad (\omega\in L^1(\G)). \]
Then $\lambda$ is a homomorphism which maps into $C_0(\hat\G)$, with dense range.
If $\G$ is commutative, then $\lambda:L^1(G)\rightarrow C^*_r(G)$ is the usual
left regular representation; if $\G$ is co-commutative, then
$\lambda:A(G)\rightarrow C_0(G)$ is just the inclusion map.

Thus, when $\G$ is commutative or co-commutative, we can identify $M(L^1(\G))$
with the algebra
\[ \{ x\in M(C_0(\hat\G)) : x\lambda(\omega), \lambda(\omega)x \in \lambda(L^1(\G))
\ (\omega\in L^1(\G)) \}. \]
Indeed, in both these cases, $\lambda$ actually extends to a contractive homomorphism
$M(L^1(\G)) \rightarrow M(C_0(\hat\G))$.  Thus we identify the abstract multiplier
algebra $M(L^1(\G))$ with a concrete multiplier algebra, or what, in this paper,
we refer to as an idealiser.

This idea is explored for Kac algebras, by Kraus and Ruan, in \cite{KR}.  If we restrict
attention to $M_{cb}(L^1(\G))$, then everything carries over.  It is necessary to
consider unbounded operators, and extensive use is made of antipode.  In this section, we
shall quickly show that the ideas of Kraus and Ruan do, in some sense, extend to the
locally compact group case: of course, here, we do not have a bounded antipode, and so
a modification of the argument is needed.  This allows us to identify $M(L^1(\G))$ with
an idealiser in $M(C_0(\hat\G))$.  Then $M(C_0(\hat\G))$ is a subalgebra of $L^\infty(\hat\G)$,
and it turns out that the closed unit ball of $M(L^1(\G))$ is weak$^*$-closed.
A standard argument then shows that the weak$^*$-topology on $L^\infty(\hat\G)$
induces a dual CCBA structure on $M(C_0(\hat\G))$.  We show that this weak$^*$-topology
agrees with that given by Theorem~\ref{lcqg_predual_mult_dba}.

Let us recall a little about $\lambda:L^1(\G)\rightarrow C_0(\hat\G)$.  As the antipode
$S$ is generally unbounded on $L^\infty(\G)$, we cannot (unlike the Kac algebra case) turn
$L^1(\G)$ into a $*$-algebra in a natural way.  However, following \cite{kus1}, we define
\[ L^1_*(\G) = \{ \omega\in L^1(\G) : \exists\,\sigma\in L^1(\G), \
\ip{x}{\sigma} = \ip{S(x)}{\omega^*} \ (x\in D(S)) \}. \]
Here $D(S)\subseteq L^\infty(\G)$ is the (strong$^*$-dense) domain of $S$, and
$\omega^*$ denotes the normal functional
\[ L^\infty(\G)\rightarrow\mathbb C; \quad x\mapsto \overline{\ip{x^*}{\omega}}. \]
Then $L^1_*(\G)$ carries a natural involution, $\omega^\dagger = \sigma$.  So, by definition,
we have $\ip{x}{\omega^\dagger} = \ip{S(x)}{\omega^*}$ for $x\in D(S)$.  As argued in
\cite[Section~3]{kus1}, we have that $L^1_*(\G)$ is a dense subalgebra of $L^1(\G)$.
Then \cite[Proposition~3.1]{kus1} shows that we can characterise $L^1_*(\G)$ as
\[ L^1_*(\G) = \{ \omega\in L^1(\G) : \exists\, \sigma\in L^1(\G), \
\lambda(\omega)^* = \lambda(\sigma) \}, \]
and furthermore, $\lambda$ is a $*$-homomorphism on $L^1_*(\G)$.  See also
\cite[Definition~2.3]{kus2} and the discussion thereafter.

\begin{proposition}
Let $\G$ be a locally compact quantum group, and let $(L,R)\in M(L^1(\G))$.
There exists a densely defined, preclosed operator $a_0$ on $L^2(\G)$ with
domain $D(a_0)$ such that $D(a_0)$ is invariant under $\lambda(\omega)$ for
all $\omega\in L^1(\G)$, and with
\[ a_0 \lambda(\omega) \xi = \lambda(L(\omega)) \xi, \quad
\lambda(\omega) a_0 \xi = \lambda(R(\omega))\xi
\qquad (\omega\in L^1(\G), \xi\in D(a_0)). \]
\end{proposition}
\begin{proof}
Let $X$ be the linear span of vectors of the form $\lambda(\omega)\xi$ where
$\omega\in L^1(\G)$ and $\xi\in L^2(\G)$.  As $\lambda(L^1(\G))$ is dense in
$C_0(\hat\G)$, and $C_0(\hat\G)$ acts non-degenerately on $L^2(\G)$, it follows
that $X$ is dense in $L^2(\G)$.

We first show that if $\xi\in L^2(\G)$ with $\lambda(\omega)\xi=0$ for all $\omega\in L^1(\G)$,
then $\xi=0$.  Indeed, for $\eta\in L^2(\G)$ and $\omega\in L^1_*(\G)$, we have
\[ 0 = \big( \lambda(\omega)\xi \big| \eta \big)
= \big( \xi \big| \lambda(\omega^\dagger)\eta \big). \]
As $L^1_*(\G)$ is dense in $L^1(\G)$, we see that $\lin\{\lambda(\omega)\eta:
\omega\in L^1_*(\G), \eta\in L^2(\G) \}$ is dense in $X$, which is dense in $L^2(\G)$.
Thus $\xi=0$, as claimed.

Define $a_0:X\rightarrow X$ by $a_0\lambda(\omega)\xi = \lambda(L(\omega))\xi$, and
linearity.  This is well-defined, for if $\sum_{i=1}^n \lambda(\omega_i)\xi_i=0$, then
for any $\omega\in L^1(\G)$, we have
\[ \lambda(\omega) \sum_{i=1}^n \lambda(L(\omega_i))\xi_i
= \sum_{i=1}^n \lambda(\omega L(\omega_i))\xi_i
= \lambda(R(\omega)) \sum_{i=1}^n \lambda(\omega_i)\xi_i = 0. \]
As $\omega$ was arbitrary, we have that $\sum_{i=1}^n \lambda(L(\omega_i))\xi_i=0$,
as required.

Then, for $\omega\in L^1(\G)$ and $\xi=\lambda(\sigma)\eta\in X$, we have that
\[ \lambda(\omega) a_0 \xi = \lambda(\omega) \lambda(L(\sigma))\eta
= \lambda(R(\omega)) \xi, \]
as required.

Finally, we show that $a_0$ is preclosed.  If $(\xi_n)\subseteq X$ with
$\xi_n\rightarrow 0$ and $a_0(\xi_n)\rightarrow \xi$, then for any $\omega\in L^1_*(\G)$
and $\eta\in L^2(\G)$,
\[ \big( \xi \big| \lambda(\omega^*)\eta \big) =
\lim_n \big( a_0(\xi_n) \big| \lambda(\omega^*)\eta \big)
= \lim_n \big( \lambda(R(\omega)) (\xi_n) \big| \eta \big) = 0. \]
Again, by density, this is enough to show that $\xi=0$ as required.
\end{proof}

For completely bounded multipliers, we can say more.

\begin{theorem}\label{thm:rep_mcb}
Let $\G$ be a locally compact quantum group, and let $(L,R)\in M_{cb}(L^1(\G))$.
There exists a unique $a\in M(C_0(\hat\G))$ with
\[ a\lambda(\omega) = \lambda(L(\omega)), \quad
\lambda(\omega)a=\lambda(R(\omega)) \qquad (\omega\in L^1(\G)). \]
The resulting map $\Lambda:M_{cb}(L^1(\G)) \rightarrow M(C_0(\hat\G))$ is a
completely contractive algebra homomorphism.
\end{theorem}
\begin{proof}
We continue with the notation of the last proof.
For $\xi\in X$ and $\eta,\xi_0,\eta_0\in L^2(\G)$, we have
\begin{align*}
\big( W(1\otimes a_0)(\xi_0\otimes \xi) \big| \eta_0\otimes\eta \big)
&= \big( \lambda(\omega_{\xi_0,\eta_0}) a_0(\xi) \big| \eta \big)
= \big( \lambda(R(\omega_{\xi_0,\eta_0})) \xi \big| \eta \big) \\
&= \ip{W}{R(\omega_{\xi_0,\eta_0}) \otimes \omega_{\xi,\eta}}
= \ip{(R^*\otimes\iota)(W)}{\omega_{\xi_0,\eta_0} \otimes \omega_{\xi,\eta}} \\
&= \big( (R^*\otimes\iota)(W)(\xi_0\otimes \xi) \big| \eta_0\otimes\eta \big).
\end{align*}
Here we use that $R\in\mc{CB}(L^1(\G))$ and so $R^*\in\mc{CB}(L^\infty(\G))$
and thus $R^*\otimes\iota \in \mc{CB}(L^\infty(\G)\vnten L^\infty(\hat G))$.
So, on the algebraic tensor product $L^2(\G) \otimes X$, we have that
\[ (R^*\otimes\iota)(W) = W(1\otimes a_0). \]
As the left-hand side is bounded, and $W$ is a unitary, it follows that $a_0$
is actually bounded.  Let $a$ be the continuous extension of $a_0$ to all of $L^2(\G)$,
so that $1\otimes a = W^* (R^*\otimes\iota)(W)$.  It follows by density that
$\lambda(\omega)a=\lambda(R(\omega))$ and $a\lambda(\omega)=\lambda(L(\omega))$
for $\omega\in L^1(\G)$.  As $\lambda(L^1(\G))$ is dense in $C_0(\hat\G)$, it
follows by continuity that $a\in M(C_0(\hat\G))$.

It is clear that the resulting map $\Lambda:M_{cb}(L^1(\G))\rightarrow M(C_0(\hat\G))$ is
an algebra homomorphism.  As $1\otimes a = W^* (R^*\otimes\iota)(W)$, it follows
immediately that $\|a\| \leq \|R\|_{cb}$, and it also follows that
$\Lambda:(L,R)\mapsto a$ is actually a complete contraction.
\end{proof}

Conversely, if $a\in M(C_0(\G))$ idealises $\lambda(L^1(\G))$ then we find
maps $L,R:L^1(\G)\rightarrow L^1(\G)$ with $\lambda(L(\omega)) = a\lambda(\omega)$
and $\lambda(R(\omega)) = \lambda(\omega) a$ for $\omega\in L^1(\G)$.  It hence
follows that $(L,R)\in M(L^1(\G))$.  It is not clear to us if there is any easily
stated condition on $a$ which will ensure that $(L,R)\in M_{cb}(L^1(\G))$.
Indeed, it is also unclear if the previous theorem can be extended to $M(L^1(\G))$.
Kraus and Ruan show slightly better results for Kac algebras in \cite{KR}.

We next show that the restriction of the weak$^*$-topology on $L^\infty(\hat\G)$ induces
a dual Banach algebra structure on $M_{cb}(L^1(\G))$ which agrees with the weak$^*$-topology
on $M_{cb}(L^1(\G))$ constructed by (the operator space version of)
Theorem~\ref{lcqg_predual_mult_dba}.

We first collect some properties of the map $\lambda$.  Again, these are weaker than
in the Kac algebra case, but are sufficient for our needs.

\begin{lemma}\label{lemma::lambdahat}
Let $\omega\in L^1(\G)$ and $\hat\omega\in L^1(\hat\G)$.
We have that $\ip{\lambda(\omega)}{\hat\omega} = \ip{\hat\lambda(\hat\omega^*)^*}{\omega}$,
that $\hat\lambda(\hat\omega^* \lambda(\omega)^*)^* =
\omega\cdot\hat\lambda(\hat\omega^*)^*$, and that
$\hat\lambda( \lambda(\omega)^* \hat\omega^*)^* =
\hat\lambda(\hat\omega^*)^*\cdot\omega$.
\end{lemma}
\begin{proof}
Let $\omega=\omega_{\xi,\eta}$ and $\hat\omega=\omega_{\alpha,\beta}$.  We have that
$\hat W = \sigma W^* \sigma$ where $\sigma$ is the ``swap map'' on $L^2(\G)\otimes L^2(\G)$.
Thus
\begin{align*} \ip{\lambda(\omega)}{\hat\omega} &=
\big( (\omega\otimes\iota)(W) \alpha \big| \beta \big)
= \big( W(\xi\otimes\alpha) \big| \eta\otimes\beta \big)
= \big( \alpha\otimes\xi \big| \hat W(\beta\otimes\eta) \big) \\
&= \overline{ \big( (\omega_{\beta,\alpha}\otimes\iota)(\hat W) \eta \big| \xi \big) }
= \overline{ \big( \hat\lambda(\omega_{\beta,\alpha}) \eta \big| \xi \big) }
= \big( \hat\lambda(\omega_{\beta,\alpha})^* \xi \big| \eta \big) 
= \ip{ \hat\lambda(\hat\omega^*)^*}{\omega}.
\end{align*}
Here we use that $\hat\omega^* = \omega_{\beta,\alpha}$, which follows as
\[ \ip{x}{\hat\omega^*} = \overline{ \big( x^*\alpha\big| \beta \big) }
= \big( x\beta \big| \alpha \big) = \ip{x}{\omega_{\beta,\alpha}}
\qquad (x\in L^\infty(\hat\G)). \]

For the second claim, for $\sigma\in L^1(\G)$, we have
\[ \ip{\hat\lambda(\hat\omega^* \lambda(\omega)^*)^*}{\sigma}
= \ip{\lambda(\sigma)}{\lambda(\omega) \hat\omega}
= \ip{\lambda(\sigma\omega)}{\hat\omega}
= \ip{\hat\lambda(\hat\omega^*)^*}{\sigma\omega}
= \ip{\omega\cdot\hat\lambda(\hat\omega^*)^*}{\sigma}, \]
using that $\hat\omega^* \lambda(\omega)^* = (\lambda(\omega)\hat\omega)^*$.
The third claim follows analogously.
\end{proof}

\begin{proposition}
Let $\Lambda:M_{cb}(L^1(\G)) \rightarrow M(C_0(\hat\G))$ be the completely
contractive homomorphism constructed in Theorem~\ref{thm:rep_mcb}.
Let $(x_\alpha)$ be a net in the closed unit ball of $M_{cb}(L^1(\G))$
such that $(\Lambda(x_\alpha))$ converges weak$^*$ in $L^\infty(\hat\G)$.
Then there exists $x\in M_{cb}(L^1(\G))$ with $\|x\|_{cb}=1$ and
$\Lambda(x_\alpha) \rightarrow \Lambda(x)$ weak$^*$.
\end{proposition}
\begin{proof}
Let $y\in L^\infty(\hat\G)$ be the weak$^*$ limit of $(\Lambda(x_\alpha))$,
and let
$X=\{ \hat\lambda(\hat\omega)^* : \hat\omega\in L^1(\hat\G) \}\subseteq C_0(\G)$.
Fix $\omega\in L^1(\G)$, and define $\mu:X\rightarrow\mathbb C$ by
\[ \mu\big( \hat\lambda(\hat\omega^*)^* \big) = \ip{\hat\lambda(y^*\hat\omega^*)^*}{\omega}. \]
As $\hat\lambda$ is an injection, it follows that $\mu$ is well-defined.
Clearly $\mu$ is linear.  Notice that $y^*\hat\omega^* = (\hat\omega y)^*$.
Then we calculate that
\begin{align*}
\ip{\hat\lambda(y^*\hat\omega^*)^*}{\omega} &=
\ip{\lambda(\omega)}{\hat\omega y}
= \lim_\alpha \ip{\Lambda(x_\alpha) \lambda(\omega)}{\hat\omega}
= \lim_\alpha \ip{\lambda(x_\alpha \omega)}{\hat\omega}
= \lim_\alpha \ip{\hat\lambda(\hat\omega^*)^*}{x_\alpha\omega}.
\end{align*}
It follows that
\[ \big|\mu\big( \hat\lambda(\hat\omega^*)^* \big)\big| \leq
\|\omega\| \|\hat\lambda(\hat\omega^*)^*\|. \]
So $\mu$ is a bounded linear map, which extends by continuity to a member
of $C_0(\G)^* = M(\G)$, as $X$ is dense in $C_0(\G)$.  We have hence defined
a bounded linear map $L:L^1(\G) \rightarrow M(\G)$ which satisfies
\[ \ip{L(\omega)}{\hat\lambda(\hat\omega^*)^*}
= \ip{\hat\lambda(y^*\hat\omega^*)^*}{\omega}
= \lim_\alpha \ip{\lambda(x_\alpha \omega)}{\hat\omega}. \]

Let $\omega_1,\omega_2\in L^1(\G)$.  Then, by the previous lemma,
\begin{align*}
\ip{L(\omega_1 \omega_2)}{\hat\lambda(\hat\omega^*)^*}
&= \lim_\alpha \ip{\lambda(x_\alpha \omega_1 \omega_2)}{\hat\omega}
= \lim_\alpha \ip{\lambda(x_\alpha \omega_1)}{\lambda(\omega_2)\hat\omega} \\
&= \ip{L(\omega_1)}{\hat\lambda(\hat\omega^* \lambda(\omega_2)^*)^*}
= \ip{L(\omega_1)}{\omega_2\cdot\hat\lambda(\hat\omega^*)^*}
= \ip{L(\omega_1)\omega_2}{\hat\lambda(\hat\omega^*)^*}.
\end{align*}
It follows that $L(\omega_1\omega_2)\in L^1(\G)$ as $L^1(\G)$ is an ideal
in $M(\G)$.  As $\{ \omega_1 \omega_2 : \omega_1,\omega_2\in L^1(\G) \}$
is dense in $L^1(\G)$ it follows by continuity that $L$ actually maps into
$L^1(\G)$.  Furthermore, the calculation shows that $L$ is a left-multiplier.

Similarly, we can construct a bounded linear map $R:L^1(\G)\rightarrow M(\G)$
which satisfies
\[ \ip{R(\omega)}{\hat\lambda(\hat\omega^*)^*} = 
\ip{\hat\lambda(\hat\omega^*y^*)^*}{\omega}
= \lim_\alpha \ip{\lambda(\omega x_\alpha)}{\hat\omega}. \]
We can then analogously show that $R$ maps into $L^1(\G)$ and is a right-multiplier.
Then, for $\omega_1,\omega_2\in L^1(\G)$, we have that
\begin{align*} \ip{\omega_1 L(\omega_2)}{\hat\lambda(\hat\omega^*)^*}
&= \ip{L(\omega_2)}{\hat\lambda(\hat\omega^*)^* \cdot \omega_1}
= \ip{L(\omega_2)}{\hat\lambda( \lambda(\omega_1)^*\hat\omega^*)^*} \\
&= \lim_\alpha \ip{\lambda(x_\alpha \omega_2)}{\hat\omega\lambda(\omega_1)}
= \lim_\alpha \ip{\lambda(\omega_1 x_\alpha \omega_2)}{\hat\omega} \\
&= \lim_\alpha \ip{\lambda(\omega_1 x_\alpha)}{\lambda(\omega_2)\hat\omega}
= \ip{R(\omega_1)}{\hat\lambda(\omega^* \lambda(\omega_2)^*)^*} \\
&= \ip{R(\omega_1)}{\omega_2 \cdot \hat\lambda(\omega^*)^*}
= \ip{R(\omega_1)\omega_2}{\hat\lambda(\omega^*)^*}.
\end{align*}
It follows that $(L,R)\in M(L^1(\G))$.

We now observe that $L(\omega)$, in $M(\G)$, is the weak$^*$-limit
of $(x_\alpha\omega)$.  As $(x_\alpha)$ is a bounded net in $M_{cb}(L^1(\G))$,
it follows that $\|L\|_{cb} \leq \sup_\alpha \|x_\alpha\|_{cb}\leq 1$.  The
same is true for $R$, so that $x=(L,R)\in M_{cb}(L^1(\G))$ with $\|x\|\leq 1$.
Then we have that
\begin{align*} \ip{\Lambda(x) \lambda(\omega)}{\hat\omega}
&= \ip{\lambda(L(\omega))}{\hat\omega}
= \ip{L(\omega)}{\hat\lambda(\hat\omega^*)^*}
= \ip{\hat\lambda(y^*\hat\omega^*)^*}{\omega} \\
&= \ip{\lambda(\omega)}{\hat\omega y}
= \ip{y\lambda(\omega)}{\hat\omega}. \end{align*}
So $\Lambda(x)\lambda(\omega) = y\lambda(\omega)$, and similarly,
$\lambda(\omega)\Lambda(x) = \lambda(\omega)y$.  By continuity, it follows
that $y\in M(C_0(\hat\G))$, and hence also that $y=\Lambda(x)$ as required.
\end{proof}

We now prove a general, abstract result.  This is probably folklore, but we do
not know of a reference.  We state this here in the operator space setting, but it
has an obvious Banach space counterpart.

\begin{proposition}
Let $A$ and $E$ be operator spaces, and let $\theta:A\rightarrow E^*$ be an injective
complete contraction such that the image of the closed unit ball
of $A$, under $\theta$, is weak$^*$-closed in $E^*$.
Suppose further that $\theta(A)$ is weak$^*$ dense in $E^*$.

Let $Q$ be the closure of $\theta^*\kappa_E(E)$ in $A^*$, so that $Q$ is an
operator space.  Then $Q^*$ is canonically completely isometrically isomorphic
to $A$, so $Q$ is a predual for $A$.
With respect to this predual, for a bounded net $(a_\alpha)$ in $A$, and $a\in A$,
we have that $a_\alpha\rightarrow a$ in $Q^*$ if and only if
$\theta(a_\alpha)\rightarrow\theta(a)$ in $E^*$.

Furthermore, if $A$ is a CCBA, $E^*$ is a dual CCBA, and $\theta$ is an algebra
homomorphism, then $A$ is a dual CCBA with respect to the predual $Q$.
\end{proposition}
\begin{proof}
As the image of $\theta$ is weak$^*$-dense, it follows that $j=\theta^*\kappa_E:
E\rightarrow A^*$ is an injection.  Then $Q$ is the closure of the image of $j$,
and we identify $Q^*$ with $A^{**}/Q^\perp$.  Let $q:A^{**}\rightarrow A^{**}/Q^\perp$
be the quotient map, so we wish to prove that $\iota = q\kappa_A:A\rightarrow Q^*$
is a completely isometric isomorphism.  Suppose that $\iota(a)=0$ for some $a\in A$.
Then, for $x\in E$, $0 = \ip{\iota(a)}{j(x)} = \ip{j(x)}{a} = \ip{\theta(a)}{x}$,
and so $\theta(a)=0$, so $a=0$.  Hence $\iota$ is injective.

For any $n\in\mathbb N$, let $(a_\alpha)$ be a bounded net in $\mathbb M_n(A)$
such that $(\theta(a_\alpha))$ converges weak$^*$ in $\mathbb M_n(E^*)$.  As
$\mathbb M_n(E^*)$ is, as a Banach space, isomorphic to $\mathbb M_n(E)^*$, we
see that each matrix component of $(\theta(a_\alpha))$ converges weak$^*$ in $E^*$.
By assumption, it follows that $(\theta(a_\alpha))$ converges weak$^*$ to something
in $\theta(\mathbb M_n(A))$; although note that we have lost norm control.

Let $\mu\in\mathbb M_n(Q^*)$, so there exists $\Phi\in\mathbb M_n(A^{**})$
with $q(\Phi)=\mu$ and $\|\Phi\|=\|\mu\|$.  As $\mathbb M_n(A)^{**} =
\mathbb M_n(A^{**})$, it follows that we can find a bounded net $(a_\alpha)$
in $\mathbb M_n(A)$ converging weak$^*$ to $\Phi$.  For $x\in\mathbb M_m(E)$,
it follows that
\[ \ipp{\mu}{j(x)} = \lim_\alpha \ipp{j(x)}{a_\alpha}
= \lim_\alpha \ipp{\theta(a_\alpha)}{x}, \]
so that the net $(\theta(a_\alpha))$ is weak$^*$-convergent in $E^*$, say to
$\lambda\in E^*$.  Thus $\ipp{\lambda}{x} = \ipp{\mu}{j(x)}$.  By hypothesis,
there exists $a\in\mathbb M_n(A)$, with $\theta(a)=\lambda$.
Then $\ipp{\mu}{j(x)} = \ipp{\theta(a)}{x} = \ipp{\iota(a)}{j(x)}$.  As $x$
was arbitrary, and $j$ has dense range, it follows that $\iota(a)=\mu$.
Thus $\iota$ is surjective, and as $\iota$ is automatically a complete contraction,
we may conclude $\iota$ is a complete isometry, as required.

Henceforth, we can identify $A$ as the dual of $Q$.
Let $(a_\alpha)$ be a bounded net in $A$, and let $a\in A$.  If
$a_\alpha\rightarrow a$ weak$^*$, then by construction, $\theta(a_\alpha)
\rightarrow \theta(a)$ in $E^*$.  Conversely, if $\theta(a_\alpha) \rightarrow
\theta(a)$ in $E^*$, then as $a_\alpha$ is a bounded net, and $\iota(E)$ is
dense in the predual of $A$, it follows that $a_\alpha\rightarrow a$ weak$^*$ in $A$.

Suppose that $E^*$ is a dual CCBA,
that $A$ is a CCBA, and that $\theta$ a homomorphism.  Let
$(a_\alpha)$ be a net in $A$ converging weak$^*$ to $a\in A$.  For $b\in A$ and $x\in E$,
\begin{align*} \lim_\alpha \ip{a_\alpha b}{\iota(x)}
&= \lim_\alpha \ip{\theta(a_\alpha)\theta(b)}{x}
= \lim_\alpha \ip{\theta(a_\alpha)}{\theta(b)\cdot x}
= \lim_\alpha \ip{a_\alpha}{\iota(\theta(b)\cdot x)} \\
&= \ip{a}{\iota(\theta(b)\cdot x)} = \ip{ab}{\iota(x)}. \end{align*}
Thus $a_\alpha b\rightarrow ab$ weak$^*$ in $A$; similarly, $ba_\alpha\rightarrow
ba$.  So $A$ is a completely contractive dual Banach algebra with predual $Q$.

\end{proof}

Combining the previous two propositions, we can construct an operator space
predual $Q$ for $M_{cb}(L^1(\G))$, which turns $M_{cb}(L^1(\G))$ into a dual CCBA.

\begin{theorem}
Let $\G$ be a locally compact quantum group, and form $Q$ as above.
The weak$^*$-topology induced on $M_{cb}(L^1(\G))$ agrees with that given by
the operator space version of Theorem~\ref{lcqg_predual_mult_dba}.
\end{theorem}
\begin{proof}
The weak$^*$-topology on $M_{cb}(L^1(\G))$ constructed by
Theorem~\ref{lcqg_predual_mult_dba} is unique under the conditions that
$M_{cb}(L^1(\G))$ is a dual CCBA, and for a bounded net $(b_\alpha)$ is
$M(\G)$ and $b\in M(\G)$, we have that $b_\alpha \rightarrow b$ weak$^*$
if and only if $\phi(b_\alpha) \rightarrow \phi(b)$ weak$^*$ in $M_{cb}(L^1(\G))$.
Here $\phi:M(\G)\rightarrow M_{cb}(L^1(\G))$ is the canonical map.

The weak$^*$-topology induced by $Q$ satisfies that a bounded net $(a_\alpha)$
in $M_{cb}(L^1(\G))$ converges weak$^*$ to $a$ if and only if $\Lambda(a_\alpha)
\rightarrow \Lambda(a)$ in $L^\infty(\hat\G)$.

As in our discussion at the start of Section~\ref{lcqg}, we can identify $L^1(\G)$
with the closed linear span of elements of the form $x\varphi y^*$ where
$x,y\in\mf n_\varphi$.  As $\mf n_\varphi$ is a left ideal, and as $C_0(\G)$
has an approximate identity, it follows that $\{ \hat x \hat\omega:
\hat x\in C_0(\hat\G), \hat\omega \in L^1(\hat\G) \}$ is linearly dense in
$L^1(\hat\G)$.  Thus also $\{ \lambda(\omega) \hat\omega:
\omega\in L^1(\G), \hat\omega \in L^1(\hat\G) \}$ is linearly dense in $L^1(\hat\G)$.

So, let $(b_\alpha)$ be a bounded net in $M(\G)$.  If $b_\alpha\rightarrow b$ weak$^*$,
then for $\hat\omega\in L^1(\G)$,
\begin{align*}
\lim_\alpha \ip{\Lambda(\phi(b_\alpha))}{\lambda(\omega)\hat\omega} &=
\lim_\alpha \ip{\lambda(b_\alpha\omega)}{\hat\omega}
= \lim_\alpha \ip{b_\alpha\omega}{\hat\lambda(\hat\omega^*)^*}
= \ip{b\omega}{\hat\lambda(\hat\omega^*)^*} \\
&= \ip{\Lambda(\phi(b))}{\lambda(\omega)\hat\omega}.
\end{align*}
As $(b_\alpha)$, and hence also $(\Lambda(\phi(b_\alpha)))$, is a bounded net,
this is enough to show that $\Lambda(\phi(b_\alpha))\rightarrow\Lambda(\phi(b))$
weak$^*$ in $L^\infty(\hat\G)$.  So $\phi(b_\alpha)\rightarrow\phi(b)$ weak$^*$
with respect to the predual $Q$.
Conversely, if $\phi(b_\alpha)\rightarrow\phi(b)$ weak$^*$ with respect to the
predual $Q$, then, by the previous calculation, and Lemma~\ref{lemma::lambdahat},
\begin{align*} \lim_\alpha \ip{b_\alpha}{\hat\lambda(\hat\omega^*\lambda(\omega)^*)^*}
&= \lim_\alpha \ip{b_\alpha}{\omega\cdot\hat\lambda(\hat\omega^*)^*}
= \lim_\alpha \ip{\Lambda(\phi(b_\alpha))}{\lambda(\omega)\hat\omega} \\
&= \ip{\Lambda(\phi(b))}{\lambda(\omega)\hat\omega}
= \ip{b}{\omega\cdot\hat\lambda(\hat\omega^*)^*}
= \ip{b}{\hat\lambda(\hat\omega^*\lambda(\omega)^*)^*}.
\end{align*}
By density, this is again enough to show that $b_\alpha\rightarrow b$ weak$^*$ in
$M(\G)$.

So the weak$^*$-topology induced by $Q$ satisfies the uniqueness condition from
Theorem~\ref{lcqg_predual_mult_dba}, and hence the proof is complete.
\end{proof}

\section{Multiplier Hopf Convolution algebras}

In the final section of this paper, we turn now to the study of \emph{Hopf Convolution
algebras}, as defined by Effros and Ruan in \cite{ERhc}.  We have seen already the
notion of a Hopf von Neumann algebra $(M,\Delta)$.  As $\Delta$ is normal, we can turn
$M_*$ into a completely contractive Banach algebra.  We tend to think of $M$ and $M_*$
as being two facets of the same object, but this would mean that $M_*$ should also
carry a coproduct which is induced by the product on $M$.  In the commutative case,
we can certainly do this.  Indeed, let $G$ be a locally compact group, and define
\begin{align*} m_*:L^1(G)\rightarrow & M(L^1(G)\proten L^1(G)) = M(G\times G); \\
&\ip{m_*(a)}{f} = \int_G a(s) f(s,s) \ ds
\qquad (f\in C_0(G\times G), a\in L^1(G)). \end{align*}
Notice the natural appearance of a multiplier algebra here.  This idea was, to our
knowledge, first explored by Quigg in \cite{quigg}, who worked in the Banach algebra
setting.  He showed that if $m:M\otimes M\rightarrow M$ is the product of a von Neumann
algebra $M$, then $m$ drops to a map $m_*:M_* \rightarrow M_* \proten M_*$ if and only
if $M$ is the direct sum of matrix algebras of bounded dimension.  (It would appear
that using multiplier algebras gives us no further examples).

The situation appears no better in the operator space setting, at least if we
use the projective tensor product.  However, Effros and Ruan showed that if we instead
work with the \emph{extended Haagerup tensor product}, then we can always define
$m_*: M_* \rightarrow M_* \ehten M_*$.  However, this tensor product is rather large:
for example, the algebraic tensor product $M_* \otimes M_*$ need not be norm dense in it.
We shall show that, in many cases, it is possible instead to work with the multiplier
algebra of the usual Haagerup tensor product $M_* \hten M_*$.  We term such a structure
a \emph{multiplier Hopf convolution algebra}.  We then go on to study the basics of a
corepresentation theory for such algebras.  In \cite{vv}, Vaes and Van Daele study
C$^*$-algebraic objects termed ``Hopf C$^*$-algebras'': it is interesting to note that
multiplier algebras associated to Haagerup tensor products play a central role in the
theory.

\subsection{Haagerup tensor products}\label{haatenprod}

Let us quickly review the Haagerup, and related, tensor products.  Let $E$ and $F$ be
operator spaces.  The \emph{Haagerup tensor norm} on $\mathbb M_n(E\otimes F)$ is defined by
\[ \| \tau \|^h_n = \inf\Big\{ \|u\| \|v\| : \tau_{ij} = \sum_{k=1}^m u_{ik} \otimes v_{kj},
u\in \mathbb M_{nm}(E), v\in\mathbb M_{mn}(F) \Big\}, \]
where $\tau\in\mathbb M_n(E\otimes F)$.
The completion is denoted by $E \hten F$.  This tensor product is both injective and
projective, and is self-dual, but it is not commutative.  For more details, see
\cite[Chapter~9]{ER} or \cite[Chapter~5]{pisier}

The \emph{extended Haagerup tensor product} is
\[ E \ehten F = (E^*\hten F^*)^*_\sigma, \]
the separately weak$^*$-continuous functionals on $E^* \hten F^*$.  Notice then that
$E \hten F$ embeds completely isometrically into $E\ehten F$.

If $\phi_i :E_i\rightarrow F_i$ are complete contractions between operator spaces,
then we have a complete contraction
\[ \phi_1\otimes\phi_2:E_1\hten E_2 \rightarrow F_1\hten F_2. \]
If the $\phi_i$ are complete isometries, then so is $\phi_1\otimes\phi_2$.
The same properties hold for the extended Haagerup tensor product.
If the $\phi_i$ are complete quotient maps, then $\phi_1\otimes\phi_2$, mapping
from between the Haagerup tensor products, is a complete quotient map.  This
property is not true for the extended Haagerup tensor product.

Recall that the map $E\proten F\rightarrow\mc{CB}(E^*,F)$
need not be injective; denote the resulting quotient of $E\proten F$ by
$E \nucten F$.  The \emph{shuffle map} is
\[ E_1 \otimes E_2 \otimes F_1 \otimes F_2 \rightarrow
E_1 \otimes F_1 \otimes E_1 \otimes F_2; \quad
a\otimes b \otimes c \otimes d \mapsto
a \otimes c \otimes b \otimes d. \]
This extends to a complete contraction
\[ S_e:(E_1\ehten E_2)\nucten (F_1\ehten F_2) \rightarrow
(E_1\nucten F_1) \ehten (E_2\nucten F_2). \]

Finally, let $M$ and $N$ be von Neumann algebras.  Then $M_* \nucten N_*
= M_* \proten N_*$ (this follows by duality and \cite[Theorem~7.2.4]{ER}).
Suppose that $M_*$ and $N_*$ are completely contractive Banach algebras
for products $\Delta^M_*$ and $\Delta^N_*$.
Then $M_* \ehten N_*$ is a completely contractive Banach algebra: the
product is just the composition of the maps
\begin{align*} (M_* \ehten N_*)\proten (M_*\ehten N_*) &\rightarrow
(M_* \ehten N_*)\nucten (M_*\ehten N_*) \\ &\rightarrow
(M_* \proten M_*) \ehten (N_*\proten N_*)
\rightarrow M_* \ehten N_*. \end{align*}
Furthermore, the multiplication map $m:M\otimes M\rightarrow M$ induces
a complete contraction $m_*:M_*\rightarrow M_*\ehten M_*$ which is a
homomorphism between the algebras $M_*$ and $M_*\ehten M_*$.
For further details, see \cite{ERhc}.

One might wonder how this relates to the classical case when $M=L^\infty(G)$.
It is shown in, for example \cite{sin}, that $M \hten M = M \otimes^\gamma M$,
where $\otimes^\gamma$ denotes the Banach space projective tensor norm (to
avoid confusion), with equivalent norms.
As the Haagerup tensor product is self-dual, it is not hard (compare
with the proof of Lemma~\ref{ex_two_fact} below) to show that
\[ L^1(G) \ehten L^1(G) = \{ T\in\mc B(L^\infty(G),L^1(G)) : T^*(L^\infty(G))
\subseteq L^1(G) \}, \]
again, with equivalent norms.  Then, for $a\in L^1(G)$, $m_*(a)\in L^1(G) \ehten
L^1(G)$ is identified with the map $L^\infty(G) \rightarrow L^1(G); f\mapsto fa$
where $fa$ is the point-wise product.

If we now do the same argument again with $C_0(G)$, we find that
$M(G) \ehten M(G) = \mc B(C_0(G),M(G))$.  Then we should have an inclusion
$L^1(G) \ehten L^1(G) \rightarrow M(G) \ehten M(G)$.  Under our identifications,
this is the map sending $T\in\mc B(L^\infty(G),L^1(G))$ to $\hat T\in\mc B(C_0(G),M(G))$,
where $\hat T$ is the composition
\[ \xymatrix{ C_0(G) \ar[r] & L^\infty(G) \ar[r]^T &
L^1(G) \ar[r] & M(G). } \]

At the beginning of this section, we considered the map $m_*':L^1(G) \rightarrow
M(G\times G)$, where, for $a\in L^1(G)$,
\[ \ip{m_*'(a)}{F} = \int_G F(s,s) a(s) \ ds
\qquad (F\in C_0(G\times G)). \]
We can use this to define $T:C_0(G)\rightarrow M(G)$ by
\[ \ip{T(f)}{g} = \ip{\mu}{f\otimes g} \qquad (f,g\in C_0(G)). \]
Then, for $f,g\in C_0(G)$,
\[ \ip{T(f)}{g} = \int_G f(s) g(s) a(s) \ ds
= \ip{m_*(a)(f)}{g}. \]
So $m_*$ is ``the same map'' as $m_*'$, but regarded as a map into
$L^1(G)\ehten L^1(G) \subseteq M(G)\ehten M(G)$ instead of $M(G\times G)$.

\subsection{Multiplier algebras}

Again, let $M$ and $N$ be von Neumann algebras, and suppose that $M_*$
and $N_*$ are completely contractive Banach algebras.  Then $M_*\hten N_*$
is a completely contractive Banach algebra.  Indeed, we have a complete
contraction $M_* \hten N_* \rightarrow M_* \ehten N_*$ and so
we have a map
\[ (M_* \hten N_*) \proten (M_* \hten N_*) \rightarrow M_* \ehten N_* \]
induced by the product on $M_*\ehten N_*$.  However, $M_*\hten N_*$ is
a closed subspace of $M_*\ehten N_*$, and so, by density, the product map
takes $(M_* \hten N_*) \proten (M_* \hten N_*)$ into $M_* \hten N_*$.
Hence $M_* \hten N_*$ is a CCBA.
In this section, we shall investigate if $M_* \ehten M_*$ may be replaced
by $M(M_* \hten M_*)$ or $M_{cb}(M_* \hten M_*)$.

A useful fact about the Haagerup tensor product is that it is \emph{self-dual}:
in particular, \cite[Theorem~9.4.7]{ER} shows that, for any operator spaces
$E$ and $F$, the natural map $E^*\otimes F^*\rightarrow (E\hten F)^*$ extends
to a complete isometry $E^*\hten F^*\rightarrow (E\hten F)^*$.  A useful
consequence of this is the following.  By \cite[Theorem~9.2.1]{ER} the
identity on $E\otimes F$ extends to a complete contraction $E\hten F
\rightarrow \mc{CB}(E^*,F)$.  Let $\tau\in E\hten F$, and let $T\in
\mc{CB}(E^*,F)$ be the induced map, so that
\[ \ip{\lambda}{T(\mu)} = \ip{\mu\otimes\lambda}{\tau}
\qquad (\mu\in E^*, \lambda\in F^*). \]
If $T=0$, then by linearity and continuity, $\tau$ annihilates all of
$E^*\hten F^*$, and so as the map $E^{**} \hten F^{**} \rightarrow
(E^*\hten F^*)$ is also a complete isometry, it follows that $(\kappa_E\otimes
\kappa_F)(\tau)=0$.  Hence $\tau=0$, so that the map $E\hten F \rightarrow
\mc{CB}(E^*,F)$ is injective.

\begin{lemma}\label{lemma::htenfaith}
Let $\G$ be a locally compact quantum group.  Then $L^1(\G) \hten L^1(\G)$ is faithful.
\end{lemma}
\begin{proof}
We first note that, as $L^1(\G)$ is faithful, then $\{x\cdot\omega:x\in L^\infty(\G),
\omega\in L^1(\G)\}$ is linearly weak$^*$-dense in $L^\infty(\G)$.  Indeed, if
$\omega_0\in L^1(\G)$ annihilates this set, then
\[ 0 = \ip{x\cdot\omega}{\omega_0} = \ip{x}{\omega\omega_0}
\qquad (x\in L^\infty(\G), \omega\in L^1(\G)). \]
Thus $\omega\omega_0=$ for all $\omega\in L^1(\G)$, so $\omega_0=0$, as required.

Let $\tau\in L^1(\G) \hten L^1(\G)$ be such that $\sigma\tau=0$ for all
$\sigma\in L^1(\G) \hten L^1(\G)$.  Let $\tau$ induce a map
$T:L^\infty(\G)\rightarrow L^1(\G)$ as above.  Then, for all
$x,y\in L^\infty(\G)$ and $a,b\in L^1(\G))$,
\[ 0 = \ip{x\otimes y}{(a\otimes b)\tau} =
\ip{x\cdot a \otimes y\cdot b}{\tau} = \ip{y\cdot b}{T(x\cdot a)}. \]
By weak$^*$-continuity, $T=0$, so $\tau=0$, as required.
\end{proof}

From now on, let $\mathbb G=(M,\Delta)$ be a locally compact quantum group.
Let $H$ be a Hilbert space arising from the GNS construction applied to the
left Haar weight, and let $W\in\mc B(H\otimes H)$ be the multiplicative unitary.
Let $\sigma:H\otimes H\rightarrow H\otimes H$ be the swap map, $\sigma(\xi\otimes\eta)
=\eta\otimes\xi$.  We say that $W$ is \emph{regular} if $\{ (\omega\otimes\iota)(W\sigma)
: \omega\in\mc B(H)_*\}$ is dense in $\mc K(H)$, the compact operators on $H$.
If $M=L^\infty(G)$ or $VN(G)$ (or, more generally, is a Kac algebra) then $W$ is regular,
but there do exists locally compact quantum groups for which $W$ is not regular,
see \cite{baaj}.  For further details, see \cite[Section~7.3]{timm}.

\begin{theorem}\label{reg_gives_mult}
Let $(M,\Delta)$ be a locally compact quantum group with regular multiplicative
unitary $W$.  For $a\in M_*$, we have that $m_*(a)$ is in the idealiser of
$M_*\hten M_*$ in $M_*\ehten M_*$.
\end{theorem}

We need to explain some more machinery before we give the proof.  For a Hilbert
space $K$, let $K_c$ be the \emph{column Hilbert space}, the operator space
induced by the isomorphism $K = \mc B(\mathbb C,K)$.  For operator spaces $E$ and
$F$, let $\tfac(E,F)$ be the space of completely bounded maps $E\rightarrow F$
which factor through a column Hilbert space $K_c$, equipped with the obvious norm.
Then $E\hten F = \tfac(F,E^*)$ for the duality given by
\[ \ip{T}{x\otimes y} = \ip{T(y)}{x} \qquad (T\in\tfac(F,E^*), x\in E,y\in F). \]
For further details, see \cite[Chapter~9]{ER}.

\begin{lemma}\label{ex_two_fact}
For operator spaces $E$ and $F$, we have
\[ E \ehten F = (E^*\hten F^*)^*_\sigma =
\{ T\in\tfac(F^*,E) : T^*(E^*)\subseteq F \}. \]
\end{lemma}
\begin{proof}
We have that $(E^*\hten F^*)^* = \tfac(F^*,E^{**})$, so we need
to show that $T\in \tfac(F^*,E^{**})$ induces a separately weak$^*$-continuous
functional if and only if $T$ maps into $E$, and $T^*(E^*)\subseteq F$.

Suppose that $T$ maps into $E$, and let $(\mu_\alpha)$ be a net in $E^*$
converging weak$^*$ to $\mu\in E^*$.  Then, for $\lambda\in F^*$,
\[ \lim_\alpha \ip{T}{\mu_\alpha\otimes\lambda} = \lim_\alpha \ip{\mu_\alpha}{T(\lambda)}
= \ip{\mu}{T(\lambda)} = \ip{T}{\mu\otimes\lambda}, \]
so $T$ is weak$^*$-continuous in the first variable.  Conversely, if
$T(\lambda)\in E^{**}\setminus E$ for some $\lambda\in F^*$, then we can find
a bounded net $(\mu_\alpha)$ in $E^*$ with $\ip{T(\lambda)}{\mu_\alpha}=1$
for each $\alpha$, and with $\mu_\alpha\rightarrow0$ weak$^*$.  However, $T$ is
weak$^*$-continuous, so
\[ 0 = \lim_\alpha \ip{T}{\mu_\alpha\otimes\lambda}
= \lim_\alpha \ip{T(\lambda)}{\mu_\alpha} = 1, \]
a contradiction.

Similarly, we can show that $T^*(E^*)\subseteq F$ if and only if the functional
induced by $T$ is weak$^*$-continuous in the second variable.
\end{proof}

So given $a,b,c\in M_*$, we have that $m_*(a)(b\otimes c) \in M_*\ehten M_*$.
We wish to show that this is really in $M_*\hten M_*$, but first let us identify
this with some $T \in \tfac(M,M_*)$ with $T^*(M)\subseteq M_*$.

\begin{lemma}
For $a,b,c\in M_*$ let $T\in\tfac(M,M_*)$ be induced by $m_*(a)(b\otimes c)
\in M_*\ehten M_*$.  Let $a=\omega_{\xi_0,\eta_0}$ for some $\xi_0,\eta_0\in H$.
Then $T$ factors through $H_c$ as
\[ \xymatrix{ M \ar[rr]^T \ar[rd]^\alpha & & M_* \\
& H_c \ar[ru]_\beta } \]
where $\alpha(x) = ((\iota\otimes c)\Delta(x))(\xi_0)$ and
$\beta(\xi) = \omega_{\xi,\eta_0} b$ for $x\in M$ and $\xi\in H_c$.
Furthermore, $\alpha$ and $\beta$ are completely bounded maps.
\end{lemma}
\begin{proof}
Let $c = \omega_{\xi_2,\eta_2}$ for some $\xi_2,\eta_2\in H$.
Define $A:H\rightarrow H\otimes H$ is given by $A(\eta)=\eta\otimes\eta_2$,
and define $B = \xi_0\otimes\xi_2 \in \mc B(\mathbb C,H\otimes H)$.  Then
for $x\in M$ and $\eta\in H$,
\[ (A^*W^*(1\otimes x)WB | \eta) =
(\Delta(x) \xi_0\otimes\xi_2 | \eta\otimes\eta_2)
= ((\iota\otimes c)\Delta(x) \xi_0 | \eta). \]
Thus $\alpha(x) = A^*W^*(1\otimes x)WB$, which shows that $\alpha$ is completely bounded.
A similar decomposition can be shown for $\beta$.

Then for $x,y\in M$, we have that
\begin{align*} \ip{y}{\beta\alpha(x)}
&= \ip{\Delta(y)}{\omega_{\alpha(x),\eta_0}\otimes b}
= ((b\cdot y) ( c\cdot x) \xi_0 | \eta_0 ) \\
&= \ip{m( (b\cdot y)(c\cdot x) )}{a}
= \ip{y\otimes x}{m_*(a)(b\otimes )}, \end{align*}
so that $T=\beta\alpha$ as required.
\end{proof}
 
\begin{proof}[Proof of Theorem~\ref{reg_gives_mult}]
Let $a,b,c\in M_*$ and form $\alpha$ and $\beta$ as in the lemma.
Let $(e_i)_{i\in I}$ be an orthonormal basis for $H$, so we can find vectors
$(\phi_i)$ with
\[ W(\xi_0\otimes\xi_2) = \sum_i e_i \otimes \phi_i. \]
Hence $\sum_i \|\phi_i\|^2 = \|\xi_0\|^2 \|\xi_2\|^2$.  Pick $\epsilon>0$
and choose a finite set $F\subseteq I$ with $\sum_{i\not\in F} \|\phi_i\|^2
< \epsilon^2$.

For each $i\in I$, let $R_i = (\omega_{\eta_2,e_i}\otimes\iota)(W\sigma)$ which is
a compact operator, as $W$ is regular.  Let $\alpha_i:M\rightarrow H_c$ be the map
$\alpha_i(x) = R_i^*x(\phi_i)$.  Clearly $\alpha_i$ is completely bounded, and as
$R_i^*$ is compact, we can approximate $\alpha_i$, in the completely bounded norm,
by a finite-rank operator.
The same is hence true of $\alpha_F = \sum_{i\in F} \alpha_i$.

Then, for $x\in M$ and $\eta\in H$,
\begin{align*} \big( (\alpha-\alpha_F) (x) \big| \eta \big) &=
\big( W^*(1\otimes x) W(\xi_0\otimes\xi_2) \big| \eta\otimes\eta_2 \big)
   - \sum_{i\in F} \big( \sigma W^* (e_i\otimes x\phi_i) \big| \eta_2 \otimes \eta \big) \\
&= \sum_{i\not\in F} \big( W^*(1\otimes x)(e_i\otimes\phi_i) \big| \eta\otimes\eta_2 \big)
= \big( T (1\otimes x) S \big| \eta \big),
\end{align*}
where $S = \sum_{i\not\in F} e_i\otimes x_i \in \mc B(\mathbb C, H\otimes H)$
and $T\in\mc B(H\otimes H, H)$ is defined by $(T\xi|\eta) = (W^*\xi|\eta\otimes\eta_2)$
for $\xi\in H\otimes H$.  It follows that $\|\alpha-\alpha_F\|_{cb} \leq \|S\| \|T\|
< \epsilon \|T\| = \epsilon \|\eta_2\|$.

As $\epsilon>0$ was arbitrary, it follows that $\alpha$ is in the cb-norm closure of
the finite-rank maps from $M$ to $H_c$.  Thus also $m_*(a)(b\otimes c) = \beta\alpha$ can be
cb-norm approximated by finite-rank maps.
As the inclusion $M_* \hten M_*\rightarrow M_*\ehten M_*$ is a complete isometry, we
conclude that $m_*(a)(b\otimes c) \in M_* \hten M_*$ as required.

To show that $(b\otimes c)m_*(a)\in M_*\hten M_*$, we use the unitary antipode.
A little care is needed, as $R$ is not completely bounded.  However, let
$r = (R\otimes R)\sigma:M\otimes M\rightarrow M\otimes M$.  Then we claim that $r$
extends to an isometry on $M\hten M$.  Indeed, for $\tau\in M\otimes M$, we have that
\[ \|\tau\|^h = \inf\Big\{ \Big\| \sum_i a_ia_i^* \Big\|^{1/2}
\Big\| \sum_i b_i^*b_i\Big\|^{1/2} : \tau=\sum_i a_i\otimes b_i \Big\}. \]
Then $\tau = \sum_i a_i\otimes b_i$ if and only if $r(\tau) = \sum_i R(b_i)\otimes R(a_i)$
and so
\begin{align*}
\|r(\tau)\|^h &= \inf\Big\{ \Big\| \sum_i R(b_i)R(b_i)^* \Big\|^{1/2}
\Big\| \sum_i R(a_i)^*R(a_i)\Big\|^{1/2} : \tau=\sum_i a_i\otimes b_i \Big\} \\
&= \inf\Big\{ \Big\| \sum_i R(b_i^*b_i) \Big\|^{1/2}
\Big\| \sum_i R(a_ia_i^*)\Big\|^{1/2} : \tau=\sum_i a_i\otimes b_i \Big\} = \|\tau\|^h,
\end{align*}
as $R$ is an isometry.

As $r$ is normal, and the Haagerup tensor product is self-dual, $r$ induces an
isometry $r_*:M_*\hten M_* \rightarrow M_*\hten M_*; a\otimes b\mapsto
R_*(b)\otimes R_*(a)$.  Similarly, as $r$ is separately weak$^*$-continuous,
$r_*$ extends to an isometry $M_* \ehten M_* \rightarrow M_*\ehten M_*$.  As $R_*$
is anti-multiplicative, the same is true of $r_*$.

For $a\in M_*$ and $x,y\in M$,
\[ \ip{x\otimes y}{rm_*(a)} = \ip{R(y)\otimes R(x)}{m_*(a)}
= \ip{R(xy)}{a} = \ip{x\otimes y}{m_*R_*(a)}, \]
so that $rm_* = m_*R_*$.

Thus, for $a,b,c\in M_*$, we see that
\[ (b\otimes c)m_*(a) = r\big( rm_*(a) (R_*(c)\otimes R_*(b)) \big)
= r\big( m_*(R_*(a)) (R_*(c)\otimes R_*(b)) \big) \in M_*\hten M_*, \]
as required.
\end{proof}

Given $a\in M_*$, we have that $m_*(a)$ idealises $M_*\hten M_*$ in
$M_* \ehten M_*$.  Thus there exist maps $L,R:M_*\hten M_* \rightarrow
M_*\hten M_*$ such that $L(\tau) = m_*(a) \tau$ and $R(\tau) = \tau m_*(a)$
for $\tau\in M_*\hten M_*$.  Thus $(L,R)\in M(M_*\hten M_*)$.  Indeed,
as $L$ and $R$ are induced by multiplication by a member of $M_*\ehten M_*$,
and this algebra is a CCBA, it follows immediately that
$(L,R)\in M_{cb}(M_*\hten M_*)$.  As $m_*:M_*\rightarrow M_*\ehten M_*$
is a complete contraction, it follows that actually we can regard $m_*$
as a completely contractive homomorphism $M_*\rightarrow M_{cb}(M_*\hten M_*)$.

It would, of course, be interesting to know if this result holds when $W$
is not regular.  Suppose that $M_*$ is unital, so that $\mathbb G$ is a discrete
quantum group (as $C_0(\hat\G)$ must also be unital, so $\hat\G$ is compact).
It is shown in \cite[Examples~7.3.4]{timm} that the multiplicative unitary $W$
associated to any algebraic quantum group is regular.  In particular, this implies
that the $W$ associated to a discrete quantum group is regular, and so our
theorem holds.  In particular, this means that $m_*$ can be regarded as
a map $M_* \rightarrow M_* \hten M_*$.

\subsection{Application to corepresentations}

Following the usual theory for C$^*$-bialgebras and Hilbert spaces, if $M_*$ is unital
(so that $W$ is regular), then we might define a corepresentation of $M_*$ on an operator
space $E$ to be a completely bounded map $\alpha:E\rightarrow E \hten M_*$ with
$(\alpha\otimes\iota)\alpha = (\iota\otimes m_*)\alpha$.  The use of the Haagerup
tensor product here is essentially forced upon us, as $m_*$ maps $M_*$ into
$M_*\hten M_*$.

If $M_*$ is not unital (and $W$ is assumed regular) then $m_*$ only maps into the
multiplier algebra $M_{cb}(M_*\hten M_*)$.  Consequently we need to consider the ideas of
Section~\ref{ten_prod_sec}.  Notice that the Haagerup tensor norm is uniformly
admissible in the sense of Section~\ref{ten_prod_sec}.  So a \emph{corepresentation} of
$M_*$ will consist of an operator space $E$ and a completely bounded linear map
$\alpha:E\rightarrow M_{cb}(E\hten M_*)$ such that the following diagram commutes
\[ \xymatrix{ E \ar[r]^\alpha \ar[d]_\alpha
& M_{cb}(E\hten M_*) \ar[d]^{I_E \otimes m_*} \\
M_{cb}(E\hten M_*) \ar[r]^{\alpha\otimes\iota}
& M_{cb}(E\hten M_*\hten M_*) } \]
Here $I_E \otimes m_*$ is the extension discussed in Section~\ref{ten_prod_sec}
(notice that the Haagerup tensor product is sufficiently well-behaved, in
particular, it is $E$-admissible for any $E$).
However, we must explain what $\alpha\otimes\iota$ is.

Indeed, we first need to check that $E\hten M_*$ is faithful as an $M_*$-bimodule.

\begin{lemma}
For an operator space $E$, and any locally compact quantum group $\G$,
we have that $E\hten M_*$ is a faithful $M_*$-bimodule.
\end{lemma}
\begin{proof}
Let $\tau\in E\hten M_*$ with $\tau\cdot\omega=0$ for each $\omega\in M_*$
(the case when $\omega\cdot\tau=0$ for all $\omega$ is similar).  As in the proof
of Lemma~\ref{lemma::htenfaith} above, let $\tau$ induce a weak$^*$-continuous
completely bounded map $T:E^*\rightarrow M_*$.  Then
\[ 0 = \ip{\mu\otimes x}{\tau\cdot\omega} = \ip{\mu\otimes\omega\cdot x}{\tau}
= \ip{\omega\cdot x}{T(\mu)} \qquad (\mu\in E^*, x\in M, \omega\in M_*). \]
Again, this shows that $T=0$, so that $\tau=0$ as required.
\end{proof}

We shall now assume that $M_*$ has a bounded approximate identity.  This occurs
when $\G$ is \emph{co-amenable}, see \cite[Theorem~3.1]{bt} or
\cite[Theorem~2]{nh} for example, in which
case we even have a contractive approximate identity, say $(e_\beta)\subseteq M_*$.
Future work would be to try to make these ideas work in more generality.  For example,
we have not been able to decide if $m_*:M_* \rightarrow M(M_*\hten M_*)$ is
inducing; if it were, then it might be possible to use ``self-induced methods''
in place of a bounded approximate identity.

It will be useful to keep track of which algebra we are considering multipliers
over, for which we use the self-explanatory notation $M^{cb}_{M_*}(E\hten M_*)$,
and so forth.
We now show how to define $\alpha\otimes\iota:M_{M_*}(E\hten M_*)
\rightarrow M_{M_*\hten M_*}(E\hten M_*\hten M_*)$.

\begin{lemma}
Let $\mc A$ be a CCBA, and let $E$ be an operator
space.  The map $\phi:M^{cb}_{\mc A}(E\hten \mc A) \hten \mc A \rightarrow
M^{cb}_{\mc A\hten \mc A}(E\hten \mc A\hten\mc A)$ defined on elementary tensors by
\[ \hat x\otimes a \mapsto (L,R); \quad L(b\otimes c)=\hat x\cdot b\otimes ac,
\ R(b\otimes c)=b\cdot\hat x\otimes ca, \]
is a complete contraction.
\end{lemma}
\begin{proof}
We claim that for any operator space $F$, the map $\psi:\mc{CB}(\mc A,F)\hten\mc A
\rightarrow \mc{CB}(\mc A\hten\mc A,F\hten\mc A)$ defined by
\[ \psi(T\otimes a)(b\otimes c) = T(b) \otimes ac \qquad (b\otimes c\in\mc A\hten\mc A) \]
is a complete contraction.
If so, then let $F=E\hten\mc A$.  Let $p_l,p_r:M^{cb}_{\mc A}(E\hten\mc A)\rightarrow
\mc{CB}(\mc A,E\hten\mc A)$ be given by $p_l(L,R)=L$ and $p_r(L,R)=R$.  Then, for
$\tau\in M^{cb}_{\mc A}(E\hten\mc A)\hten \mc A$, we define $\phi(\tau)=(L,R)$, where
$L = \psi\big((p_l\otimes I_{\mc A})\tau)$ and
$R = \psi\big((p_r\otimes I_{\mc A})\tau)$.  Thus $\phi$ is a complete contraction, as required.

So, we show that $\psi$ is a complete contraction.  Let $\tau\in\mathbb M_n(\mc{CB}(\mc A,F)
\otimes \mc A)$ with $\|\tau\|^h\leq1$, so by \cite[Proposition~9.2.6]{ER}, we can find
$T\in\mathbb M_{n,r}(\mc{CB}(\mc A,F))$ and $a\in\mathbb M_{r,n}(\mc A)$ with
\[ \tau_{ij} = \sum_{k=1}^r T_{ik}\otimes a_{kj} \qquad (1\leq i,j\leq n), \]
and such that $\|T\| \|a\|\leq 1$.
Let $u\in\mathbb M_m(\mc A\otimes\mc A)$, we similarly we can find $b\in\mathbb M_{m,s}(\mc A)$
and $c\in\mathbb M_{s,m}(\mc A)$ with
\[ u_{\alpha\beta} = \sum_{\gamma=1}^s b_{\alpha\gamma}\otimes c_{\gamma\beta}
\qquad (1\leq \alpha,\beta\leq m), \]
and such that $\|u\|^h=\|b\|\|c\|$.
Then $\psi(\tau)(u) \in \mathbb M_{nm}(F\hten\mc A)$ has matrix entries
\[ \Big( \sum_{k,\gamma} T_{ik}(b_{\alpha\gamma}) \otimes a_{kj}c_{\gamma\beta}
\Big)_{(i\alpha),(j\beta)}, \]
and so has norm at most
\begin{align*} \Big\| & \Big( T_{ik}(b_{\alpha\gamma}) \Big)_{(i\alpha,k\gamma)}
   \Big\|_{\mathbb M_{nm,rs}(F)}
\Big\| \Big( a_{kj}c_{\gamma\beta} \Big)_{(k\gamma,j\beta)} \Big\|_{\mathbb M_{rs,nm}(\mc A)}
\\ &\leq \|T\|_{\mathbb M_{n,r}(\mc{CB}(\mc A,F))} \|b\|_{\mathbb M_{m,s}(\mc A)}
   \|a\|_{\mathbb M_{r,n}(\mc A)} \|c\|_{\mathbb M_{s,n}(\mc A)}. \end{align*}
It follows that $\psi$ is a complete contraction.
\end{proof}

We now define $\alpha\otimes\iota:M_{M_*}^{cb}(E\hten M_*)
\rightarrow M_{M_*\hten M_*}^{cb}(E\hten M_*\hten M_*)$.
Given $\hat x\in M_{M_*}^{cb}(E\hten M_*)$, for each $\beta$ we have that
$\hat x\cdot e_\beta \in E\hten M_*$ and so $(\alpha\otimes I_{M_*})(\hat x\cdot e_\beta)
\in M_{M_*}^{cb}(E\hten M_*)\hten M_*$.  Thus $\phi (\alpha\otimes I_{M_*})(\hat x\cdot e_\beta)
\in M_{M_*\hten M_*}^{cb}(E\hten M_*\hten M_*)$, by the previous lemma.
So we may (try to) define $L,R:M_*\hten M_* \rightarrow E\hten M_*\hten M_*$ by
\[ L(a) = \lim_\beta \phi(\alpha\otimes I_{M_*})(\hat x\cdot e_\beta) \cdot a, \quad
R(a) = a \cdot \phi(\alpha\otimes I_{M_*})(e_\beta\cdot\hat x)
\qquad (a\in M_*\hten M_*). \]

\begin{proposition}
These limits exist, and we have that
$(L,R)\in M^{cb}_{M_*\hten M_*}(E\hten M_*\hten M_*)$.  The map
$\hat x\mapsto (L,R)$ is completely contractive.
\end{proposition}
\begin{proof}
Define $E^l_{a_1}: M_{M_*}^{cb}(E\hten M_*) \rightarrow E\hten M_*$ to be
the map $\hat y\mapsto \hat y\cdot a_1$, which is completely bounded.
Similarly define $E^r_{a_1}$.

Suppose that $a=a_1\otimes a_2 \in M_* \otimes M_*$.  Let $\epsilon>0$, and let
$\beta$ be such that $\|e_\beta a_2 - a_2\|<\epsilon$.  We can find $\tau = \sum_{i=1}^n
x_i\otimes c_i \in E \otimes M_*$ such that $\|\hat x\cdot e_\beta - \tau\|^h<\epsilon$.
Then, using the definition of $\phi$, we see that
\[ \big\| \phi(\alpha\otimes I_{M_*})(\hat x\cdot e_\beta) \cdot a
- \sum_i \alpha(x_i)\cdot a_1 \otimes c_i a_2 \big\|^h < \epsilon\|a\|. \]
However, as $\|\hat x\cdot e_\beta \cdot a_2 - \hat x \cdot a_2 \|^h
< \epsilon\|\hat x\|$, we also have that $\| \sum_i x_i \otimes c_ia_2 -
\hat x\cdot a_2\|^h < \epsilon(\|a_2\|+\|\hat x\|)$.  So
\[ \big\| \phi(\alpha\otimes I_{M_*})(\hat x\cdot e_\beta) \cdot a
- (E^l_{a_1}\alpha\otimes I_{M_*})(\hat x\cdot a_2) \big\|^h <
\epsilon(\|a_2\|+\|\hat x\|+\|a\|). \]
It follows that the net $(\phi(\alpha\otimes I_{M_*})(\hat x\cdot e_\beta) \cdot a)$
converges, and so $L(a)$ is well-defined.  Similar remarks apply to $R(a)$, as
\[ R(a) = (E^r_{a_1}\alpha\otimes I_{M_*})(a_2\cdot\hat x). \]
Indeed, we could \emph{define} $L$ and $R$ by this formula, but it's
not clear (to the author) that this formula defines a bounded map on
$M_*\hten M_*$.  However, clearly the limit does exist for all $a\in M_*\hten M_*$.

It is then easy to verify that $(b_1\otimes b_2)\cdot L(a_1\otimes a_2)
= R(b_1\otimes b_2)\cdot(a_1\otimes a_2)$, so that $(L,R)$ is a multiplier.
The net $(\phi(\alpha\otimes I_{M_*})(\hat x\cdot e_\beta))$ is bounded in
$M^{cb}_{M_*\hten M_*}(E\hten M_*\hten M_*)$, and so $L$ is completely bounded;
similarly $R$.  Indeed, $\|(L,R)\|_{cb} \leq \|\hat x\| \|\alpha\|_{cb}$.
Similarly, it now easily follows that the linear map $\hat x\mapsto (L,R)$ is
a complete contraction.
\end{proof}

We haven't really motivated the construction of $\alpha\otimes\iota$, so let
us do so now.
Suppose that $\hat x = x\otimes\hat a\in E\hten M(M_*) \subseteq M(E\hten M_*)$.
Then, for $a=b\otimes c\in M_*\hten M_*$, we have that
\[ L(a) = \lim_\beta (\alpha(x) \otimes \hat a e_\beta) \cdot a
= \alpha(x)\cdot b \otimes \hat a c, \quad
R(a) = b\cdot\alpha(x) \otimes c\hat a. \]
Thus $(L,R)$ can be identified with $\alpha(x)\otimes\hat a$, as we might hope.

Thus we have defined all of our maps, and so have (a proposal for) a notion
of a corepresentation of a Multiplier Hopf convolution algebra.

\subsection{Avoiding multipliers}

An alternative way to define a corepresentation would be to use the extended
Haagerup tensor product directly.  That is, a corepresentation of $M_*$ would
consist of an operator space $E$ and a completely bounded map $\alpha:E
\rightarrow E \ehten M_*$ such that $(\alpha\otimes I_{M_*}) \alpha =
(I_E\otimes m_*)\alpha$.  Here $\alpha\otimes I_{M_*}$ and $(I_E\otimes m_*)$
are defined without further work, and map into $E \ehten M_*\ehten M_*$.

We hence have two proposals for what a corepresentation of a (multiplier) Hopf
convolution algebra should be.  In particular, these apply to $A(G)$ for amenable
$G$.  It would be interesting to explore this theory further: recently Runde has
shown in \cite{runde3} that, essentially, one cannot move away from Hilbert spaces
when considering co-representations of Hopf von Neumann algebras.  Is the theory
for Hopf convolution algebras richer?  Alternatively, as the corepresentation
theory of $M_*$ should correspond to the representation theory of $M$, perhaps
we should only be interested in the case when $E$ is a Hilbert space (maybe
even with the column structure).  Is the theory easier in this case?

\section{Weak$^*$ topologies}

The following is surely known, but we have been unable to find a suitably
self-contained reference.  As we also wish to check that the result holds for
completely bounded maps, we include a proof here for convenience.

\begin{lemma}\label{weakstar_lemma}
Let $E$ and $F$ be Banach spaces, and let $T:E^*\rightarrow F^*$ be
a bounded linear map.  Then the following are equivalent:
\begin{enumerate}
\item\label{ws:one} $T$ is weak$^*$ continuous;
\item\label{ws:onea} for a bounded net $(\mu_\alpha)$ in $E^*$, if
$\mu_\alpha\rightarrow\mu\in E^*$ weak$^*$, then $T(\mu_\alpha)\rightarrow T(\mu)$
weak$^*$ in $F$.
\item\label{ws:two} $T^*\kappa_F(F) \subseteq \kappa_E(E)$;
\item\label{ws:three} there exists a bounded linear map $S:F\rightarrow E$ with $S^*=T$;
\end{enumerate}
If $T$ is an isomorphism, then $S$ in (\ref{ws:three}) is also an isomorphism,
and the properties above are also equivalent to:
\begin{enumerate}
\item[(5)] for a bounded net $(\mu_\alpha)$ in $E^*$ and $\mu\in E^*$,
we have that $\mu_\alpha\rightarrow\mu$ weak$^*$ if and only if
$T(\mu_\alpha)\rightarrow T(\mu)$ weak$^*$.
\end{enumerate}
The same holds for operator spaces and completely bounded maps.
\end{lemma}
\begin{proof}
Clearly (\ref{ws:one}) implies (\ref{ws:onea}).
If (\ref{ws:onea}) holds, then let $x\in F$ and let $M\in\kappa_E(E)^\perp
\subseteq E^{***}$.  Let $(\mu_\alpha)$ be a bounded net in $E^*$ tending
weak$^*$ to $M$ in $E^{***}$.  Thus $\mu_\alpha\rightarrow0$ weak$^*$ in
$E^*$ and so $T(\mu_\alpha)\rightarrow 0$ weak$^*$ in $F^*$.  Thus
$\ip{M}{T^*\kappa_F(x)} = \lim_\alpha \ip{T^*\kappa_F(x)}{\mu_\alpha}
= \lim_\alpha \ip{T(\mu_\alpha)}{x} = 0$.  This shows that $T^*\kappa_F(x)
\in \kappa_E(E)$, which implies that (\ref{ws:two}) holds.

If (\ref{ws:two}) holds then there exists $S:F\rightarrow E$ with
$\kappa_E S = T^*\kappa_F$.  As $\kappa_E$ and $\kappa_F$ are linear isometries,
it follows that $S$ is linear and bounded.  For $\mu\in E^*$ and $x\in F$,
we have that $\ip{S^*(\mu)}{x} = \ip{\mu}{S(x)} = \ip{T^*\kappa_F(x)}{\mu}
= \ip{T(\mu)}{x}$, showing that $S^*=T$.  So (\ref{ws:three}) holds.
Finally, (\ref{ws:three}) clearly implies (\ref{ws:one}).

If $T$ is an isomorphism, then so is $T^*$, and hence, as $\kappa_E S = T^*\kappa_F$,
it follows that $S$ is injective and bounded below.  If $\mu\in S(F)^\perp$ then
$S^*(\mu) = T(\mu)=0$ so $\mu=0$.  So $S$ is an isomorphism.  Then clearly (5)
holds, and obviously (5) implies (\ref{ws:one}).

If now $E$ and $F$ are operator spaces and $T$ is completely bounded, then
the only part to check is that (\ref{ws:two}) implies (\ref{ws:three}).
As $\kappa_E$ and $\kappa_F$ are complete isometries, and $\kappa_E S = T^*\kappa_F$,
it follows that $S$ is completely bounded, with $\|S\|_{cb} \leq \|T^*\|_{cb}
= \|T\|_{cb}$.  Finally, if $T$ is a complete isomorphism, then $S^{-1}$
exists and is bounded.  For $x\in\mathbb M_n(F)$, we have
\[ \| S(x)\| = \|\kappa_E S(x)\| = \|T^*\kappa_F(x)\|
\geq \|(T^*)^{-1}\|_{cb}^{-1} \|\kappa_F(x)\| = \|T^{-1}\|_{cb}^{-1} \|x\|. \]
It hence follows that $\|S^{-1}\|_{cb} \leq \|T^{-1}\|_{cb}$, as required.
\end{proof}

\bigskip
\noindent\textbf{Author's address:}
\parbox[t]{5in}{School of Mathematics,\\
University of Leeds,\\
Leeds LS2 9JT\\
United Kingdom}

\smallskip
\noindent\textbf{Email:} \texttt{matt.daws@cantab.net}

\end{document}